\documentclass[11p,reqno]{amsart}
\usepackage[english]{babel}\usepackage [latin1]{inputenc}
\usepackage{amsmath,amsthm}\usepackage [latin1]{inputenc}
\usepackage{amsfonts}\usepackage [latin1]{inputenc}
\usepackage{amssymb}\usepackage [latin1]{inputenc}
\usepackage{graphicx}\usepackage [latin1]{inputenc}
\usepackage{color}\usepackage [latin1]{inputenc}
\usepackage{amsfonts,mathtools,mathrsfs}\usepackage [latin1]{inputenc}
\usepackage{amscd}\usepackage [latin1]{inputenc}
\usepackage{bm}\usepackage [latin1]{inputenc}

\newtheorem{theorem}{Theorem}[section]
\newtheorem*{theorem*}{Theorem}
\newtheorem{lemma}{Lemma}[section]

\newtheorem{definition}[theorem]{Definition}

\newtheorem{remark}[theorem]{Remark}

\numberwithin{equation}{section}
\newtheorem{claim}[lemma]{Claim}

\begin{document}

\title{From the Landau-de Gennes theory to the Ericksen-Leslie theory in dimension two}

\author{Zhouping Xin}
\address{The Institute of Mathematical Sciences, The Chinese University of Hong Kong, Hong Kong, China}
\email{zpxin@ims.cuhk.edu.hk}
\author{Xiaotao Zhang*}
\address{South China Research Center for Applied Mathematics and Interdisciplinary Studies, South China Normal University, Guangzhou, China}
\email{xtzhang@m.scnu.edu.cn}
\thanks{* Corresponding author.}%    General info
\maketitle

\begin{abstract}
In this paper, we study the connection between the Ericksen-Leslie equations and the Beris-Edwards equations in dimension two. It is shown that the weak solutions to the Beris-Edwards equations converge to the one to the Ericksen-Leslie equations as the elastic coefficient tends to zero. Moreover, the limiting weak solutions to the Ericksen-Leslie equations may have singular points.
\end{abstract}


\section{introduction}
Liquid crystals are states of matter between conventional liquid and solid crystal, they may flow like a liquid, but their molecules may be oriented in a crystal-like way.
%According to the different molecular orientations, liquid crystalline phase can be classified as nematic phase, cholesteric phase and smectic phase etc. Besides, nematic liquid crystals are one of the most common liquid crystalline phases.
In physics, different order parameters are introduced to characterize the anisotropic behavior of liquid crystals, which lead to different theories. There are several competing mathematical theories for nematic liquid crystals in the literature, such as the Oseen-Frank theory \cite{O33,F58}, the Erickse-Leslie theory \cite{E61,L68}, the Landau-de Gennes theory \cite{D74}, and the Doi-Onsager theory \cite{DE86,O49}. The Oseen-Frank theory and the Erickse-Leslie theory are vector theories, in which the average direction of the liquid crystal molecules at a certain point is described by a unit vector. The Landau-de Gennes theory uses a $3\times 3$ symmetric traceless tensor $Q$ as the order parameter to describe the orientation of liquid crystal molecules. The Doi-Onsager theory is a molecular theory where the molecule has a continuous distribution of orientations. As these theories are derived from different considerations and are widely used in liquid crystal studies, it is important to explore the relationships of different theories.

%Majumdar-Zarnescu \cite{MZ10} considered the Oseen-Frank limit of the static Q-tensor model in dimension three, and the convergence is uniform away from the singularities of the Oseen-Frank minimizer. Later, Nguyen-Zarnescu \cite{NZ13} improved the result that the convergence is smooth except the singularity set of the Oseen-Frank minimizer. Wang-Wang-Zhang \cite{WWZ17} studied the harmonic map flow limit of the Q-tensor flow. Recently, Feng-Hong \cite{FH20} studied the  general Oseen-Frank limit of the static Q-tensor model with four elastic constants. For two dimension case, Bauman-Park-Phillips \cite{BPP12} and Golovaty-Montero \cite{GM14} and Canevari \cite{C2015} investigated static Q-tensor model with Dirichlet boundary conditions on the sides of nonzero degree, and its limiting  uniaxial nematic texture forms with a finite number of defects. Recently, Liu-Wang \cite{LW18} considered the Oseen-Frank limit of the Onsager's model.
The Ericksen-Leslie equations can be derived from the Doi-Onsager equations by taking small Deborah number limit in dimension three. This limit was formally derived in \cite{KD83,EZ06}, and rigorously justified before the first singular time of the Ericksen-Leslie equations in \cite{WZZ151}. It is noted that a new dynamic Q-tensor model was derived from the Doi's kinetic theory in \cite{HLWZZ15}, and the Ericksen-Leslie model can also be formally derived from this dynamic Q-tensor model in dimension three, which was later rigorously justified for smooth solutions in \cite{LWZ15}. Similarly, in dimension three, a rigorous derivation of the Ericksen-Leslie equations from the Beris-Edwards model \cite{BE94} in the Landau-de Gennes framework was given in \cite{WZZ15}, and similar result was established in dimension three recently in \cite{WL20} concerning the connection between the Ericksen-Leslie equations and the Qian-Sheng model in Landau-de Gennes framework. It should be noted that all these results have been established under the main assumption that the solutions to the Ericksen-Leslie equations are suitable smooth. If one takes no account of the velocity of the fluid, the Beris-Edwards model becomes the $Q$-tensor flows, and the Ericksen-Leslie system becomes the harmonic map flows. Moreover, in dimension three, \cite{WWZ17} has shown the connection between the solutions to the Q-tensor flows and the weak solutions to the harmonic map flows which may contain singular points. However, non-trivial singular weak solutions (with non-trivial velocity) to the Ericksen-Leslie equations do exist, see \cite{HLLW16} for three-dimensional case and \cite{LLWWZ19} for two-dimensional case. Our main goal in this paper is to study the connections between the solutions to the Beris-Edwards model and the weak solutions to the Ericksen-Leslie equations in dimension two. Note that the weak solutions to the Ericksen-Leslie equations may contain singular points.

\subsection{Notations} The following convertions will be used. $\Omega_T= (0,T)\times \mathbb{R}^2$ for $0<T<+\infty$, and $$||f||_{L^q_tL^p_x}^q=\int_0^T||f(t,\cdot)||_{L^p(\mathbb{R}^2)}^qdt, \,||f||_{L^p_tH^m_x}^p=\int_0^T ||f(t,\cdot)||_{H^m(\mathbb{R}^2)}^pdt,\,||f||_{L^p(\Omega_T)}=||f||_{L^p_tL^p_x},$$
for $p ,q\in[1,\infty]$. For any two vectors $m=(m_1, m_2, m_3), n=(n_1,n_2,n_3) \in \mathbb{R}^3$, $m\otimes n=[m_in_j]_{1\leq i,j\leq 3}$. $A\cdot B$ denotes the usual matrix/vector-matrix/vector product. Einstein summation is used throughout the paper. $A:B=A_{ij}B_{ij}$ and $|A|=\sqrt{A:A}$. The divergence of a tensor is defined by $(\nabla \cdot \sigma)_i=\partial_j\sigma_{ij}$, where $\partial_jf=\partial_{x_j}f$. $(\nabla A\odot \nabla A)_{ij}=\partial_i A : \partial_j A$ for matrix $A$ and $(\nabla d\odot \nabla d)_{ij}=\partial_i d \cdot \partial_j d$ for vector $d$. $\mathbb{I}$ denotes the $3\times 3$ identity matrix.  $\mathbb{S}^2=\{d\in \mathbb{R}^2, |d|=1\}$. For simplicity, the subsequences of $\{(v^\epsilon,Q^\epsilon)\}_{\epsilon>0}$ and $\{Q^\epsilon\}_{\epsilon>0}$ are still denoted as $\{(v^\epsilon,Q^\epsilon)\}_{\epsilon>0}$ and $\{Q^\epsilon\}_{\epsilon>0}$. Let $\mathcal{Q}_0\subset M^{3\times3}$ denote the space of $Q$-tensors i.e.
$$\mathcal{Q}_0=\{Q\in M^{3\times 3}, \text{tr}\, Q=0, Q_{ij}=Q_{ji}, i,j=1,2,3\}.$$
%In particular, we denote
%$$\mathcal{N}=\{ Q\in \mathcal{Q}_0, Q=s_+(d\otimes d -\frac13 \mathbb{I}), d\in \mathbb{S}^2\},\quad s_+=\frac{b+\sqrt{b^2+24ac}}{4c}.$$
Set
$$\mathcal{D}=C^\infty_0(\mathbb{R}^2, \mathbb{R}^3) \cap \{\phi=(\phi_1,\phi_2,\phi_3)^T :\partial_1\phi_1 + \partial_2 \phi_2 =0\},$$
$$\mathring{H}=\text{closure of } \mathcal{D} \text{ in } L^2 (\mathbb{R}^2,\mathbb{R}^3), \quad\mathring{J}=\text{closure of }\mathcal{D} \text{ in } H^1 (\mathbb{R}^2,\mathbb{R}^3),$$
and
$$\overline{\nabla F} =\left(
              \begin{array}{ccc}
                \partial_1 F_1 & \partial_2 F_1 & 0 \\
                \partial_1 F_2 & \partial_2 F_2 & 0 \\
                \partial_1 F_3 & \partial_2 F_3 & 0 \\
              \end{array}
            \right), \quad \underline{\nabla F} =\left(
                                                \begin{array}{cc}
                                                  \partial_1 F_1 & \partial_2 F_1 \\
                                                  \partial_1 F_2 & \partial_2 F_2 \\
                                                \end{array}
                                              \right)
             \text{ for } F:(0,T)\times \mathbb{R}^2 \mapsto \mathbb{R}^3.
$$
%We consider a special case of Beris-Edwards model, that is, the conditions (\ref{QCon1}) and (\ref{QCon2}) hold. Note that $\xi=0$ means the molecules are such that they only tumble in a shear flow, but are not aligned by such a flow \cite{PZ12}.
For any $Q\in \mathcal{Q}_0$, since $Q_{ii}=0$, it holds that
$$Q=s_1(\tilde{d}\otimes \tilde{d}-\frac13 \mathbb{I})+s_2(\hat{d}\otimes \hat{d}-\frac13 \mathbb{I}),$$
where $s_1,s_2\in \mathbb{R}$ and $\tilde{d}, \hat{d}\in \mathbb{S}^2$ are the eigenvectors of $Q$ satisfying $\tilde{d}\cdot\hat{d}=0$. When $s_1=s_2=0$, the nematic liquid crystal is said to be isotropic. When $s_1,s_2$ are different and nonzero, it is said to be biaxial. When $s_1=s_2\neq 0$ or $s_1=0,s_2\neq 0$ or $s_1\neq 0, s_2=0$, it is said to be uniaxial and $Q$ can be rewritten as
$$Q=s(\bar{d}\otimes \bar{d}-\frac13 \mathbb{I}), \quad \bar{d}\in \mathbb{S}^2.$$

\subsection{The Landau-de Gennes theory}
Let $Q\in \mathcal{Q}_0$. For the bulk energy density
$$F_b(Q)=-\frac{a}{2} |Q|^2 -\frac{b}{3} \text{tr} Q^3 +\frac{c}{4}|Q|^4,$$
one can verify that if $c>0$, then $F_b(Q)$ is bounded from below(\cite[Proposition 1]{M10}). Note that $a,b$ and $c$ are material-dependent and temperature-dependent constants. Furthermore, $F_b$ attains its minimum on the uniaxial $Q$-tensor with constant order parameter
$$s_+=\frac{b+\sqrt{b^2+24ac}}{4c}.$$ Thus, $F_b$ has a corresponding non-negative bulk energy density $\hat{F}_b$ defined by
\begin{equation}\label{hatF}
\hat{F}_b(Q)=F_b(Q)-\min_{Q\in \mathcal{Q}_0} F_b(Q),
\end{equation}
and
\begin{equation}\label{uniaxial}
\hat{F}_b(Q)=0 \Longleftrightarrow  Q\in \mathcal{N},
\end{equation}
where
$$\mathcal{N}=\left\{ Q\in \mathcal{Q}_0, Q=s_+(d\otimes d -\frac13 \mathbb{I}), d\in \mathbb{S}^2\right\}.$$
Then we define the following Landau-de Gennes energy functional $\mathcal{F}$ as
\begin{equation*}\label{L-D}
  \mathcal{F}(Q,\nabla Q)=\hat{\mathcal{F}}_b(Q)+\mathcal{F}_e(Q,\nabla Q),
\end{equation*}
where $\hat{\mathcal{F}}_b$ and $\mathcal{F}_e$ are the bulk energy and the elastic energy defined respectively by
\begin{equation*}\label{Fb}
  \hat{\mathcal{F}}_b(Q)=\int_{\mathbb{R}^n} \hat{F}_b(Q)dx ,
\end{equation*}
\begin{equation*}
  \mathcal{F}_e(Q,\nabla Q)=\frac12\int_{\mathbb{R}^n} \left( L_1 |\nabla Q|^2 +L_2 Q_{ij,j}Q_{ik,k}+L_3 Q_{ij,k}Q_{ik,j} +L_4 Q_{ij}Q_{kl,i}Q_{kl,j}\right)dx,
\end{equation*}
with $L_1,L_2,L_3$ and $L_4$ being material dependent elastic constants and $n=2,3$.

The Beris-Edwards model in $\mathbb{R}^3$ takes the form:
\begin{eqnarray}
  v_t+v\cdot \nabla v &=& -\nabla P+\nabla \cdot(\sigma^s+\sigma^a+\sigma^d), \label{Q1}\\
  \nabla \cdot v&=& 0,\label{Q2}\\
  Q_t +v\cdot \nabla Q&=& \frac{1}{\Gamma} H+R(\nabla v, Q),\label{Q3}
\end{eqnarray}
where $v:(0,T)\times\mathbb{R}^3 \mapsto \mathbb{R}^3$ is the velocity of the fluid, $Q:(0,T)\times\mathbb{R}^3 \mapsto \mathcal{Q}_0$ is the macroscopic $Q$-tensor order parameter, $P:(0,T)\times\mathbb{R}^3 \mapsto \mathbb{R}$ is the pressure, $\Gamma$ is a collective
rotational diffusion constant, $D=\frac12(\nabla v+(\nabla v)^T)$, $\Lambda=\frac12(\nabla v-(\nabla v)^T)$,
$\sigma^s, \sigma^a$ and $\sigma^d$ are the symmetric viscous stress, antisymmetric
viscous stress, and distortion stress, respectively, defined by
$$\sigma^s=\eta D-S_Q(H), \quad \sigma^a=Q\cdot H-H\cdot Q, \quad
 \sigma_{ij}^d=-\frac{\delta \mathcal{F}}{\delta Q_{kl,j}} Q_{kl,i},$$
 \begin{equation}\label{SQ-HD}
   S_Q(A)=\xi\left[ A\cdot(Q+\frac13\mathbb{I})+(Q+\frac13 \mathbb{I})\cdot A-2(Q+\frac13\mathbb{I})Q:A\right], \text{ for } Q, A\in M^{3\times 3},
 \end{equation}
where $\eta $ is the viscous coefficient, $H$ is the molecular field given by
$H(Q) =-\frac{\delta \mathcal{F}}{\delta Q},$ $\xi$ is a constant depending on the molecular details of a given liquid crystal and measures the ratio between the tumbling and the aligning effect that a shear flow would exert over the liquid crystals directors.
$R(\nabla v, Q)$ describes the rotating and stretching effects on the order parameter $Q$ due to the fluid, $R(\nabla v, Q)$ is defined by
$$R(\nabla v, Q)= S_Q(D)+\Lambda \cdot  Q -Q\cdot \Lambda.$$

There have been quite many works on the global existence of solutions to the system (\ref{Q1})-(\ref{Q3}), see \cite{PZ12,PZ11,CRWX16,LWX19,GR15,ADL14,W15} and the references therein. In particular, the first existence of global weak solutions to the cauchy problem for (\ref{Q1})-(\ref{Q3}) in the dimension two (2D) and dimension three (3D) is established by Paicu-Zarnescu in \cite{PZ12} under the conditions that $\xi=0$ and
\begin{equation}\label{QCon2}
  L_2=L_3=L_4=0,\quad c>0, \quad \eta>0, \quad \Gamma >0,
\end{equation}
where they also showed the existence of global regular solutions in 2D for suitably regular initial data, and the condition $\xi=0$ can be relaxed to $|\xi|$ being small in \cite{PZ11}. These results in  \cite{PZ12,PZ11} have been generalized to many interesting cases. In particular, for 2D periodic initial data, the global well-posedness of strong solutions to (\ref{Q1})-(\ref{Q3}) was obtained in \cite{CRWX16} under just condition (\ref{QCon2}), which was relaxed to allow $L_2, L_3$ and $L_4$ being non-zero with some other minor conditions recently in \cite{LWX19}; and for initial-boundary value problems for 2D and 3D, the global existence of weak solutions to the system (\ref{Q1})-(\ref{Q3}) has been proved in \cite{GR15,ADL14} under conditions that $\xi=0$ and (\ref{QCon2}) holds. Similar results have been obtained in \cite{W15}, where the bulk potential (\ref{Fb}) is replaced by Ball-Majumdar type bulk potential.

In this paper, we assume that condition (\ref{QCon2}) holds. Since the elastic constant $L_1$ is typically very small compared with $a,b$ and $c$, one can introduce a small parameter $\epsilon$, and consider the following Landau-de Gennes energy functional:
\begin{equation}\label{F-s}
\mathcal{F}_\epsilon(Q,\nabla Q)=\frac{1}{\epsilon}\int_{\mathbb{R}^2}\hat{F}_b(Q)dx+\frac{L_1}{2}\int_{\mathbb{R}^2} |\nabla Q|^2dx.
\end{equation}

Thus, we look for a solution $(v^\epsilon,Q^\epsilon)$ to (\ref{Q1})-(\ref{Q3}) which is independent of $x_3$. Then, $v^\epsilon=(v^\epsilon_1,v^\epsilon_2,v^\epsilon_3)^T: (0,T)\times\mathbb{R}^2\mapsto \mathbb{R}^3$ and $Q^\epsilon:(0,T)\times \mathbb{R}^2 \mapsto \mathcal{Q}_0$ satisfy
\begin{equation}\label{addxin}
 \partial_3 v^\epsilon =0,\quad \partial_3 Q ^\epsilon=0,\quad \partial_3 P^\epsilon=0.
\end{equation}
Then in this case, the system (\ref{Q1})-(\ref{Q3}) is reduced to the following two-dimensional one:
\begin{equation}\label{QQ}\left
\{\begin{array}{l}
\large{\partial_t v^\epsilon_i+v^\epsilon_j(\overline{\nabla v^\epsilon})_{ij} = -\partial_i P^\epsilon +\sum_{j=1}^2\partial_j[\eta \overline{D^\epsilon}-S_{Q^\epsilon}(H^\epsilon) +Q^\epsilon\cdot H^\epsilon-H^\epsilon \cdot Q^\epsilon]_{ij}}\\
\large{ \quad\quad\quad\quad\quad\quad\quad\,\,\,\,\,-L_1\sum_{k=1}^2\partial_k(\nabla Q^\epsilon \odot \nabla Q^\epsilon)_{ik},\quad i=1,2,}\\
\large{\partial_tv_3^\epsilon+v^\epsilon_j (\overline{\nabla v^\epsilon})_{3,j} = \sum_{j=1}^2\partial_j[\eta \overline{D^\epsilon}-S_{Q^\epsilon}(H^\epsilon) +Q^\epsilon\cdot H^\epsilon-H^\epsilon \cdot Q^\epsilon]_{3j},}\\
\large{\partial _1 v_1^\epsilon +\partial_2 v_2^\epsilon= 0,}\\
\large{Q_t^\epsilon +\underline{v^\epsilon}\cdot \nabla Q^\epsilon= \frac{1}{\Gamma} H^\epsilon+S_{Q^\epsilon}(\overline{D^\epsilon})+\overline{\Lambda^\epsilon} \cdot  Q^\epsilon -Q^\epsilon\cdot \overline{\Lambda^\epsilon},}
\end{array}
\right.
\end{equation}
where $P^\epsilon:(0,T)\times\mathbb{R}^2 \mapsto \mathbb{R}$, and
\begin{equation}\label{X-1}
  \overline{D^\epsilon}=\frac{\overline{\nabla v^\epsilon}+(\overline{\nabla v^\epsilon})^T}{2},\quad \overline{\Lambda^\epsilon}=\frac{\overline{\nabla v^\epsilon}-(\overline{\nabla v^\epsilon})^T}{2}, \quad \underline{v^\epsilon}=(v_1^\epsilon, v_2^\epsilon)^T,
\end{equation}
\begin{equation}\label{H}
  H^\epsilon=L_1\Delta Q^\epsilon -\frac{\mathcal{J}(Q^\epsilon)}{\epsilon},
\end{equation}
\begin{equation}\label{J}
   \mathcal{J}(Q)=\frac{\delta \hat{\mathcal{F}}_b(Q)}{\delta Q}=-aQ-b(Q^2-\frac13 |Q|^2\mathbb{I})+c|Q|^2Q,
\end{equation}
with $S_{Q^\epsilon}(H^\epsilon)$ and $S_{Q^\epsilon}(\overline{D^\epsilon})$ given in $(\ref{SQ-HD})$. Note that $\Delta Q^\epsilon$ in (\ref{H}) is equal to $\sum_{i=1}^2 \frac{\partial^2 Q^\epsilon}{\partial x_i^2}$ due to $Q^\epsilon:(0,T)\times \mathbb{R}^2 \mapsto \mathcal{Q}_0$.
\begin{remark}
Let $\mathbb{T}^1$ denote the periodic interval with period $A>0$ and $\Omega^\prime=\mathbb{R}^2 \times \mathbb{T}^1$. It can be checked that the solution $(v,Q)$ to the system $(\ref{Q1})$-$(\ref{Q3})$ is unique when $v,\nabla Q\in L^\infty(0,T;L^2(\Omega^\prime)) \cap L^2(0,T;H^1(\Omega^\prime))\cap L^\infty(0,T;W^{1,\infty}(\Omega^\prime))$. Then, the condition $(\ref{addxin})$ can be guaranteed by $\partial_3 v^\epsilon_0=0$ and $\partial_3 Q^\epsilon_0=0$ $(v^\epsilon|_{t=0} =v^\epsilon_0, Q^\epsilon|_{t=0} =Q_0^\epsilon)$ if there exists a global smooth solution $(v^\epsilon,Q^\epsilon)$ to the system $(\ref{QQ})$ with $v^\epsilon, \nabla Q^\epsilon \in L^\infty(0,T;L^2(\mathbb{R}^2)) \cap L^2(0,T;H^1(\mathbb{R}^2))\cap L^\infty(0,T;W^{1,\infty}(\mathbb{R}^2))$, which will be given in our forthcoming paper. Meanwhile, it should be noted that this two-dimensional system $(\ref{QQ})$ includes the two-dimensional system in \cite{PZ12,PZ11,CRWX16,LWX19,GR15,ADL14,W15}, in which $v:(0,T)\times \mathbb{R}^2\mapsto \mathbb{R}^2$, $Q: (0,T)\times \mathbb{R}^2\mapsto  \mathcal{Q}_1$ and $\mathcal{Q}_1=\{Q\in \mathbb{M}^{2\times 2}, \text{tr} \, Q=0, Q_{ij}=Q_{ji}, i,j=1,2\}$. Moreover, under conditions $(\ref{QCon2})$ and $|\xi|$ being sufficiently small, Paicu-Zarnescu \cite{PZ12} proved the existence of global regular solutions with sufficiently regular initial data for this two-dimensional problem.
\end{remark}

Corresponding to (\ref{QQ}), the initial data for $(v^\epsilon,Q^\epsilon)$ can be taken as
\begin{equation}\label{Q-IB}
v^\epsilon|_{t=0}=v_0^\epsilon \in \mathring{J},\quad Q^\epsilon|_{t=0}=Q_0^\epsilon, \quad Q_0^\epsilon-Q^\infty\in H^2(\mathbb{R}^2, \mathcal{Q}_0),
\end{equation}
where $Q^\infty=s_+(d^\infty\otimes d^\infty -\frac13 \mathbb{I})$ for the constant vector $d^\infty \in \mathbb{S}^2$. Then, the energy inequality of the Beris-Edwards system (\ref{QQ}) corresponding to the initial data (\ref{Q-IB}) is read as:
\begin{eqnarray}
   && \int_{\mathbb{R}^2} \left(\frac12 |v^\epsilon|^2 +\frac{L_1}{2} |\nabla Q^\epsilon|^2 +\frac{\hat{F}_b(Q^\epsilon)}{\epsilon}\right)(\cdot,t)dx +\int_0^t\int_{\mathbb{R}^2} \left(\eta|\overline{D^\epsilon}|^2 +\frac{1}{\Gamma}|H^\epsilon|^2\right)dxdt \nonumber\\
  &\leq& \int_{\mathbb{R}^2} \left(\frac12 |v^\epsilon_0|^2 +\frac{L_1}{2} |\nabla Q^\epsilon_0|^2 +\frac{\hat{F}_b(Q_0^\epsilon)}{\epsilon}\right)dx\label{Q-energy}
\end{eqnarray}
for $t\in(0,T)$, see \cite[Proposition 1]{PZ11} for the detailed derivation of (\ref{Q-energy}).
It is assumed further that
\begin{equation}\label{QSQ}
  v_0^\epsilon \rightarrow v_0^* \text{ in } L^2(\mathbb{R}^2,\mathbb{R}^3), \quad Q_0^\epsilon-Q^\infty \rightarrow Q_0^*-Q^\infty \text{ in } H^1(\mathbb{R}^2, \mathcal{Q}_0), \quad \int_{\mathbb{R}^2} \frac{\hat{F}_b(Q^\epsilon_0)}{\epsilon} \rightarrow 0,
\end{equation}
as $\epsilon \rightarrow 0$, where $Q^*_0=s_+(d_0^*\otimes d_0^*-\frac13 \mathbb{I})$. %Meanwhile, let $\lambda_{01}^\epsilon\leq \lambda_{02}^\epsilon\leq \lambda_{03}^\epsilon$ be the eigenvalues of $Q^\epsilon_0$, then $\int_{\mathbb{R}^2} \frac{\hat{F}_b(Q^\epsilon_0)}{\epsilon} \rightarrow 0$ (as $\epsilon \rightarrow0$) and (\ref{uniaxial}) imply that $(\lambda_{01}^\epsilon, \lambda_{02}^\epsilon, \lambda_{03}^\epsilon) \rightarrow (-\frac{s_+}{3}, -\frac{s_+}{3}, \frac{2s_+}{3})$ a.e. in $\mathbb{R}^2$ as $\epsilon \rightarrow0$. In this respect, we assume further that
%\begin{equation}\label{d33}
%\lambda_{02}^\epsilon< \lambda_{03}^\epsilon,
%\end{equation}
%for all $\epsilon$.
%and $d_{03}^\epsilon$ be the eigenvector of $Q_0^\epsilon$ corresponding to $\lambda_{03}^\epsilon$. \textcolor[rgb]{1.00,0.00,0.00}{It is assumed further that}
%\begin{equation}\label{d33}
 %d_{03}^\epsilon-d^\infty \rightarrow d_0^*-d^\infty \text{ in } H^1(\mathbb{R}^2), \text{ as } \epsilon\rightarrow 0.
%\end{equation}
Note that (\ref{QSQ}) implies that there exists $E_0>0$ such that for suitably small and positive $\epsilon$, it holds that
\begin{equation}\label{E01}
  \int_{\mathbb{R}^2} \left(\frac12 |v^\epsilon_0|^2 +\frac{L_1}{2} |\nabla Q^\epsilon_0|^2 +\frac{\hat{F}_b(Q_0^\epsilon)}{\epsilon}\right)dx \leq E_0.
\end{equation}

\subsection{The Ericksen-Leslie theory}
The general Ericksen-Leslie system in $\mathbb{R}^3$ takes the form
\begin{eqnarray}
  v_t +v\cdot \nabla v +\nabla P&=&  \nabla \cdot \sigma, \label{G-EL1}\\
  \nabla \cdot v &=& 0, \label{G-EL2}\\
  d\times (h-\gamma_1 N-\gamma_2D\cdot d) &=& 0 \label{G-EL3},
\end{eqnarray}
where $v:(0,T)\times\mathbb{R}^3 \mapsto \mathbb{R}^3$ is the velocity of the fluid, $P:(0,T)\times\mathbb{R}^3 \mapsto \mathbb{R}$ is the pressure, $d:(0,T)\times\mathbb{R}^3 \mapsto \mathbb{S}^2$ is the macroscopic orientation of the nematic liquid crystal molecules, and the stress $\sigma$ is modeled by the phenomenological constitutive relation
$$\sigma=\sigma^L+\sigma^E.$$
$\sigma^L$ is the viscous (Leslie) stress given by
\begin{equation}\label{L-stress}
  \sigma^L=\alpha_1(d\otimes d : D)d\otimes d+\alpha_2 N\otimes d +\alpha_3 d\otimes N   +\alpha_4 D+\alpha_5 D\cdot (d\otimes d) +\alpha_6  (d\otimes d)\cdot D
\end{equation}
with
$$N=d_t +v\cdot \nabla d -\Lambda \cdot d.$$
The six viscous coefficients $\alpha_1,\cdots , \alpha_6$ are called the Leslie coefficients. $\sigma^E$ is the elastic (Ericksen) stress
\begin{equation}\label{E-stress}
  \sigma_{ij}^E=-\frac{\partial E_{OF}}{\partial d_{k,j}}d_{k,i},
\end{equation}
where $E_{OF}=E_{OF}(d,\nabla d)$ is the Oseen-Frank energy density with the form
\begin{equation*}\label{OF}
  E_{OF}=\frac{k_1}{2}(\nabla \cdot d)^2 +\frac{k_2}{2}|d\cdot (\nabla \times d)|^2 +\frac{k_3}{2}|d\times (\nabla \times d )|^2
\end{equation*}
$$+\frac{k_2+k_4}{2}\left[\text{tr} (\nabla d)^2 -(\nabla \cdot d)^2\right].\qquad\qquad\quad$$
Here $k_1,k_2,k_3$ and $k_4$ are the elastic constants. The molecular field $h$ is given by
\begin{equation*}\label{h}
  h= -\frac{\delta \mathcal{E}_{OF}}{\delta d}, \quad \mathcal{E}_{OF}=\int_{\mathbb{R}^3} E_{OF} dx.
\end{equation*}

In order to obtain a basic energy law to the system (\ref{G-EL1})-(\ref{G-EL3}), one requires the Leslie coefficients, $\gamma_1$, and $\gamma_2$ to satisfy the following relations:
\begin{equation}\label{Les1}
  \alpha_2+\alpha_3=\alpha_6-\alpha_5,
\end{equation}
\begin{equation}\label{Les2}
  \gamma_1=\alpha_3-\alpha_2, \quad \gamma_2=\alpha_6-\alpha_5,
\end{equation}
where (\ref{Les1}) is called Parodi's relation.

For the system (\ref{G-EL1})-(\ref{G-EL3}), the local well-posedness in dimension three has been proved in \cite{WZZ12} under the physical constraints on the Leslie coefficients (\ref{Les1})-(\ref{Les2}), which ensure that the energy of the system is dissipated. When $v_3=0$, $\partial_3v=0$, $d_3=0$ and $\partial_3 d=0$, the system (\ref{G-EL1})-(\ref{G-EL3}) becomes a two-dimensional one.  For this case, the global weak solution has been shown in \cite{HLW12} under the conditions that {(\ref{Les1}) and (\ref{Les2}) hold and $k_1=k_2=k_3=1, k_4=0$ and
%\begin{equation}\label{E0}
 % k_1=k_2=k_3=1, k_4=0,
%\end{equation}
%and
\begin{equation}\label{Les3}
  \gamma_1>0,\quad \alpha_1+\frac{\gamma_2^2}{\gamma_1}\geq 0, \quad \alpha_4> 0,\quad \alpha_5+\alpha_6-\frac{\gamma_2^2}{\gamma_1}\geq 0.
\end{equation}
Similar results have been obtained in \cite{WW14} under weaker conditions that (\ref{Les1}), (\ref{Les2}), $v_3=0, \partial_3 v=0, \partial_3 d=0,  \partial_1 \sigma^L_{31}+\partial_2 \sigma_{32}^L=0$ and $\min\{k_1,k_2,k_3\} >0$ and
%$$\min\{k_1,k_2,k_3\} >0,
%$$
%and
\begin{equation}\label{Les4}
  \beta_2\geq 0, \beta_1+2\beta_2+\beta_3\geq 0, \beta_1<0 \text{ or } \beta_2\geq 0, 2\beta_2+\beta_3\geq0, \beta_1\geq0,
\end{equation}
where $\beta_1=\alpha_1+\frac{\gamma_2^2}{\gamma_1}, \beta_2=\alpha_4, \beta_3=\alpha_5+\alpha_6-\frac{\gamma_2^2}{\gamma_1}$.

For the simplified Ericksen-Leslie system:
\begin{equation}\label{A2}\left
\{\begin{array}{l}
\large{v_t-\nu\Delta v+v\cdot\nabla v +\nabla P=-\lambda\nabla \cdot(\nabla d \odot \nabla d),}\\
\large{\nabla \cdot v=0,}\\
\large{d_t+v\cdot \nabla d=\gamma(\Delta d+|\nabla d|^2 d).}\\
\end{array}
\right.
\end{equation}
where $v: (0,T)\times\mathbb{R}^n \mapsto \mathbb{R}^n$, $d:(0,T)\times\mathbb{R}^n \mapsto \mathbb{S}^2$, $n=2,3$ and $\nu,\lambda,\gamma$ are constants. This system was proposed first by Lin \cite{Lin89} in 1989. It has been shown in \cite{LLW10,H11} that global weak solutions to (\ref{A2}) in dimension two exist, which are smooth with possible exceptions of finitely many singular times. Similar results for more general case were obtained in \cite{HX12,HLW12,WW14}. For the uniqueness of this kind of weak solutions, we refer to \cite{LW10,LTX16}. For the three dimensional case,
global weak solutions with the initial data $d_0\in \mathbb{S}^2_+$ have been obtained in \cite{LW16}, which are weak limits of sequences of weak solutions to the Ginzburg-Landau approximate equations of (\ref{A2}) (see (\ref{A0}) below). Note that the weak solutions to (\ref{A2}) with smooth initial data may not be smooth. In fact, two examples of weak solutions of finite time singularity in dimension three have been constructed in  \cite{HLLW16}. Recently, weak solutions with finite time singularities in dimension two have been constructed in \cite{LLWWZ19}. For the blow-up criteria of strong solutions to the Ericksen-Leslie system, we refer to \cite{HW12,HLX14} and the references therein.

It should be noted that there are many studies on the following Ginzburg-Landau type approximation of the simplified Ericksen-Leslie system,
\begin{equation}\label{A0}\left
\{\begin{array}{l}
\large{v_t^\epsilon-\nu\Delta v^\epsilon+v^\epsilon\cdot\nabla v^\epsilon +\nabla P^\epsilon=-\lambda\nabla \cdot(\nabla d^\epsilon \odot \nabla d^\epsilon),}\\
\large{\nabla \cdot v^\epsilon=0,}\\
\large{d_t^\epsilon+v^\epsilon\cdot \nabla d^\epsilon=\gamma\left(\Delta d^\epsilon+ \frac{(1-|d^\epsilon|^2)d^\epsilon}{\epsilon}\right).}\\
\end{array}
\right.
\end{equation}
For (\ref{A0}) with fixed $\epsilon >0$ and the initial data $(v^\epsilon,d^\epsilon)|_{t=0} =(v_0,d_0)$, $v_0\in L^2(\mathbb{R}^n,\mathbb{R}^n), d_0\in \mathbb{S}^2, d_0-d^\infty\in H^1(\mathbb{R}^n,\mathbb{R}^3), n=2,3$, the global existence of weak solutions (even strong solutions for $n=2$) have been established by Lin-Liu \cite{LL95} (see also the extension to the case with Leslie stress \cite{LL00}). Such solutions satisfy the following energy inequality
$$\frac12\int_{\mathbb{R}^n} \left( |v^\epsilon|^2+ |\nabla d^\epsilon|^2 +\frac{(1-|d^\epsilon|^2)^2}{2\epsilon}\right) +\int_0^T\int_{\mathbb{R}^n} \left(|\nabla v^\epsilon|^2 +\left|\Delta d^\epsilon +\frac{(1-|d^\epsilon|^2)d^\epsilon}{\epsilon}\right |^2\right)\leq G_0$$
with $G_0=\frac12\int_{\mathbb{R}^n}( |v_0|^2 +|\nabla d_0|^2)dx$. This then implies that as $\epsilon\rightarrow 0^+$, $|d^\epsilon|\rightarrow 1$ as $\epsilon \rightarrow 0^+$ a.e., and $(v^\epsilon,d^\epsilon)$ is expected to converge to a solution of (\ref{A2}). Indeed, this convergence has been shown in \cite{H11,HLX14} on the time interval where the solution to (\ref{A2}) remains regular, and the methods in \cite{H11,HLX14} depend crucially on the regularity of the strong solutions to (\ref{A2}). However, the extension of this approach in \cite{H11,HLX14} to larger times seems impossible due to the existence of singular weak solutions to the Ericksen-Leslie equations \cite{HLLW16,LLWWZ19}. In this respect,  Kortum \cite{K2020} proved the convergence of the weak solutions to (\ref{A0}) to the global-in-time weak solutions to  (\ref{A2}) in two dimensional torus $\mathbb{T}^2$. The convergence of weak solutions to the Ginzburg-Landau approximation of the two dimensional simplified Ericksen-Leslie equations for both uniaxial and biaxial nematics has been obtained by Du-Huang-Wang \cite{DHW20}.
%And similar result was established by Du-Huang-Wang \cite{DHW20} for considering the Ginzburg-Landau approximation of the simplified Ericksen-Leslie equations for both uniaxial and biaxial nematics.

%Since the weak solution with singularity do exist \cite{HLLW16, LLWWZ19}, it is worth to study the weaker convergence
For a solution $(v^\epsilon,Q^\epsilon)$ to (\ref{QQ}) and (\ref{Q-IB}), due to (\ref{uniaxial}), the energy inequality (\ref{Q-energy}) and the condition (\ref{QSQ}), one may expect that $(v^\epsilon,Q^\epsilon)\rightarrow (v^*,Q^*)$ with $Q^*=s_+(d^*\otimes d^*-\frac13 \mathbb{I})$ as $\epsilon\rightarrow 0^+$, and $(v^*,d^*)$ is a solution to the Ericksen-Leslie system with the coefficients satisfying
\begin{equation}\label{MCon2}
  k_1=k_2=k_3=2L_1s_+^2,\quad k_4=0,\quad \gamma_1=2\Gamma s_+^2, \quad \gamma_2=-\frac{2\Gamma \xi s_+(s_+ +2)}{3},
\end{equation}
\begin{equation}\label{MCon3}
 \alpha_1=-\frac{2\Gamma \xi^2 s_+^2(3-2s_+)(1+2s_+)}{3},\quad \alpha_2 =-\Gamma s_+^2 -\frac{\Gamma \xi s_+(2+s_+)}{3},
\end{equation}
\begin{equation}\label{MCon4}
\alpha_3=\Gamma s_+^2 -\frac{\Gamma \xi s_+ (2+s_+)}{3}, \quad \alpha_4 =\eta +\frac{4\Gamma \xi^2(1-s_+)^2}{9},
\end{equation}
\begin{equation}\label{MCon5}
\alpha_5 = \frac{\Gamma \xi^2 s_+(4-s_+)}{3} +\frac{\Gamma \xi s_+(2+s_+)}{3}, \quad \alpha_6 =\frac{\Gamma\xi^2 s_+(4-s_+)}{3} -\frac{\Gamma \xi s_+(2+s_+)}{3}.
\end{equation}
Furthermore, the solution $(v^*,d^*)$ to the limiting Ericksen-Leslie system must satisfy
\begin{equation}\label{addxin3}
 \partial_3 v^*=0, \quad \partial_3 d^*=0 \text{ and } \partial_3 P^*=0
\end{equation}
due to $(v^\epsilon, Q^\epsilon)$ satisfying (\ref{addxin}). Thus, as (\ref{QQ}), $(v^*,d^*)$ solves the following two-dimensional system
\begin{equation}\label{dd}\left
\{\begin{array}{l}
\large{\partial_t v^*_i+v^*_j(\overline{\nabla v^*})_{ij} = -\partial_i P^* +\sum_{j=1}^2\partial_j (\overline{\sigma_*^L})_{ij} -k_1\sum_{l=1}^2\partial_l(\nabla d^* \odot \nabla d^*)_{il},\quad i=1,2,}\\
\large{\partial_tv_3^*+v^*_j (\overline{\nabla v^*})_{3,j} = \sum_{j=1}^2\partial_j(\overline{\sigma_*^L})_{3j},}\\
\large{\partial _1 v_1^* +\partial_2 v_2^*= 0,}\\
\large{k_1(\Delta d^*+|\nabla d^*|^2 d^*) -\gamma_1  \overline{N^*}- \gamma_2 [\overline{D^*}\cdot d^*-(\overline{D^*}:d^*\otimes d^*)d^*]=0,}
\end{array}
\right.
\end{equation}
where $v^*:(0,T)\times \mathbb{R}^2 \mapsto \mathbb{R}^3$, $d^*:(0,T)\times \mathbb{R}^2 \mapsto \mathbb{S}^2$, $P^*:(0,T)\times \mathbb{R}^2 \mapsto \mathbb{R}$,
\begin{equation}\label{X-2}
\overline{D^*}=\frac{\overline{\nabla v^*}+(\overline{\nabla v^*})^T}{2},\, \overline{\Lambda^*}=\frac{\overline{\nabla v^*}-(\overline{\nabla v^*})^T}{2}, \, \underline{v^*}=(v^*_1,v^*_2)^T,\,\overline{N^*}=d_t^*+\underline{v^*}\cdot\nabla d ^*-\overline{\Lambda^*}\cdot d^*,
\end{equation}
and
$$\overline{\sigma_*^L}=\alpha_1(d^*\otimes d^* : \overline{D^*})d^*\otimes d^*+\alpha_2 \overline{N^*}\otimes d^*+\alpha_3 d^*\otimes \overline{N^*} +\alpha_4 \overline{D^*}+\alpha_5 \overline{D^*}\cdot (d^*\otimes d^*) +\alpha_6  (d^*\otimes d^*)\cdot \overline{D^*}.$$
Note that $\Delta d^*$ in (\ref{dd}) is equal to $\sum_{i=1}^2\frac{\partial^2 d^*}{\partial x_i^2}$ due to $d^*:(0,T)\times\mathbb{R}^2\mapsto \mathbb{S}^2$, and the corresponding initial data for $(v^*,d^*)$ can be taken as
\begin{equation}\label{EL-IB}
 v^*|_{t=0} = v_0^*\in \mathring{H},\quad d^*|_{t=0}=d_0^*, \quad d_0^*-d^\infty\in H^1(\mathbb{R}^2, \mathbb{R}^3),
\end{equation}
where $v_0^*$ and $ d_0^*$ are given in (\ref{QSQ}). Then, the energy inequality for the limiting Ericksen-Leslie system (\ref{dd}) corresponding to the initial data (\ref{EL-IB}) is
\begin{eqnarray}
   &&  \int_{\mathbb{R}^2} \left( \frac12 |v^*|^2 +\frac{k_1}{2}|\nabla d^*|^2 \right)(\cdot, t)dx +\int_0^t \int_{\mathbb{R}^2} \left[\alpha_4|\overline{D^*}|^2+ (\alpha_1 +\frac{\gamma_2^2}{\gamma_1})|\overline{D^*}:(d^*\otimes d^*)|^2 \right]dxdt\nonumber\\
   &&  +\int_0^t \int_{\mathbb{R}^2} \left[(\alpha_5+\alpha_6-\frac{\gamma_2^2}{\gamma_1})|\overline{D^*}\cdot d^*|^2 +\frac{1}{\gamma_1} |d^*\times h^*|^2\right]dxdt\nonumber\\
   &\leq &  \int_{\mathbb{R}^2} \left( \frac12 |v_0^*|^2 +\frac{k_1}{2}|\nabla d_0^*|^2 \right)dx,\label{d-energy}
\end{eqnarray}
where $t\in (0,T)$, see \cite[Proposition 2.1]{WW14} for the detailed derivation of (\ref{d-energy}).
\begin{remark}\label{R1}
Under conditions $(\ref{MCon2})$-$(\ref{MCon5})$, the system $(\ref{G-EL1})$-$(\ref{G-EL3})$ can be regarded as the uniaxial limit of the Beris-Edwards system $(\ref{Q1})$-$(\ref{Q3})$ by sending $\epsilon\rightarrow 0$. In dimension three, this has been shown rigorously by Wang-Zhang-Zhang \cite{WZZ15} before the first singular time of the Ericksen-Leslie system $(\ref{G-EL1})$-$(\ref{G-EL3})$. Our main goal in this paper is to show that such an asymptotic convergence holds true for weak solutions to the Beris-Edwards system $(\ref{Q1})$-$(\ref{Q3})$ and the Ericksen-Leslie system $(\ref{G-EL1})$-$(\ref{G-EL3})$ in the 2-dimensional case specified by $(\ref{addxin})$.
\end{remark}

%we refer to \cite{LLW10,H11,LW16,HLLW16,LLWWZ19} and the references therein for the results of the existence of weak solutions and the existence of singular weak solutions.

%Lin \cite{Lin89} proposed a simplified Ericksen-Leslie system in 1989, Lin-Lin-Wang \cite{LLW10} and Hong \cite{H11} established the global weak solutions of this system in dimension two. Later, Lin-Wang \cite{LW16} established global weak solutions of the simplified Ericksen-Leslie system with the initial data $d_0\in \mathbb{S}^2$. In particular, Huang-Lin-Liu-Wang \cite{HLLW16} constructed two examples of weak solutions to the simplified Ericksen-Leslie system with finite time singularity. Recently, Lai-Lin-Wang-Wei-Zhou \cite{LLWWZ19} established the singular weak solutions in dimension two.

\subsection{Main results}
We give first the definitions of weak solutions to the system (\ref{QQ}) and the limiting weak solutions to the system (\ref{dd}).

Weak solutions to (\ref{QQ}) subject to the initial data (\ref{Q-IB}) can be defined as:
\begin{definition}\label{def1}
For $0<T<\infty$, a pair $(v^\epsilon,Q^\epsilon)$ is a weak solution to the system $(\ref{QQ})$ subject to the initial data $(\ref{Q-IB})$, if $v^\epsilon\in L^\infty(0,T;\mathring{H}) \cap L^2(0,T;\mathring{J})$ and $Q^\epsilon\in L^\infty(0,T; H^1(\mathbb{R}^2,\mathcal{Q}_0))   \cap L^2(0,T;H^2(\mathbb{R}^2,\mathcal{Q}_0))$ satisfy the energy inequality $(\ref{Q-energy})$ and
\begin{eqnarray}
    &&\int_{\Omega_T} [-v^\epsilon\cdot \psi_t -(v^\epsilon\cdot \overline{\nabla\psi})\cdot v^\epsilon+\eta \overline{D^\epsilon}:\overline{\nabla\psi} -L_1 \nabla Q^\epsilon \odot \nabla Q^\epsilon:\underline{\nabla \psi}] dxdt \nonumber \\
    && +\int_{\Omega_T}[ Q^\epsilon \cdot H^\epsilon-H^\epsilon \cdot Q^\epsilon-S_{Q^\epsilon}(H^\epsilon)]:\overline{\nabla \psi} dxdt= \int_{\mathbb{R}^2}v_0^\epsilon(x)\cdot \psi(0,x) dx,\label{Qw1}
\end{eqnarray}
\begin{eqnarray}
   && \int_{\Omega_T} [-Q^\epsilon:\varphi_t -(\underline{v^\epsilon}\cdot \nabla  \varphi): Q^\epsilon -\frac{1}{\Gamma}H^\epsilon:  \varphi  ]dxdt\nonumber\\
&+&\int_{\Omega_T} [Q^\epsilon\cdot \overline{\Lambda^\epsilon}-\overline{\Lambda^\epsilon}\cdot Q^\epsilon-S_{Q^\epsilon}(\overline{D^\epsilon})]:\varphi dxdt= \int_{\mathbb{R}^2} Q_0^\epsilon(x):\varphi(0,x) dx\label{Qw2}
\end{eqnarray}
for every $\psi \in C_0^\infty([0,T)\times\mathbb{R}^2,\mathbb{R}^3),\partial_1\psi_1 +\partial_2 \psi_2=0, \varphi\in C_0^\infty([0,T)\times\mathbb{R}^2, \mathcal{Q}_0)$, where $H^\epsilon$ is given in $(\ref{H})$ and $\overline{D^\epsilon},\, \overline{\Lambda^\epsilon}, \,\underline{v^\epsilon}$ are given in $(\ref{X-1})$, and $S_{Q^\epsilon}(H^\epsilon)$, $S_{Q^\epsilon}(\overline{D^\epsilon})$ are given in $(\ref{SQ-HD})$.
\end{definition}

While weak solutions to (\ref{dd}) subject to the initial data (\ref{EL-IB}) are defined as:

\begin{definition}\label{ELDe}
For $0<T<\infty$, a pair $(v^*,d^*)$ is a weak solution to the system $(\ref{dd})$ subject to the initial data $(\ref{EL-IB})$, if $v^*\in L^\infty(0,T;\mathring{H})\cap L^2(0,T;\mathring{J})$ and $d^*-d^\infty \in L^\infty(0,T;H^1(\mathbb{R}^2,\mathbb{R}^3))$ satisfy the energy inequality $(\ref{d-energy})$ and
\begin{eqnarray}
  &&\int_{\Omega_T}[-v^*\cdot \psi_t-(v^*\cdot\overline{\nabla \psi})\cdot v^*+\alpha_4 \overline{D^*}: \overline{\nabla \psi}-2L_1s_+^2 \nabla d^*\odot \nabla d^*:\underline{\nabla \psi} ] dxdt \nonumber\\
  &&+\int_{\Omega_T}[\alpha_1(d^*\otimes d^* : \overline{D^*})d^*\otimes d^* +\alpha_2 \overline{N^*}\otimes d^* +\alpha_3 d^*\otimes \overline{N^*} + \alpha_5 \overline{D^*}\cdot (d^*\otimes d^*)]:\overline{\nabla \psi}dxdt\nonumber\\
  &&+\int_{\Omega_T}\alpha_6  (d^*\otimes d^*)\cdot \overline{D^*} :\overline{\nabla \psi} dxdt =\int_{\Omega}v_0^*(x)\cdot \psi(0,x)dx, \label{ELw1}
\end{eqnarray}
\begin{eqnarray}
  &&\int_{\Omega_T}\{\gamma_1[-d^*\cdot\zeta_t-(\underline{v^*}\cdot \nabla \zeta)\cdot d^*-(\overline{\Lambda^*}\cdot d^*)\cdot \zeta]+ \gamma_2(\overline{D^*}\cdot d^* -d^*\otimes d^*:\overline{D^*} d^*)\cdot \zeta\} dxdt \nonumber \\
   &&+2L_1s_+^2\int_{\Omega_T}(\partial_k d^*\cdot\partial_k \zeta -|\nabla d^*|^2d^*\cdot \zeta )dxdt  =\gamma_1\int_{\Omega}d_0^*(x)\cdot \zeta(0,x) dx\label{ELw2}
\end{eqnarray}
for every $\psi \in C_0^\infty([0,T)\times\mathbb{R}^2,\mathbb{R}^3),\partial_1\psi_1 +\partial_2\psi_2=0, \zeta\in C^\infty_0([0,T)\times \mathbb{R}^2,\mathbb{R}^3)$, where $\overline{D^*},\,\overline{\Lambda^*}, \,\underline{v^*}, \,\overline{N^*}$ are given in $(\ref{X-2})$, and $\gamma_1,\gamma_2, \alpha_1,\cdots, \alpha_6$ are given in $(\ref{MCon2})$-$(\ref{MCon5})$.
\end{definition}

Then, the main results in this paper can be stated as follows.
\begin{theorem}\label{MT}
Assume that $\xi$ is suitably small, the conditions $(\ref{QCon2})$, $(\ref{addxin})$ and $(\ref{QSQ})$ hold, and the parameters $a,b$, and $c$ satisfy
\begin{equation}\label{abc}
  b>0, \quad b^2+27ac>0.
\end{equation}
Let $(v^\epsilon, Q^\epsilon)$ be weak solutions to the system $(\ref{QQ})$ subject to the initial data $(\ref{Q-IB})$, and $v^\epsilon\in L^\infty_tH^1_x\cap L^2_tH^2_x, Q^\epsilon-Q^\infty \in L^\infty_t H^2_x \cap L^2_tH^3_x$. Then, there exists a convergent subsequence of $\{(v^\epsilon, Q^\epsilon)\}_{\epsilon>0}$, such that
$$v^\epsilon \rightharpoonup v^* \text{ in } L^2(0, T; \mathring{J}), \quad v^\epsilon \mathop{\rightharpoonup}\limits^{\star} v^* \text{ in }L^\infty(0,T;\mathring{H} ), \quad \nabla Q^\epsilon \mathop{\rightharpoonup}\limits^{\star} \nabla Q^* \text{ in } L^\infty(0, T; L^2(\mathbb{R}^2,\mathcal{Q}_0)),$$
as $\epsilon \rightarrow 0^+$. Furthermore, $Q^*$ has the form
$$Q^*=s_+(d^*\otimes d^*-\frac13 \mathbb{I}),\quad d^*\in \mathbb{S}^2,\quad s_+=\frac{b+\sqrt{b^2+24ac}}{4c},$$
and $(v^*,d^*)$ is a weak solution to the Ericksen-Leslie system $(\ref{dd})$ subject to the initial data $(\ref{EL-IB})$ with the coefficients satisfying $(\ref{MCon2})$-$(\ref{MCon5})$.

\end{theorem}

\begin{remark}
In the proof of Theorem \ref{MT}, we need only the conditions $b>0$ and $b^2+24ac>0$. However, to make sure that $\min_{Q\in \mathcal{Q}_0}F_b(Q)$ is achieved at $Q=s_+(d\otimes d-\frac13 \mathbb{I}), d\in \mathbb{S}^2$, one needs the second condition in (\ref{abc}). Indeed, as shown in \cite{MN14,M10}, if $F_b(Q_m)=\min_{Q\in \mathcal{Q}_0}F_b(Q)$, then $Q_m$ must be uniaxial, i.e.  $Q_m=s(d\otimes d-\frac13 \mathbb{I}), d\in \mathbb{S}^2$. Hence, $F_b(Q_m)$ can be rewritten as
$$F_b(Q_m)=-\frac{a}{3}s^2-\frac{2b}{27}s^3+\frac{4c}{9}s^4:=f_b(s),$$
whose critical points are
$$s_0=0,\quad s_+=\frac{b+\sqrt{b^2+24ac}}{4c},\quad s_{-}=\frac{b-\sqrt{b^2+24ac}}{4c}.$$
Therefore, $\min_{Q\in \mathcal{Q}_0}F_b(Q)=\min\{ f_b(s_0),f_b(s_{+}),f_b(s_-)\}$. Note that $f_b(s_{\pm})=\frac{s_{\pm}^2}{54}(-9a -bs_{\pm})$ or $f_b(s_{\pm})=\frac{s_{\pm}^3}{9}(\frac{b}{3}-cs_{\pm})$ due to $-3a -bs_{\pm}+2cs_{\pm}^2=0$. Then,
\begin{itemize}
  \item if $b>0$ and $a>0$, one has $f_b(s_{+}) < f_b(s_{-}) <f_b(0)$,
  \item if $b>0$ and $0\geq a> -\frac{b^2}{27c}$, one has $f_b(s_{+})  <f_b(0)\leq f_b(s_{-})$,
  \item if $b>0$ and $ -\frac{b^2}{27c} \geq a \geq -\frac{b^2}{24c}$, one has $f_b(s_+)\geq f_b(0), f_b(s_{-})\geq f_b(0)$.
\end{itemize}
These mean
$$\min_{Q\in\mathcal{Q}_0}F_b(Q) =f_b(s_+), \text{ when } b >0 \text{ and } b^2+27ac>0. $$
\end{remark}

%\begin{remark}\label{rem1}
%Under conditions $(\ref{QCon2})$ and $|\xi|$ being sufficiently small, Paicu-Zarnescu \cite{PZ12} proved the existence of global regular solutions with sufficiently regular initial data in $\mathbb{R}^2$. Therefore, the solutions $(v^\epsilon, Q^\epsilon)$ to $(\ref{Q1})$-$(\ref{Q3})$ in Theorem \ref{MT} can be chosen to be smooth enough , such as $v^\epsilon\in L^\infty_tH^1_x\cap L^2_tH^2_x$, and $Q^\epsilon-Q^\infty \in L^\infty_t H^2_x \cap L^2_tH^3_x$ with the initial data $(v_0^\epsilon,Q^\epsilon_0-Q^\infty) \in \mathring{J}\times H^2$.
%\end{remark}
%\begin{remark}
%When $v^\epsilon: (0,T)\times \mathbb{R}^2 \mapsto \mathbb{R}^2$ and $Q^\epsilon: (0,T)\times \mathbb{R}^2 \mapsto \mathbb{M}^{2\times 2}$, under the condition $(\ref{QCon2})$, Cavaterra-Rocca-Wu-Xu \cite{CRWX16} proved the the existence and uniqueness of global strong solutions to the cauchy problem for Beris-Edwards system $(\ref{Q1})$-$(\ref{Q3})$ in the two dimensional periodic case. Therefore, for the two dimensional periodic case, we can find smooth enough sequence $(v^\epsilon, Q^\epsilon)$ just under the condition $(\ref{QCon2})$ such that the convergence of Theorem \ref{MT} holds.
%\end{remark}
\begin{remark}
It follows from direct calculations that
$$\alpha_5+\alpha_6-\frac{\gamma_2^2}{\gamma_1}=-\frac{8\Gamma\xi^2(1-s_+)^2}{9}\leq 0$$
due to $(\ref{MCon2})$-$(\ref{MCon5})$. This, $(\ref{Les3})$ and $(\ref{Les4})$ imply that only some special cases of Ericksen-Leslie systems can be derived from the Beris-Edwards system. The weak solutions obtained in \cite{LLW10,H11,HX12,HLW12,WW14} have at most finite number singular times and are smooth away from the singular times, furthermore, the values of the singular times can be uniquely redefined by the weak-$L^2$ limit through the energy inequalities. Therefore, the uniqueness of the weak solution between the nearest two singular times implies the uniqueness of the global weak solutions, and this kind of results are proved in \cite{LW10,LTX16}. In this paper, different from \cite{LLW10,H11,HX12,HLW12,WW14} where the global weak solution is defined by extending the local strong solution to the Ericksen-Leslie system, we obtain a global-in-time solution to the Ericksen-Leslie system as a limit of the global-in-time solutions to the Beris-Edwards system, and the regularity and uniqueness of such solution are not clear.
\end{remark}

%For simplicity, in the rest of this paper, we use $(v^\epsilon,Q^\epsilon)$ and $(v^*,d^*)$ to denote the weak solutions to the system (\ref{Q1})-(\ref{Q3}) and the system (\ref{G-EL1})-(\ref{G-EL3}) in the Theorem \ref{MT}.

We now make some comments on the main ideas of the proof of Theorem \ref{MT}. As mentioned in Remark \ref{R1}, the asymptotic convergence of solutions to the Beris-Edwards system (\ref{Q1})-(\ref{Q3}) to the regular solutions to the Ericksen-Leslie system (\ref{G-EL1})-(\ref{G-EL3}) with the the coefficients satisfying (\ref{MCon2})-(\ref{MCon5}) has been proved in \cite{WZZ15}. However, the analysis in \cite{WZZ15} is based on the Hilbert expansion, which depends crucially on the high order differentiability of the limiting solutions to the Ericksen-Leslie system and thus cannot be applied to the case that the solutions to the Ericksen-Leslie system have singularities whose existence had been confirmed in  \cite{HLLW16, LLWWZ19}. Here we will establish the asymptotic convergence of these two systems as $\epsilon \rightarrow 0^+$ for weak solutions by analysing the a priori energy inequality (\ref{Q-energy}) for $(v^\epsilon,Q^\epsilon)$ and showing that the weak-limit $(v^*,Q^*)$ of $(v^\epsilon,Q^\epsilon)$ solves the Ericksen-Leslie system (\ref{dd}) subject to the initial data (\ref{EL-IB}) and satisfies the energy inequalities (\ref{d-energy}). This approach is strongly motivated by studies in \cite{LW16,K2020,DHW20} where the weak solutions to the simplified Ericksen-Leslie system are obtained as weak limits of solutions to the Ginzburg-Landau approximation system by weak convergence methods. Here we outline some major elements of the proof of Theorem \ref{MT}. First, the existence of the weak *-limit, $(v^*,Q^*)$, of the $(v^\epsilon,Q^\epsilon)$ is guaranteed by the basic energy inequality (\ref{Q-energy}) and Aubin-Lious Lemma by a standard argument. The key step of the analysis is passing this weak limit into nonlinear terms in the systems. Due to the super-critical nonlinear term $\sum_{l=1}^2\partial_l  (\nabla d\odot \nabla d)_{il}, i=1,2$ in the Ericksen-Leslie system (see (\ref{dd})), it turns out that the most difficult part of the proof of Theorem \ref{MT} is to show that there exists a subsequence of $\{Q^\epsilon\}_{\epsilon>0}$ such that
\begin{equation}\label{I-1}
  \int_0^T\int_{\mathbb{R}^2} \nabla Q^\epsilon (t,x)\odot \nabla Q^\epsilon(t,x):\underline{\nabla \psi}(t,x)dxdt \rightarrow \int_0^T\int_{\mathbb{R}^2} \nabla Q^* (t,x)\odot \nabla Q^*(t,x):\underline{\nabla \psi}(t,x)dxdt
\end{equation}
as $\epsilon\rightarrow 0^+$, for each $\psi \in C_0^\infty([0,T)\times\mathbb{R}^2,\mathbb{R}^3)$, $\partial_1\psi_1 +\partial_2 \psi_2=0$. As in \cite{LW16,K2020,DHW20}, one can define a good time $t$ such that
\begin{equation}\label{I-2}
  \liminf_{\epsilon \rightarrow 0^+}\int_{\mathbb{R}^2} |H^\epsilon|^2(t,x) dx <\infty.
\end{equation}
By the energy inequality (\ref{Q-energy}), (\ref{I-1}) is satisfied as long as
\begin{equation}\label{keyqq}
  \int_{\mathbb{R}^2} \nabla Q^\epsilon (t,x)\odot \nabla Q^\epsilon(t,x):\underline{\nabla \psi}(t,x)dx \rightarrow \int_{\mathbb{R}^2} \nabla Q^* (t,x)\odot \nabla Q^*(t,x):\underline{\nabla \psi}(t,x)dx
\end{equation}
as $\epsilon\rightarrow 0^+$, for each good time and each $\psi \in C_0^\infty([0,T)\times\mathbb{R}^2,\mathbb{R}^3), \partial_1\psi_1 +\partial_2 \psi_2=0$. To prove (\ref{keyqq}), we can establish the following important claim:
\begin{claim}[the strong convergence under samll energy condition]\label{keyclaim}
At a good time, the local strong $H^1$ convergence of $Q^\epsilon$ can be obtained if the local total energy is suitably small (see Lemma $\ref{Lm}$).
\end{claim}
The proof of this claim is the crucial step in the proof of Theorem \ref{MT} and the most technical part in this paper. Once the Claim \ref{keyclaim} is established, we can prove easily the convergence (\ref{keyqq}) by modifying the analysis in \cite{K2020,DHW20}. Indeed, the Claim \ref{keyclaim} implies that $\nabla Q^\epsilon(t,\cdot) \odot \nabla Q^\epsilon(t,\cdot)$ may concentrate on only at a finite number of points at good times. Based on this fact and
\begin{equation*}\label{I-3}
\nabla Q^\epsilon \odot \nabla Q^\epsilon:\underline{\nabla \psi}=\left(
                                                                                        \begin{array}{cc}
                                                                                          \frac12(|\partial_1 Q^\epsilon|^2-|\partial _2 Q^\epsilon|^2) & \partial_1 Q^\epsilon:\partial_2 Q^\epsilon \\
                                                                                          \partial_1 Q^\epsilon:\partial_2 Q^\epsilon & -\frac12(|\partial_1 Q^\epsilon|^2-|\partial _2 Q^\epsilon|^2) \\
                                                                                        \end{array}
                                                                                      \right):\underline{\nabla \psi},
\end{equation*}
one can rule out the potential isolated concentrate points of $\nabla Q^\epsilon(t,\cdot) \odot \nabla Q^\epsilon(t,\cdot)$ by studying the convergence of $ |\partial_1 Q^\epsilon|^2-|\partial _2 Q^\epsilon|^2$ and $\partial_1 Q^\epsilon:\partial_2 Q^\epsilon$ through a Pohozaev type argument as in \cite{DHW20} where the method was used to study the compensated compactness property of solutions to the Ginzburg-Landau approximate equations of the simplified Ericksen-Leslie equations for both uniaxial and biaxial nematics. It should be noted that the convergence (\ref{keyqq}) can also be proved by combining the Claim \ref{keyclaim} here with the concentration-cancellation method in \cite{K2020} which was developed by DiPerna and Majda \cite{DM88} for the incompressible  Euler equations. Note also that the method in \cite{LW16} to rule out the potential concentrate points of $\nabla d^\epsilon\odot \nabla d^\epsilon$ for the Ginzburg-Landau approximate solutions cannot be used here since it depends crucially on $d^\epsilon\in \mathbb{S}^2_+$ which implies $d^*\in\mathbb{S}^2_+$ and that the Liouville theorem of harmonic maps holds.

We now make some comments on the proof of the Claim \ref{keyclaim} above (for more details, see the proof of Lemma \ref{Lm}). Note first that though the corresponding results on the strong convergence under small energy conditions have been proved for the Ginzburg-Landau approximate solutions in \cite{K2020,DHW20},  yet the analysis in \cite{K2020,DHW20} depends crucially on the geometric structure of the Ginzburg-Landau approximation, such as in the case of (\ref{A0}) for uniaxial nematics \cite{K2020}, it holds that
\begin{equation}\label{I-4}
  f(d^\epsilon)-\left( f(d^\epsilon)\cdot \frac{d^\epsilon}{|d^\epsilon|}\right)\cdot \frac{d^\epsilon}{|d^\epsilon|}=0
\end{equation}
with $f(d^\epsilon)=(|d^\epsilon|^2-1)d^\epsilon$. Indeed, one of the key observations in \cite{K2020} is that (\ref{I-4}) implies that the phase function $\psi^\epsilon=\frac{d^\epsilon}{|d^\epsilon|}$ satisfies the following quasi-linear elliptic equation
\begin{equation}\label{I-5}
    \Delta \psi^\epsilon =-|\nabla \psi^\epsilon|^2 \psi^\epsilon -\frac{2}{|d^\epsilon|}\partial_k \psi^\epsilon \partial_k |d^\epsilon | +\frac{1}{|d^\epsilon|}(\tau^\epsilon-(\tau^\epsilon\cdot \psi^\epsilon)\psi^\epsilon)
\end{equation}
with $\tau^\epsilon=d^\epsilon_t +v^\epsilon \cdot \nabla d^\epsilon$. Note that the right hand side of (\ref{I-5}) contains no terms of order $\epsilon^{-1}$, so by the classic theory of elliptic equations, one can obtain the $\epsilon$-independent uniform bound of $||\nabla^2 \psi^\epsilon||_{L^\frac43}$ provided that $||\nabla d^\epsilon||_{L^2}$ is suitably small. This implies the local strong $H^1$ convergence of $d^\epsilon$. For the case of the Ginzburg-Landau approximate equations of the Ericksen-Leslie equations for both uniaxial and biaxial nematics, similar quasi-linear elliptic equations as (\ref{I-5}) were obtained in \cite{DHW20} by using the geometric structure as (\ref{I-4}), see the equation (2.9) in \cite{DHW20}, which yields the local strong $H^1$ convergence under small energy conditions. Unfortunately, this elegant argument cannot be applied easily to the solutions $(v^\epsilon, Q^\epsilon)$ to the Beris-Edwards system due to the structure of the bulk energy density. To see this, one sets
$$Q^\epsilon=e^\epsilon \phi^\epsilon, \quad e^\epsilon =|Q^\epsilon|, \quad \phi^\epsilon=\frac{Q^\epsilon}{|Q^\epsilon|}.$$
It then follows from (\ref{H}) that
\begin{equation}\label{I-6}
  \Delta \phi^\epsilon = -|\nabla \phi^\epsilon|^2\phi^\epsilon -\frac{2}{e^\epsilon}\partial_k e^\epsilon \partial_k \phi^\epsilon +\frac{1}{L_1e^\epsilon }[H^\epsilon -(H^\epsilon:\phi^\epsilon) \phi^\epsilon] +\frac{1}{\epsilon L_1 e^\epsilon}\left[\mathcal{J}(Q^\epsilon)-(\mathcal{J}(Q^\epsilon) :\phi^\epsilon) \phi^\epsilon\right].
\end{equation}
However, it can be checked that
\begin{equation}\label{I-7}
  \mathcal{J}(Q^\epsilon)-(\mathcal{J}(Q^\epsilon) :\phi^\epsilon) \phi^\epsilon\neq 0,
\end{equation}
and so the right hand side of (\ref{I-6}) contains a term of order $\epsilon^{-1}$, which makes it difficult to use the approach in \cite{K2020,DHW20} to obtain the local uniform estimate of $||\nabla^2 \phi^\epsilon||_{L^\frac43}$ even for small energy. Thus new ideas and techniques are needed to establish the strong convergence under small energy conditions for solutions to the Beris-Edwards system. We will prove this by making use of both the geometric structure of $Q^\epsilon$ and the bulk energy density for the Beris-Edwards system and some algebraic properties of $Q^\epsilon$ (see \textbf{Step 3} in the proof of Lemma \ref{Lm}). The main steps and ideas are sketched as follows.
\begin{itemize}
  \item \textbf{Step 1 ($L^\infty$ estimate)} The aim is to show that there exists a suitably small constant $\delta_0>0$ with the corresponding $r_0$ such that if
\begin{equation}\label{small-1}
  \int_{B_{4r_0}(x_0)}\left(|\nabla Q^\epsilon|^2+\frac{\hat{F}_b(Q^\epsilon)}{\epsilon}\right)dx<\delta_0,
\end{equation}
then,
\begin{equation}\label{small-2}
  \text{dist}(Q^\epsilon(x), \mathcal{N})<\delta_0^{\frac18}, x\in B_{3r_0}(x_0).
\end{equation}
Since there is no maximum principle for the system (\ref{Q3}) due to $\xi\ne 0$, it is difficult to get the uniform bound for $||Q^\epsilon||_{L^\infty}$. To overcome this difficulty, we decompose $B_{3r_0}(x_0)$ as the disjoint union of $\mathcal{C}_1^\epsilon$ and $\mathcal{C}_2^\epsilon$ defined as:
$$\mathcal{C}_1^\epsilon=\{x\in B_{3r_0}(x_0): \text{dist}(Q^\epsilon(x), \mathcal{N})<\delta_0^{\frac{1}{10}} \}, \quad \mathcal{C}_2^\epsilon= B_{3r_0(x_0)} \setminus \mathcal{C}_1^\epsilon.$$
Let $\delta^*>0$ be the geometric constant depending on the bulk energy density to be given in Lemma \ref{le-eq}, and choose $\delta_0>0$ so that $\delta_0^{\frac{1}{10}}<\frac14 \delta^*$. It then follows from the continuity of $Q^\epsilon$ that for each $y\in \mathcal{C}_1^\epsilon$, there exists $r^\epsilon_y>0$ such that $ \text{dist}(Q^\epsilon(x),\mathcal{N})<\delta^*$ for $x\in B_{r^\epsilon_y}(y)$, which implies $|\mathcal{J}(Q^\epsilon)|^2\leq C \hat{F}_b(Q^\epsilon)$ on a neighbourhood of $\mathcal{C}_1^\epsilon$ by Lemma \ref{le-eq}. Using this, (\ref{H}), and (\ref{small-1}), one can get by a proper scaling and elliptic estimates that
$$  |Q^\epsilon(x)-Q^\epsilon(y)| \leq C_1 \left[\left(\frac{r_{\bar{y}}^\epsilon}{\sqrt{\epsilon}}\right)^2+1\right]\left( \frac{|x-y|}{\sqrt{\epsilon}}\right)^\frac12 \text{ for } x,y\in B_{r^\epsilon_{\bar{y}}}(\bar{y})$$
with $\bar{y} \in \mathcal{C}_1^\epsilon$, furthermore, $r^\epsilon_{\bar{y}} \geq C_2(\delta^*)^2 \sqrt{\epsilon}$, where $C_2$ is independent of $\epsilon$. This implies that $\hat{F}_b(Q^\epsilon)$ cannot decay too fast on a neighbourhood of $\mathcal{C}_1^\epsilon$. Then choosing $r_{\bar{y}}^\epsilon=C_2(\delta^*)^2 \sqrt{\epsilon}$, one can show by contradiction that $\text{dist}(Q^\epsilon(x), \mathcal{N})<\delta_0^{\frac18}$ for $x\in \mathcal{C}_1^\epsilon$ and suitably small $\delta_0$. This yields that $\mathcal{C}_2^\epsilon=\emptyset$ so (\ref{small-2}) holds.
  \item \textbf{Step 2} We show that there exists $r_3^\epsilon\in(r_0,3r_0)$ such that
\begin{equation}\label{hh1}
  \left|\int_{\partial B_{r_3^\epsilon}(x_0)} \frac{1}{\epsilon} \frac{\partial\hat{F}_b(Q^\epsilon)}{\partial \nu}dS\right|\leq C_0,
\end{equation}
where $C_0$ is independent of $\epsilon$ and $\nu$ is the radial direction. Note that
\begin{equation}\label{small-3}
  \frac{1}{\epsilon} \frac{\partial\hat{F}_b(Q^\epsilon)}{\partial \nu}=\frac{\mathcal{J}(Q^\epsilon)}{\epsilon}:(\frac{x-x_0}{|x-x_0|} \cdot \nabla Q^\epsilon),
\end{equation}
\begin{eqnarray}
  &&\Delta Q^\epsilon :(\frac{x-x_0}{|x-x_0|} \cdot \nabla Q^\epsilon) =  \partial_k \left[\partial_k Q^\epsilon :(\frac{x-x_0}{|x-x_0|} \cdot \nabla Q^\epsilon) -\frac12 |\nabla Q^\epsilon |^2 \frac{(x-x_0)_k}{|x-x_0|}\right]\nonumber\\
   && \qquad-Q^\epsilon_{ij,k}Q^\epsilon_{ij,l} \partial_k\left(\frac{(x-x_0)_l}{|x-x_0|}\right)+\frac12 |\nabla Q^\epsilon|^2 \nabla \cdot\left(\frac{x-x_0}{|x-x_0|}\right). \label{small-4}
\end{eqnarray}
It follows from (\ref{small-1}) that there exist $r_1^\epsilon\in (r_0, \frac32 r_0)$ and $r_2^\epsilon\in (2r_0,3r_0)$ such that
\begin{equation}\label{small-5}
||\nabla Q^\epsilon||_{L^2(\partial B_{r_1^\epsilon}(x_0))}^2<\frac{8\delta_0}{r_0},\quad ||\nabla Q^\epsilon||_{L^2(\partial B_{r_2^\epsilon}(x_0))}^2< \frac{8\delta_0}{r_0}.
\end{equation}
Due to (\ref{H}), it holds that
\begin{equation}\label{small-6}
  \int_{B_{r_2^\epsilon}(x_0)\setminus B_{r_1^\epsilon}(x_0)}(L_1\Delta Q^\epsilon -\frac{\mathcal{J}(Q^\epsilon)}{\epsilon}-H^\epsilon):(\frac{x-x_0}{|x-x_0|} \cdot \nabla Q^\epsilon)dx=0.
\end{equation}
Using (\ref{small-3})-(\ref{small-5}), (\ref{I-2}) and (\ref{small-1}), one can derive from (\ref{small-6}) that
$$\left|\int_{B_{r_2^\epsilon}(x_0)\setminus B_{r_1^\epsilon}(x_0)} \frac{x-x_0}{|x-x_0|}\cdot\nabla \frac{\hat{F}_b(Q^\epsilon)}{\epsilon}dx\right|\leq C_0,$$
which yields the desired (\ref{hh1}) immediately.

\item \textbf{Step 3} Observe that (\ref{H}) implies
\begin{eqnarray}
   && \int_{B_{r^\epsilon_3}(x_0)}\left( L_1^2|\Delta Q^\epsilon|^2+\left|\frac{\mathcal{J}(Q^\epsilon)}{\epsilon}\right|^2\right)\nonumber \\
   &=& -\int_{B_{r^\epsilon_3}(x_0)}\frac{2L_1}{\epsilon}\partial_k Q^\epsilon: \partial_k \mathcal{J}(Q) +\int_{B_{r^\epsilon_3}(x_0)}|H^\epsilon|^2-\int_{\partial B_{r_3^\epsilon}(x_0)} \frac{2L_1}{\epsilon} \frac{\partial\hat{F}_b(Q^\epsilon)}{\partial \nu}. \label{hh2}
\end{eqnarray}
The last two terms on the right hand side of (\ref{hh2}) have been estimated by \textbf{Step 2} and (\ref{I-2}). The most difficult task is to estimate the first
integral on the right hand side of (\ref{hh2}). Observe that
\begin{eqnarray}
   \frac{1}{\epsilon}\partial_k Q^\epsilon: \partial_k \mathcal{J}(Q)&= &\frac{1}{\epsilon}\left[-b\lambda_1^\epsilon |\nabla Q^\epsilon|^2-b \partial_k (Q^\epsilon)^2 :\partial_k Q^\epsilon +\frac{c}{2}|\nabla |Q^\epsilon|^2|^2\right] \nonumber\\
   && +\frac{1}{\epsilon}\left( -a +b\lambda_1^\epsilon +c|Q^\epsilon|^2\right)|\nabla Q^\epsilon|^2:= J_1+J_2  ,\label{small-7}
\end{eqnarray}
where $\lambda_1^\epsilon$ is the minimum eigenvalue of $Q^\epsilon$. One of the key facts is that $J_1\geq 0$ provided that $b>0$, $b^2 +24ac>0$ and $\text{dist} (Q^\epsilon, \mathcal{N})$ is suitably small. This can be proved by very careful and delicate calculations based on the algebraic structure of $Q^\epsilon$ and the properties established in \textbf{Step 1} (for details, see \textbf{Step 3} in the proof of Lemma \ref{Lm}).

  \item \textbf{step 4} It remains to estimate the integral $J_2$. To this end, by using some algebraic properties of $Q^\epsilon$, (\ref{small-2}) in \textbf{Step 1}, and the geometric structure of $\mathcal{J}$ (Lemma \ref{le-eq}), one can obtain easily that $|-a +b\lambda_1^\epsilon +c|Q^\epsilon|^2|\leq C|\mathcal{J}(Q^\epsilon)|$ with $\text{dist} (Q^\epsilon, \mathcal{N})$ being suitably small. It then follows that
\begin{eqnarray}
   && \int_{B_{r_3^\epsilon}} \frac{1}{\epsilon}\left( -a +b\lambda_1^\epsilon +c|Q^\epsilon|^2\right)|\nabla Q^\epsilon|^2 \nonumber\\
  &\leq& \int_{B_{r_3^\epsilon}} \frac12 \left|\frac{\mathcal{J}(Q^\epsilon)}{\epsilon}\right|^2 + C_0 \int_{B_{r_3^\epsilon}} |\nabla Q^\epsilon|^2 \int_{B_{r_3^\epsilon}}(|\Delta Q^\epsilon|^2 + |\nabla Q^\epsilon|^2), \label{small-8}
\end{eqnarray}
 where H\"{o}lder's and Ladyzhenskaya's inequalities have been used. Consequently, we can obtain the uniform estimate of $||\Delta Q^\epsilon||_{L^2(B_{r_3^\epsilon})}$ for small $\delta_0$ by collecting (\ref{hh2})-(\ref{small-8}), which yields the main part of the proof of the Claim \ref{keyclaim}.
\end{itemize}

The rest of this paper is organized as follows: In section 2, some properties of the $Q$-tensor and bulk energy density are discussed; In section 3, we proved the strong convergence under small energy condition at good times; Section 4 is devoted to the proof of the Theorem \ref{MT}.

\section{Properties of the $Q$-tensor and bulk energy density}

\subsection{Properties of the $Q$-tensor}
For a matrix $Q\in \mathcal{N}$, $T_Q\mathcal{N}$ denotes the tangent space to $\mathcal{N}$ at $Q$ in $\mathcal{Q}_0$, $(T_Q\mathcal{N})^\bot_{\mathcal{Q}_0}$ denotes the orthogonal complement of $T_Q\mathcal{N}$ in $\mathcal{Q}_0$, and $\mathcal{P}_{\mathcal{N}}$ denotes the projection operator on $\mathcal{N}$. We list some important geometric properties of $\mathcal{J}(Q^\epsilon)$(see (\ref{J}) for the definition), which will be used later.
%For $Q\in \mathcal{M}$, [][] showed the orthogonal basis of $T_Q\mathcal{M}$ and $(T_Q\mathcal{M})_{\mathcal{N}}^{\bot}$.

\begin{lemma}\cite[Lemma 2.2, Lemma 2.3]{WWZ17}\label{le1}
Let $Q=s_+(d_3\otimes d_3-\frac13 \mathbb{I})\in \mathcal{N}$, and $d_1,d_2$ be unit perpendicular vectors in $V_{d_3}=\{ d^\bot \in \mathbb{S}^2: d^\bot \cdot d_3=0\}$. Then, it holds that
\begin{enumerate}
  \item
  \begin{equation*}
  T_Q\mathcal{N}=\text{Span} \left\{\frac{1}{\sqrt{2}}(d_3\otimes d_2+d_2\otimes d_3), \frac{1}{\sqrt{2}}(d_3\otimes d_1+d_1\otimes d_3)\right\},
  \end{equation*}
  \item \begin{eqnarray*}
          (T_Q\mathcal{N})^\bot_{\mathcal{Q}_0}&=&\text{Span}\left\{\frac{1}{\sqrt{2}}(d_2\otimes d_1+d_1\otimes d_2), \frac{1}{\sqrt{2}}(d_1\otimes d_1-d_2\otimes d_2), \right. \\
           && \left.\sqrt{6} (\frac12 d_1\otimes d_1 +\frac12 d_2\otimes d_2-\frac13 \mathbb{I}) \right\},
        \end{eqnarray*}
  \item $\mathcal{Q}_0 =T_Q \mathcal{N} \oplus (T_Q \mathcal{N})_{\mathcal{Q}_0}^{\bot}$,
  \item For $Q\in \mathcal{Q}_0$, there exists $\delta^*>0$ such that if \text{dist} $(Q,\mathcal{N})<\delta^*$, then
      $$\mathcal{J}(Q)\in (T_{\mathcal{P}_{\mathcal{N}}(Q)}(\mathcal{N}))_{\mathcal{Q}_0}^{\bot}.$$

\end{enumerate}
\end{lemma}
\subsection{The equivalence of bulk energy density}
To estimate the term $\mathcal{J}(Q^\epsilon)$, one needs also the following equivalence of the bulk energy.
\begin{lemma}\cite{MZ10,WWZ17}\label{le-eq}
There exists $\delta^*>0$ such that if $\text{dist}(Q,\mathcal{N})<\delta^*$, then
\begin{equation}\label{E-Fb1}
\frac{1}{C}\text{dist}(Q,\mathcal{N})^2\leq \hat{F}_b(Q)\leq C\text{dist}(Q,\mathcal{N})^2,
\end{equation}
\begin{equation}\label{E-Fb2}
  \frac{1}{C}\hat{F}_b(Q)\leq |\mathcal{J}(Q)|^2\leq C\hat{F}_b(Q),
\end{equation}
where $C$ depends on $a,b$ and $c$, but independent of $Q$, and
$$\text{dist}(Q,\mathcal{N}) =|Q-\mathcal{P}_{\mathcal{N}}(Q)|=\min_{A\in \mathcal{N}}\{ |Q-A|\}.$$
\end{lemma}

\section{The strong convergence under small energy condition}
As discussed in the introduction, in this section, we establish the strong convergence in $H^1$ under the small energy condition for solutions $Q^\epsilon$ to (\ref{Q3}), which is crucial to estimate the set of potential concentration points of $\nabla Q^\epsilon \odot \nabla Q^\epsilon$.
%singular points of weak solutions to the limiting Ericksen-Leslie system.
\begin{lemma}\label{Lm}
For $x_0\in \mathbb{R}^2, r_0>0$ and $B_{4r_0}(x_0)\subset \mathbb{R}^2$, let
\begin{equation}\label{Q-harmonic}
  L_1\Delta Q^\epsilon -\frac{\mathcal{J}(Q^\epsilon)}{\epsilon}=H^\epsilon \quad\text{in } B_{4r_0}(x_0),
\end{equation}
and $Q^\epsilon \in H^3(B_{4r_0}(x_0))$. Assume that
\begin{description}
  \item[(I)] there exists small $\delta_0>0$ such that
  $$\int_{B_{4r_0}(x_0)}\left(|\nabla Q^\epsilon|^2+\frac{\hat{F}_b(Q^\epsilon)}{\epsilon}\right)dx<\delta_0,$$
  \item[(II)] $||H^\epsilon||_{L^2(B_{4r_0}(x_0))}< C_0$,
  \item[(III)]   $b>0, \, b^2+24ac>0,$
\end{description}
where $\delta_0$ and $C_0$ are independent of $\epsilon$. Then there exists a subsequence of $\{Q^\epsilon\}_{\epsilon>0}$ such that
$$Q^\epsilon\rightarrow Q^* ,\text{ in } H^1(B_{r_0}(x_0)),\quad \frac{\mathcal{J}(Q^\epsilon)}{\epsilon}\rightharpoonup \mathcal{J}^* ,\text{ in } L^2(B_{r_0}(x_0)),\quad \int_{B_{r_0}(x_0)}\frac{\hat{F}_b(Q^\epsilon)}{\epsilon}dx\rightarrow 0,$$
as $\epsilon \rightarrow 0$, where $Q^*=s_+(d^*\otimes d^*-\frac13 \mathbb{I}), \, d^*\in \mathbb{S}^2$ satisfies
$$L_1\Delta Q^* -\mathcal{J}^*=H^*$$
in the weak sense, with $H^\epsilon \rightharpoonup H^* \text{ in } L^2(B_{r_0}(x_0))$ and
$$\mathcal{J}^*\in (T_{Q^*}\mathcal{N})_{\mathcal{Q}_0}^{\bot}\text{  a.e. in } B_{r_0}(x_0).$$

\end{lemma}

\begin{proof}
This will be proved by the following four steps.

\textbf{Step 1. Claim:} $\text{\textbf{dist}}\bm{(Q^\epsilon, \mathcal{N})<\delta_0^{\frac18}}$, \textbf{for all} $\bm{ x\in B_{3r_0}(x_0)}$. To prove this claim, we decompose $B_{3r_0}(x_0)$ as a disjoint union of $\mathcal{C}_1^\epsilon$ and $\mathcal{C}_2^\epsilon$ defined as
$$\mathcal{C}_1^\epsilon=\{x\in B_{3r_0}(x_0): \text{dist}(Q^\epsilon(x), \mathcal{N})<\delta_0^{\frac{1}{10}} \}, \quad \mathcal{C}_2^\epsilon= B_{3r_0}(x_0) \setminus \mathcal{C}_1^\epsilon.$$
Since $\hat{F}_b(Q)=0$ if and only if $Q\in \mathcal{N}$, one has
$$\hat{F}_b(Q^\epsilon(x)) \geq C_* \delta_0^{\frac15} \text{ for all } x\in \mathcal{C}_2^\epsilon$$
with some $C_*>0$. This and the condition \textbf{(I)} imply that
$$|\mathcal{C}_2^\epsilon|\leq \frac{\epsilon}{C_*}\delta_0^{\frac45}.$$
Therefore, $\mathcal{C}_1^\epsilon$ is not empty when $\epsilon$ is sufficiently small. Let $\delta^*$ be the geometric quantity given in Lemma \ref{le-eq} and choose $\delta_0$ suitably small so that $\delta_0^\frac{1}{10}<\frac14 \delta^*$. Then it follows from the continuity of $Q^\epsilon$ that for each $y\in \mathcal{C}_1^\epsilon$, there exists $r^\epsilon_y>0$ such that
\begin{equation}\label{r-e}
  \text{dist}(Q^\epsilon(x),\mathcal{N})<\delta^* \text{ for } x\in B_{r^\epsilon_y}(y).
\end{equation}

For fixed $\bar{y}\in \mathcal{C}_{1}^\epsilon$, $0<\sqrt{\epsilon} <r_0$, the rescaled quantity $\hat{Q}^\epsilon (x)= Q^\epsilon(\bar{y}+\sqrt{\epsilon} x)$ satisfies
$$L_1\Delta \hat{Q}^\epsilon  -\mathcal{J}(\hat{Q}^\epsilon) =\hat{H}^\epsilon \text{ in } B_{r_{\bar{y}}^\epsilon/\sqrt{\epsilon}}(\mathbf{0}),$$
where $\hat{H}^\epsilon(x)=\epsilon H(\bar{y}+\sqrt{\epsilon} x)$ and $r_{\bar{y}}^\epsilon$ is given in (\ref{r-e}). It follows from (II), (I), (\ref{r-e}), and the structure of $\mathcal{N}$ that
$$\int_{B_{r_{\bar{y}}^\epsilon/\sqrt{\epsilon}}(\mathbf{0})} |\hat{H}^\epsilon|^2dy =\epsilon\int_{B_{r_{\bar{y}}^\epsilon}(\bar{y})}|H^\epsilon|^2 dx\leq \epsilon C_0$$
and
$$ \int_{B_{r_{\bar{y}}^\epsilon/\sqrt{\epsilon}}(\mathbf{0})}|\hat{Q}^\epsilon|^2 dy \leq \pi \left(\frac{r_{\bar{y}}^\epsilon}{\sqrt{\epsilon}}\right)^2\left(\sqrt{\frac23} s_+ +\delta^*\right)^2.$$
Meanwhile, Lemma \ref{le-eq} and condition \textbf{(I)} imply that
$$\int_{B_{r_{\bar{y}}^\epsilon}/\sqrt{\epsilon}(\mathbf{0})} |\mathcal{J}(\hat{Q}^\epsilon)|^2 dy =\int_{B_{r_{\bar{y}}^\epsilon}(\bar{y})} \frac{|\mathcal{J}(Q^\epsilon)|^2}{\epsilon} dx \leq C \int_{B_{r_{\bar{y}}^\epsilon}(\bar{y})} \frac{\hat{F}_b(Q^\epsilon)}{\epsilon} dx \leq C\delta_0.$$

%\begin{eqnarray*}
 %  \int_{B_{r_{\bar{y}}^\epsilon/\sqrt{\epsilon}}(\mathbf{0})}|\hat{Q}^\epsilon|^2 dy &\leq& C\left(\frac{r_{\bar{y}}^\epsilon}{\sqrt{\epsilon}}\right)^2(||\nabla \hat{Q}^\epsilon||_{L^2(B_{r_{\bar{y}}^\epsilon/\sqrt{\epsilon}}(0))}^2+1) \\
  %&\leq &  C\left(\frac{r_{\bar{y}}^\epsilon}{\sqrt{\epsilon}}\right)^2(||\nabla Q^\epsilon||_{L^2(B_{r_{\bar{y}}^\epsilon}(\bar{y}))}^2+1).
%\end{eqnarray*}
Then, these and the classic elliptic theory \cite[Theorem 9.9]{GT01} yield
$$||\hat{Q}^\epsilon ||_{H^2(B_{r_{\bar{y}}^\epsilon/\sqrt{\epsilon}}(\mathbf{0}))} \leq C \left[\left(\frac{r_{\bar{y}}^\epsilon}{\sqrt{\epsilon}}\right)^2+1\right],$$
where $C$ is independent of $\epsilon$. Therefore, one obtains that
\begin{equation}\label{Hold2}
  |Q^\epsilon(x)-Q^\epsilon(y)| \leq C_1 \left[\left(\frac{r_{\bar{y}}^\epsilon}{\sqrt{\epsilon}}\right)^2+1\right] \left( \frac{|x-y|}{\sqrt{\epsilon}}\right)^\frac12 \text{ in } B_{r_{\bar{y}}^\epsilon}(\bar{y})
\end{equation}
by the embedding theorem.

Next we show that $r_{\bar{y}}^\epsilon\geq C_2 (\delta^*)^2 \sqrt{\epsilon}$ for some $C_2>0$. If $r_{\bar{y}}^\epsilon\leq (\delta^*)^2 \sqrt{\epsilon}$ and $r_{\bar{y}}^\epsilon\leq \frac12 r_0$, one can find a $x^*\in B_{4r_0}(x_0)\cap B_{r_{\bar{y}}^\epsilon}(\bar{y})$ such that
$\text{dist}(Q^\epsilon(x^*),\mathcal{N}) \geq \frac34 \delta^*$. Then, taking $y=\bar{y}$ and $x=x^*$ in the estimate (\ref{Hold2}) leads to
$$\frac{\delta^*}{2}\leq |Q^\epsilon(x^*)-Q^\epsilon(\bar{y})|\leq 2C_1\left(\frac{|x^*-\bar{y}|}{\sqrt{\epsilon}}\right)^{\frac12}\leq 2C_1\left(\frac{r_{\bar{y}}^\epsilon}{\sqrt{\epsilon}}\right)^{\frac12},$$
which yields
$$r_{\bar{y}}^\epsilon\geq \left(\frac{\delta^*}{4C_1}\right)^2\sqrt{\epsilon}.$$
Therefore, setting
$$C_2=\min\left\{1,\frac{1}{16C_1^2}, \frac{r_0}{2(\delta^*)^2}\right\},$$
one gets that $r_{\bar{y}}^\epsilon\geq C_2 (\delta^*)^2 \sqrt{\epsilon}$.

We are now ready to prove that $\text{dist}(Q^\epsilon(x),\mathcal{N})<\delta_0^\frac18$ for $x\in \mathcal{C}_{1}^\epsilon$. If not, there exists $x_1\in \mathcal{C}_1^\epsilon$ with $\text{dist}(Q^\epsilon(x_1),\mathcal{N})\geq \delta_0^\frac18$. Then, using Lemma \ref{le-eq} and the estimate (\ref{Hold2}) with $r_{x_1}^\epsilon=C_2(\delta^*)^2 \sqrt{\epsilon}$($r_{x_1}^\epsilon$ is given in (\ref{r-e})), one can get that
$$\hat{F}_b(Q^\epsilon)\geq C\text{dist}(Q,\mathcal{N})^2\geq C\frac14\delta_0^\frac14=C_3\delta_0^\frac14 \text{ in } B_{\delta_\epsilon}(x_1), \text{ and } \delta_\epsilon=\frac{\delta_0^\frac14\sqrt{\epsilon}}{4C_1^2[C_2^2(\delta^*)^4+1]^2} .$$
Note that $\delta_\epsilon <r^\epsilon_{x_1}$ for suitably small $\delta_0$. Then,
$$\int_{B_{\delta_\epsilon}(x_1)}\frac{\hat{F}_b(Q^\epsilon)}{\epsilon} dx \geq \pi \frac{\delta_0^\frac12}{16C_1^4[C_2^2(\delta^*)^4+1]^4} C_3\delta_0^\frac14 =C_+ \delta_0^{\frac34},$$
which contradicts the assumption that
$$\int_{B_{4r_0}(x_0)} \frac{\hat{F}_b(Q^\epsilon)}{\epsilon} \leq \delta_0$$
for a sufficiently small $\delta_0>0$(for example $\delta_0<C_+^4$). Thus one has shown that
\begin{equation}\label{le-new}
  \text{dist}(Q^\epsilon(x),\mathcal{N}) \leq \delta_0^\frac18 \text{ for } x\in C_1^\epsilon.
\end{equation}

If $C_2^\epsilon$ is not empty, by the definition of $C_1^\epsilon$, then $\text{dist}(Q^\epsilon(z),\mathcal{N}) =\delta_0^{\frac{1}{10}}$ for $z\in \partial \mathcal{C}_1^\epsilon$. This contradicts the estimate (\ref{le-new}) when $\delta_0<1$. Consequently, $\mathcal{C}_2^\epsilon= \emptyset$. Thus the desired claim holds.

\textbf{Step 2. The goal is to show that there exists a $\bm{r_3^\epsilon\in(r_0 ,3r_0)}$ such that}
\begin{equation}\label{Key-L3}
  \bm{\left|\int_{\partial B_{r_3^\epsilon}(x_0)} \frac{1}{\epsilon} \frac{\partial\hat{F}_b(Q^\epsilon)}{\partial \nu}dS\right|\leq C_0.}
\end{equation}
It follows from $||\nabla Q^\epsilon||_{L^2(B_{4r_0}(x_0))}^2 \leq \delta_0$ that for every $\epsilon>0$, there exist $r_1^\epsilon\in (r_0, \frac32 r_0)$ and $r_2^\epsilon\in (2r_0,3r_0)$ such that
\begin{equation}\label{l1-p1}
||\nabla Q^\epsilon||_{L^2(\partial B_{r_1^\epsilon}(x_0))}^2<\frac{8\delta_0}{r_0},\quad ||\nabla Q^\epsilon||_{L^2(\partial B_{r_2^\epsilon}(x_0))}^2< \frac{8\delta_0}{r_0}.
\end{equation}
Multiplying (\ref{Q-harmonic}) by $\frac{x-x_0}{|x-x_0|}\cdot \nabla Q^\epsilon $ and integrating over $B_{r_2^\epsilon}(x_0)\setminus B_{r_1^\epsilon}(x_0)$, one gets
\begin{equation}\label{l1-p2}
  0=\int_{B_{r_2^\epsilon}(x_0)\setminus B_{r_1^\epsilon}(x_0)}(L_1\Delta Q^\epsilon -\frac{\mathcal{J}(Q^\epsilon)}{\epsilon}-H^\epsilon):(\frac{x-x_0}{|x-x_0|} \cdot \nabla Q^\epsilon)dx=I_1+I_2+I_3.
\end{equation}
Now we estimate $I_1,I_2$ and $I_3$ respectively as follows.
\begin{eqnarray*}
  I_1 &=& L_1\int_{B_{r_2^\epsilon}(x_0)\setminus B_{r_1^\epsilon}(x_0)}\Delta Q^\epsilon :(\frac{x-x_0}{|x-x_0|} \cdot \nabla Q^\epsilon)dx \\
  %&=& \int_{B_{r_2^\epsilon}(x_0)\setminus B_{r_1^\epsilon}(x_0)}\left\{Q^\epsilon_{ij,k}Q^\epsilon_{ij,lk}\frac{x_l-x_{0l}}{|x-x_0|}+Q^\epsilon_{ij,k}
  %Q^\epsilon_{ij,l} \partial_k\left(\frac{x_l-x_{0l}}{|x-x_0|}\right)\right\}dx\\
 % &&  -\int_{\partial (B_{r_2^\epsilon}(x_0)\setminus B_{r_1^\epsilon}(x_0))} \frac{\partial Q^\epsilon}{\partial \nu}:(\frac{x-x_0}{|x-x_0|} \cdot \nabla Q^\epsilon)dS\\
  &=& -L_1\int_{B_{r_2^\epsilon}(x_0)\setminus B_{r_1^\epsilon}(x_0)}\left\{Q^\epsilon_{ij,k}Q^\epsilon_{ij,l} \partial_k\left(\frac{x_l-x_{0l}}{|x-x_0|}\right)-\frac12 |\nabla Q^\epsilon|^2\nabla \cdot\left(\frac{x-x_{0}}{|x-x_0|}\right)\right\}dx\\
  &&+L_1\int_{\partial (B_{r_2^\epsilon}(x_0)\setminus B_{r_1^\epsilon}(x_0))} \left\{\frac{\partial Q^\epsilon}{\partial \nu}:(\frac{x-x_0}{|x-x_0|} \cdot \nabla Q^\epsilon) -\frac12|\nabla Q^\epsilon|^2\frac{(x-x_{0})\cdot \nu}{|x-x_0|}\right \}dS,
\end{eqnarray*}
where $\nu$ is the external normal vector of $B_{r_2^\epsilon}(x_0)\setminus B_{r_1^\epsilon}(x_0)$. Therefore, the estimate (\ref{l1-p1}) and condition \textbf{(I)} imply that
$$|I_1|\leq C_0.$$
$$-I_2=\int_{B_{r_2^\epsilon}(x_0)\setminus B_{r_1^\epsilon}(x_0)}\frac{\mathcal{J}(Q^\epsilon)}{\epsilon}:(\frac{x-x_0}{|x-x_0|} \cdot \nabla Q^\epsilon)dx=\int_{B_{r_2^\epsilon}(x_0)\setminus B_{r_1^\epsilon}(x_0)} \frac{x-x_0}{|x-x_0|}\cdot\nabla \frac{\hat{F}_b(Q^\epsilon)}{\epsilon}dx.$$
$$|I_3|=\left|\int_{B_{r_2^\epsilon}(x_0)\setminus B_{r_1^\epsilon}(x_0)}H^\epsilon:(\frac{x-x_0}{|x-x_0|} \cdot \nabla Q^\epsilon)dx\right| \leq ||\nabla Q^\epsilon||_{L^2(B_{4r_0}(x_0))}||H^\epsilon||_{L^2(B_{4r_0}(x_0))}\leq C_0.$$
Substituting the above three estimates into (\ref{l1-p2}) yields
$$\left|\int_{B_{r_2^\epsilon}(x_0)\setminus B_{r_1^\epsilon}(x_0)} \frac{x-x_0}{|x-x_0|}\cdot\nabla \frac{\hat{F}_b(Q^\epsilon)}{\epsilon} dx\right|\leq C_0.$$
Note that $Q^\epsilon\in H^3(B_{4r_0}(x_0))$. Then $\nabla Q^\epsilon$ is continuous. Therefore, there exists $r_3^\epsilon\in (r_1^\epsilon,r_2^\epsilon)$ such that
\begin{equation*}%\label{Key-L3}
  \left|\int_{\partial B_{r_3^\epsilon}(x_0)} \frac{1}{\epsilon} \frac{\partial\hat{F}_b(Q^\epsilon)}{\partial \nu}dS\right|=\left|\int_{\partial B_{r_3^\epsilon}(x_0)} \frac{x-x_0}{|x-x_0|}\cdot\nabla \frac{\hat{F}_b(Q^\epsilon)}{\epsilon} dS\right|\leq \frac{2C_0}{r_0}\leq C_0,
\end{equation*}
which gives the desired estimate (\ref{Key-L3}).

To obtain the strong convergence, we are going to derive the uniform estimate on $||\Delta Q^\epsilon||_{L^2(B_{r^\epsilon_3}(x_0))}$ as follows.
Multiplying (\ref{Q-harmonic}) by $L_1\Delta Q^\epsilon - \frac{\mathcal{J}(Q^\epsilon)}{\epsilon}$ and integrating over $B_{r_3^\epsilon}(x_0)$ lead to
\begin{eqnarray}
   && \int_{B_{r^\epsilon_3}(x_0)}\left( L_1^2|\Delta Q^\epsilon|^2+\left|\frac{\mathcal{J}(Q^\epsilon)}{\epsilon}\right|^2\right)dx\nonumber \\
  &=& \int_{B_{r^\epsilon_3}(x_0)} |H^\epsilon|^2dx+\int_{\partial B_{r_3^\epsilon}(x_0)} \frac{2L_1\partial\hat{F}_b(Q^\epsilon)}{\epsilon\partial \nu}dS - \int_{B_{r^\epsilon_3}(x_0)}\frac{2L_1}{\epsilon}\partial_k Q^\epsilon: \partial_k \mathcal{J}(Q)dx. \label{l1-p5}
\end{eqnarray}
Due to condition (\textbf{II}) and (\ref{Key-L3}), the desired uniform estimate on the left hand side of (\ref{l1-p5}) can be achieved once the last integral on the right hand side of (\ref{l1-p5}) can be handled.

To this end, one calculates that
\begin{eqnarray}
   \frac{1}{\epsilon}\partial_k Q^\epsilon: \partial_k \mathcal{J}(Q)&= &\frac{1}{\epsilon}\left[-b\lambda_1^\epsilon |\nabla Q^\epsilon|^2-b \partial_k (Q^\epsilon)^2 :\partial_k Q^\epsilon +\frac{c}{2}|\nabla |Q^\epsilon|^2|^2\right] \nonumber\\
   && +\frac{1}{\epsilon}\left( -a +b\lambda_1^\epsilon +c|Q^\epsilon|^2\right)|\nabla Q^\epsilon|^2:=J_1+J_2, \label{j1j2}
\end{eqnarray}
where $\lambda_1^\epsilon$ is the smallest eigenvalue of $Q^\epsilon$. Then our key observation here is that $J_1$ has the favorable sign and the integral involving $J_2$ can be bounded uniformly, which will be derived in \textbf{Step 3} and \textbf{Step 4} respectively.

\textbf{Step 3. we will show that $\bm{J_1}$ is nonnegative, i.e.}
  \begin{equation}\label{J1}
    \bm{J_1(x) \geq 0} \text{ \textbf{for all} } \bm{x\in B_{3r_0}(x_0)}.
  \end{equation}
To prove this, we will examine the structure of $Q^\epsilon$ in details.
Let $\lambda_1^\epsilon, \lambda_2^\epsilon$ and $\lambda_3^\epsilon$ be the eigenvalues of $Q^\epsilon$ with the corresponding eigenvectors $d_1^\epsilon, d_2^\epsilon$ and $d_3^\epsilon \in \mathbb{S}^2$ respectively. Then, Since $Q^\epsilon$ is symmetric and with zero trace,  $Q^\epsilon$ can be represented as
\begin{equation}\label{Qepsilon}
  Q^\epsilon=\lambda_1^\epsilon d_1^\epsilon\otimes d_1^\epsilon+\lambda_2^\epsilon d_2^\epsilon\otimes d_2^\epsilon+\lambda_3^\epsilon d_3^\epsilon\otimes d_3^\epsilon, \quad \lambda_1^\epsilon+\lambda_2^\epsilon+\lambda_3^\epsilon=0.
\end{equation}
Without loss of generality, it can be assumed that
$$\lambda_1^\epsilon\leq\lambda_2^\epsilon\leq\lambda_3^\epsilon.$$
Set
\begin{equation}\label{DD}
  D^\epsilon_{12}=\sum_{k=1}^3|d_1^\epsilon\cdot\partial_kd_2^\epsilon  |^2,\quad D_{13}^\epsilon =\sum_{k=1}^3|d_1^\epsilon\cdot\partial_k d_3^\epsilon  |^2, \quad D_{23}^\epsilon =\sum_{k=1}^3|d_2^\epsilon\cdot\partial_k d_3^\epsilon  |^2.
\end{equation}
Since $d_1^\epsilon,d_2^\epsilon$ and $d_3^\epsilon$ form an orthonormal basis to $\mathbb{R}^3$, it holds that
\begin{equation}\label{DDD}
|\nabla d_1^\epsilon|^2 =D_{12}^\epsilon+D_{13}^\epsilon,\quad |\nabla d_2^\epsilon|^2 =D_{12}^\epsilon+D_{23}^\epsilon, \quad |\nabla d_3^\epsilon|^2 =D_{13}^\epsilon+D_{23}^\epsilon.
\end{equation}
It follows from the structure of $Q^\epsilon$ ((\ref{Qepsilon})), (\ref{DD}), (\ref{DDD}) and detailed calculations that
\begin{eqnarray}
   |\nabla Q^\epsilon|^2 &=&  |\nabla \lambda_1^\epsilon|^2 +|\nabla\lambda_2^\epsilon|^2+|\nabla\lambda_3^\epsilon|^2 +2(\lambda_1^\epsilon)^2 |\nabla d_1^\epsilon|^2+2(\lambda_2^\epsilon)^2 |\nabla d_2^\epsilon|^2+2(\lambda_3^\epsilon)^2 |\nabla d_3^\epsilon|^2 \nonumber\\
   &&-4 \lambda_1^\epsilon\lambda_2^\epsilon D_{12}^\epsilon -4 \lambda_1^\epsilon\lambda_3^\epsilon D_{13}^\epsilon-4 \lambda_2^\epsilon\lambda_3^\epsilon D_{23}^\epsilon\nonumber\\
   &=& |\nabla \lambda_1^\epsilon|^2 +|\nabla\lambda_2^\epsilon|^2+|\nabla\lambda_3^\epsilon|^2 +2(\lambda_3^\epsilon-\lambda_2^\epsilon)^2 |\nabla d_3^\epsilon|^2 + E_1^\epsilon,\label{tiduQ}
\end{eqnarray}
where
$$E_1^\epsilon =6\lambda_3^\epsilon (\lambda_2^\epsilon-\lambda_1^\epsilon)D_{13}^\epsilon+2(\lambda_2^\epsilon -\lambda_1^\epsilon)^2D_{12}^\epsilon \geq 0.$$
Meanwhile, it holds that
$$(Q^\epsilon)^2=(\lambda_1^\epsilon)^2 d_1^\epsilon \otimes d_1^\epsilon +(\lambda_2^\epsilon)^2 d_2^\epsilon \otimes d_2^\epsilon +(\lambda_3^\epsilon)^2 d_3^\epsilon \otimes d_3^\epsilon.$$
Then, a detailed calculation using (\ref{Qepsilon})-(\ref{DDD}) yields
\begin{eqnarray*}
   \partial_k (Q^\epsilon)^2 : \partial _k Q^\epsilon& =& 2\lambda_1^\epsilon|\nabla \lambda_1^\epsilon|^2 +2\lambda_2^\epsilon|\nabla \lambda_2^\epsilon|^2+2\lambda_3^\epsilon|\nabla \lambda_3^\epsilon|^2 +2(\lambda_1^\epsilon)^3 |\nabla d_1^\epsilon|^2  \nonumber\\
   && + 2(\lambda_2^\epsilon)^3 |\nabla d_2^\epsilon|^2+2(\lambda_3^\epsilon)^3 |\nabla d_3^\epsilon|^2 -2[(\lambda_1^\epsilon)^2 \lambda_2 +(\lambda_2^\epsilon)^2 \lambda_1^\epsilon]D_{12}^\epsilon \nonumber\\
   &&-2[(\lambda_1^\epsilon)^2 \lambda_3 +(\lambda_3^\epsilon)^2 \lambda_1^\epsilon]D_{13}^\epsilon-2[(\lambda_2^\epsilon)^2 \lambda_3 +(\lambda_3^\epsilon)^2 \lambda_2^\epsilon]D_{23}^\epsilon.
\end{eqnarray*}
Due to (\ref{Qepsilon}) and (\ref{DDD}) again, $\partial_k (Q^\epsilon)^2 : \partial _k Q^\epsilon$ can be rewritten as
\begin{equation}\label{bb}
 \partial_k (Q^\epsilon)^2 : \partial _k Q^\epsilon = 2\lambda_1^\epsilon|\nabla \lambda_1^\epsilon|^2 +2\lambda_2^\epsilon|\nabla \lambda_2^\epsilon|^2+2\lambda_3^\epsilon|\nabla \lambda_3^\epsilon|^2 -2\lambda_1^\epsilon(\lambda_3^\epsilon-\lambda_2^\epsilon)^2|\nabla d_3^\epsilon|^2  +E^\epsilon_2,
\end{equation}
where
$$E_2^\epsilon=2(\lambda_1^\epsilon +\lambda_2^\epsilon)(\lambda_1^\epsilon -\lambda_2^\epsilon)^2 D_{12}^\epsilon + 2(\lambda_1^\epsilon-\lambda_2^\epsilon)[(\lambda_1^\epsilon)^2+(\lambda_2^\epsilon)^2 +\lambda_1^\epsilon\lambda_2^\epsilon]D_{13}^\epsilon \leq 0.$$
It follows from \textbf{Step 1} that $\lambda_1^\epsilon<0, \lambda_2^\epsilon<0$ and $\lambda_3^\epsilon>0$ in $B_{3r_0}(x_0)$ for suitable small $\delta_0$. This and  $\lambda_1^\epsilon\leq\lambda_2^\epsilon\leq\lambda_3^\epsilon$ imply that
\begin{equation}\label{zhong-E3}
  -b\lambda_1^\epsilon|\nabla Q^\epsilon|^2 -b\partial_k (Q^\epsilon)^2: \partial_k Q^\epsilon +\frac12 c|\nabla |Q^\epsilon|^2|^2\geq E^\epsilon_3,
\end{equation}
where
$$E_3^\epsilon=-b\lambda_1^\epsilon(|\nabla \lambda_1^\epsilon|^2+|\nabla \lambda_2^\epsilon|^2+|\nabla \lambda_3^\epsilon|^2)-b(2\lambda_1^\epsilon|\nabla \lambda_1^\epsilon|^2 +2\lambda_2^\epsilon|\nabla \lambda_2^\epsilon|^2+2\lambda_3^\epsilon|\nabla \lambda_3^\epsilon|^2)+\frac{c}{2}|\nabla |Q^\epsilon|^2|^2.$$
Since $\lambda_1^\epsilon +\lambda_2^\epsilon+\lambda_3^\epsilon=0$, one has
\begin{eqnarray*}
  &&|\nabla |Q^\epsilon|^2|^2 =4|2\lambda_1^\epsilon \nabla \lambda_1^\epsilon +2\lambda_2^\epsilon \nabla \lambda_2^\epsilon +\lambda_2^\epsilon\nabla \lambda_1^\epsilon +\lambda_1^\epsilon \nabla \lambda_2^\epsilon|^2 \\
  &&\qquad =4(2\lambda_1^\epsilon +\lambda_2^\epsilon)^2|\nabla \lambda_1^\epsilon|^2 + 8(2\lambda_1^\epsilon +\lambda_2^\epsilon)(2\lambda_2^\epsilon +\lambda_1^\epsilon)\nabla \lambda_1^\epsilon\cdot \nabla \lambda_2^\epsilon +4(2\lambda_2^\epsilon +\lambda_1^\epsilon)^2|\nabla \lambda_2^\epsilon|^2.
\end{eqnarray*}
On the other hand,
\begin{eqnarray*}
   && -b[ \lambda_1^\epsilon(|\nabla \lambda_1^\epsilon|^2+|\nabla \lambda_2^\epsilon|^2+|\nabla \lambda_3^\epsilon|^2)+(2\lambda_1^\epsilon|\nabla \lambda_1^\epsilon|^2 +2\lambda_2^\epsilon|\nabla \lambda_2^\epsilon|^2+2\lambda_3^\epsilon|\nabla \lambda_3^\epsilon|^2)]\\
   &=& -b[(2\lambda_1^\epsilon-2\lambda_2^\epsilon)|\nabla \lambda_1^\epsilon|^2 -2(\lambda_1^\epsilon+2\lambda_2^\epsilon)\nabla \lambda_1^\epsilon\cdot \nabla \lambda_2^\epsilon]\geq 2b(\lambda_1^\epsilon+2\lambda_2^\epsilon)\nabla \lambda_1^\epsilon\cdot \nabla \lambda_2^\epsilon.
\end{eqnarray*}
Therefore, these together with condition \textbf{(III)} imply that
\begin{eqnarray}
  E_3^\epsilon &\geq& 2c (2\lambda_1^\epsilon +\lambda_2^\epsilon)^2|\nabla \lambda_1^\epsilon|^2 +2c(2\lambda_2^\epsilon +\lambda_1^\epsilon)^2|\nabla \lambda_2^\epsilon|^2 \nonumber\\
   &&+ 2\left[2c(2\lambda_1^\epsilon +\lambda_2^\epsilon)+b\right](2\lambda_2^\epsilon+\lambda_1^\epsilon) \nabla \lambda_1^\epsilon\cdot \nabla \lambda_2^\epsilon\geq 0 \label{zhong-E33}
\end{eqnarray}
for sufficiently small $\delta_0$, where one has used the expression of $s_+$ and the fact that
$$\text{dist}(Q^\epsilon(x),\mathcal{N})^2=(\lambda_1^\epsilon(x)+\frac{s_+}{3})^2+(\lambda_2^\epsilon(x)+\frac{s_+}{3})^2 +(\lambda_3^\epsilon(x)-\frac{2s_+}{3})^2 \leq \delta_0^\frac14, \, x\in B_{3r_0}(x_0).$$
Substituting the estimate (\ref{zhong-E33}) into (\ref{zhong-E3}) yeilds
\begin{equation}\label{Key-L1}
  -b\lambda_1^\epsilon|\nabla Q^\epsilon|^2 -b\partial_k (Q^\epsilon)^2: \partial_k Q^\epsilon +\frac{c}{2}|\nabla |Q^\epsilon|^2|^2\geq 0 \text{ in } B_{3r_0}(x_0)
\end{equation}
for sufficiently small $\delta_0$, which proves the desired estimate (\ref{J1}). \\

\textbf{Step 4. We now estimate the integral involving $\bm{J_2}$ in (\ref{j1j2}) and establish finally the uniform estimate on $\bm{||\Delta Q^\epsilon||_{L^2(B_{r^\epsilon_3}(x_0))} }$.} Set
$$g(\lambda_1^\epsilon, \lambda_3^\epsilon) =-a +b\lambda_1^\epsilon +c|Q^\epsilon|^2=-a +b\lambda_1^\epsilon +2c[(\lambda_1^\epsilon)^2 +(\lambda_3^\epsilon)^2 +\lambda_1^\epsilon\lambda_3^\epsilon].$$
It follows from the definition of $s_+$ that $g(-\frac{s_+}{3}, \frac{2s_+}{3})=0$. Then, it holds that
\begin{equation*}
g(\lambda_1^\epsilon, \lambda_3^\epsilon)=b\left(\lambda_1^\epsilon +\frac{s_+}{3}\right) +2cs_+\left(\lambda_3^\epsilon-\frac{2s_+}{3}\right) +o\left(\sqrt{|\lambda_1^\epsilon +\frac{s_+}{3}|^2 +|\lambda_3^\epsilon-\frac{2s_+}{3}|^2}\right)
\end{equation*}
near the point $(-\frac{s_+}{3}, \frac{2s_+}{3})$. Note that
$$\text{dist}(Q^\epsilon, \mathcal{N})^2=(\lambda^\epsilon_1+\frac{s_+}{3})^2+(\lambda^\epsilon_2+\frac{s_+}{3})^2+(\lambda^\epsilon_3-\frac{2s_+}{3})^2.$$
Thus, for suitably small $\delta_0$, one has that
\begin{equation*}%\label{Key-L2}
|g(\lambda_1^\epsilon, \lambda_3^\epsilon)|\leq C\sqrt{|\lambda_1^\epsilon +\frac{s_+}{3}|^2 +|\lambda_3^\epsilon-\frac{2s_+}{3}|^2} \leq C \text{dist}(Q^\epsilon, \mathcal{N})\leq  C |\mathcal{J}(Q^\epsilon)|,
\end{equation*}
where Lemma \ref{le-eq} and the estimate (\ref{le-new}) have been used. Hence
\begin{equation}\label{Key-L2}
  |J_2|\leq \frac{C}{\epsilon}|\mathcal{J}(Q^\epsilon)||\nabla Q^\epsilon|^2 \text{ in } B_{3r_0}(x_0).
\end{equation}
It then follows from (\ref{Key-L3}), (\ref{l1-p5})-(\ref{J1}), and (\ref{Key-L1})-(\ref{Key-L2}) that
\begin{eqnarray*}
   &&\int_{B_{r_3^\epsilon}(x_0)}\left( L_1^2|\Delta Q^\epsilon|^2  + \left|\frac{\mathcal{J}(Q^\epsilon)}{\epsilon}\right|^2 \right)dx\\
   &\leq&\int_{B_{r_3^\epsilon}(x_0)}\left( |H^\epsilon|^2 + \frac{C}{\epsilon}|\mathcal{J}(Q^\epsilon)||\nabla Q^\epsilon|^2\right)dx +C_0\\
  &\leq&  \int_{B_{r_3^\epsilon}(x_0)}\left( |H^\epsilon|^2 + \frac12\left|\frac{\mathcal{J}(Q^\epsilon)}{\epsilon} \right|^2+ C_0 |\nabla Q^\epsilon|^4\right)dx +C_0.
\end{eqnarray*}
Meanwhile, by Ladyzhenskaya's inequality, one has
$$\int_{B_{r_3^\epsilon}(x_0)} |\nabla Q^\epsilon|^4 dx \leq C_0 \int_{B_{r_3^\epsilon}(x_0)} |\nabla Q^\epsilon|^2 dx \int_{B_{r_3^\epsilon}(x_0)}(|\Delta Q^\epsilon|^2 + |\nabla Q^\epsilon|^2)dx.$$
Since $||\nabla Q^\epsilon||_{L^2(B_{4r_0}(x_0))}^2 <\delta_0$ and $\delta_0$ can be chosen very small, one can get
\begin{equation}\label{Key}
  \int_{B_{r_3^\epsilon}(x_0)}\left( |\Delta Q^\epsilon|^2  + \left|\frac{\mathcal{J}(Q^\epsilon)}{\epsilon}\right|^2 \right)dx \leq C_4,
\end{equation}
where $C_4$ is independent of $\epsilon$. Therefore, there exists a subsequence of $\{Q^\epsilon\}_{\epsilon>0}$ such that
\begin{equation}\label{LStrong1}
  Q^{\epsilon}\rightarrow Q^*= s_+(d^*\otimes d^*-\frac13 \mathbb{I}) \text{ in } H^1(B_{r_0}(x_0)), d^*\in \mathbb{S}^2,
\end{equation}
\begin{equation}\label{LStrong2}
  \frac{\mathcal{J}(Q^{\epsilon})}{\epsilon} \rightharpoonup \mathcal{J}^*\text{ in } L^2(B_{r_0}(x_0)),
\end{equation}
as $ \epsilon\rightarrow 0$. Furthermore, by Step 1 and lemma \ref{le-eq}, one can get
\begin{eqnarray}
  \int_{B_{r_0}(x_0)} \frac{\hat{F}_b(Q^\epsilon)}{\epsilon}dx &\leq& C \int_{B_{r_0}(x_0)} |Q^\epsilon-Q^*|\left|\frac{\mathcal{J}(Q^\epsilon)}{\epsilon}\right|dx  \nonumber\\
  &\leq & C ||Q^\epsilon-Q^*||_{L^2(B_{r_0}(x_0))}\left\|\frac{\mathcal{J}(Q^\epsilon)}{\epsilon}\right\|_{L^2(B_{r_0}(x_0))} \rightarrow 0 \label{LStrong3}
\end{eqnarray}
as $ \epsilon\rightarrow 0$.

\textbf{Step 5. Finally, we show that}
\begin{equation}\label{daJ}
  \bm{\mathcal{J}^*\in (T_{Q^*}\mathcal{N})_{\mathcal{Q}_0}^{\bot} \text{ \textbf{a.e. in} } B_{r_0}(x_0).}
\end{equation}
Recalling (\ref{tiduQ}) in \textbf{Step 3}, one has that
\begin{eqnarray*}
   |\nabla Q^\epsilon|^2 &=& |\nabla \lambda_1^\epsilon|^2 +|\nabla\lambda_2^\epsilon|^2+|\nabla\lambda_3^\epsilon|^2 +2(\lambda_3^\epsilon-\lambda_2^\epsilon)^2 |\nabla d_3^\epsilon|^2\\
   && +6\lambda_3^\epsilon (\lambda_2^\epsilon-\lambda_1^\epsilon)D_{13}^\epsilon+2(\lambda_2^\epsilon -\lambda_1^\epsilon)^2D_{12}^\epsilon  \geq 2(\lambda_3^\epsilon-\lambda_2^\epsilon)^2 |\nabla d_3^\epsilon|^2.
\end{eqnarray*}
On the other hand, the claim in \textbf{Step 1} implies that
$|\lambda_3^\epsilon -\lambda_2^\epsilon|\geq \frac12 s_+$ for all $x\in B_{3r_0}(x_0)$. Therefore $|\nabla d_3^\epsilon|^2 \leq \frac{2}{s_+^2}|\nabla Q^\epsilon|^2$ holds, which yields that
\begin{equation}\label{d3Q}
  ||\nabla d^\epsilon_3||_{L^p(B_{r_0}(x_0))} \leq C ||\nabla Q^\epsilon||_{L^p(B_{r_0}(x_0))},\quad 1\leq p<\infty.
\end{equation}
Note also that (\ref{Key}) in \textbf{Step 4} implies $||Q^\epsilon||_{H^2(B_{r_0}(x_0))} \leq C$. Thus
$$||\nabla Q^\epsilon||_{L^p(B_{r_0}(x_0))}\leq C, \quad 1\leq p<\infty$$
by the embedding theorem, which together with (\ref{d3Q}), yields that
\begin{equation}\label{dstrongC}
  d_3^\epsilon \rightarrow d^* \text{ in } C^\alpha(B_{r_0}(x_0)), 0<\alpha<1, \text{ as } \epsilon \rightarrow 0.
\end{equation}
Set
$$e_1^\epsilon=\frac{1}{\sqrt{2}}(d^\epsilon_3\otimes d_2^\epsilon+d_2^\epsilon\otimes d^\epsilon_3), e_2^\epsilon=\frac{1}{\sqrt{2}}(d^\epsilon_3\otimes d_1^\epsilon+d_1^\epsilon\otimes d^\epsilon_3),e_3^\epsilon=\frac{1}{\sqrt{2}}(d_2^\epsilon\otimes d_1^\epsilon+d_1^\epsilon\otimes d_2^\epsilon),$$
$$ e_4^\epsilon=\frac{1}{\sqrt{2}}(d_1^\epsilon\otimes d_1^\epsilon-d_2^\epsilon\otimes d_2^\epsilon),e_5^\epsilon=\sqrt{6} (\frac12 d_1^\epsilon\otimes d_1^\epsilon +\frac12 d_2^\epsilon\otimes d_2^\epsilon-\frac13 \mathbb{I}),$$
where $d_1^\epsilon,d_2^\epsilon\in V_{d^\epsilon_3}$(see Lemma \ref{le1} for the definition of $V_{d^\epsilon_3}$) and $d_1^\epsilon\cdot d_2^\epsilon=0$. Similarly, $e_1^*$ and $e_2^*$ can also be defined in a same way as $e^\epsilon_1$ and $e^\epsilon_2$ with $d^*, d^*_1,d^*_2$, where $d^*$ is given in (\ref{LStrong1}) and $d^*_1,d^*_2\in V_{d^*}, d^*_1\cdot d^*_2=0$. Obviously, $(T_{\mathcal{P}_{\mathcal{N}}(Q^\epsilon)}\mathcal{N})_{\mathcal{Q}_0}^{\bot}=Span\{e_3^\epsilon,e_4^\epsilon,e_5^\epsilon\}$.
Lemma \ref{le1} yields that
$$\mathcal{J}(Q^{\epsilon}) =a_1^\epsilon e_3^\epsilon+a_2^\epsilon e_4^\epsilon+a_3^\epsilon e_5^\epsilon,$$
where $a_1^\epsilon,a_2^\epsilon,a_3^\epsilon\in L^2(B_{r_0}(x_0))$.
Then, by taking the limit $\epsilon\rightarrow 0^+$ and (\ref{LStrong2}) and (\ref{dstrongC}), one can get
$$\frac{\mathcal{J}(Q^{\epsilon})}{\epsilon}:( d^\epsilon_3\otimes d_2^* +d_2^*\otimes d^\epsilon_3)=-\frac{2\sqrt{6}}{3}a_3^\epsilon d^\epsilon_3\cdot d_2^*  \rightharpoonup \mathcal{J}^*:e_1^*=0 \text{ in } L^2(B_{r_0}(x_0)),$$
$$\frac{\mathcal{J}(Q^{\epsilon})}{\epsilon}:( d^\epsilon_3\otimes d_1^* +d_1^*\otimes d^\epsilon_3)=-\frac{2\sqrt{6}}{3}a_3^\epsilon d^\epsilon_3\cdot d_1^* \rightharpoonup \mathcal{J}^*:e_2^*=0 \text{ in } L^2(B_{r_0}(x_0)).$$
This implies (\ref{daJ}) by Lemma \ref{le1}.

%Since $\frac{\mathcal{J}(Q^{\epsilon_i})}{\epsilon_i}\in \mathcal{Q}_0$, it is not hard to see $\mathcal{J}^*\in \mathcal{Q}_0$ a.e. $x\in B_{r_0}(x_0)$, thus there exists $\lambda\in L^2(B_{r_0}(x_0), \mathbb{R}^5)$ such that
%$$\mathcal{J}^*(x)=\sum_{i=1}^5\lambda_i(x)e_i \text{ for almost every } x\in B_{r_0}(x_0),$$
%by Lemma \ref{le1}. By $\mathcal{J}^*\cdot n^*=\lambda_5 n^*$ and direct calculation, it is easy to see $\lambda_1(x)=\lambda_2(x)=0$ a.e. $x\in B_{r_0}(x_0)$.

\end{proof}

\section{Proof of main results}
%This section is divided into three parts: the limit of Beris-Edwards system (\ref{BE1})-(\ref{Q-IB}) satisfies the Ericksen-Leslie system (\ref{EL1})-(\ref{EL-IB}) in weak sense, and it satisfies the energy inequality (\ref{ELenergy}), and it continuous to its initial data.
 %Firstly, we prove that the limit of Beris-Edwards system (\ref{BE1})-(\ref{Q-IB}) satisfies the Ericksen-Leslie system (\ref{EL1})-(\ref{EL-IB}) in weak sense. Secondly, we prove that the limit of Beris-Edwards system (\ref{BE1})-(\ref{Q-IB}) satisfies the energy inequality (\ref{ELenergy}). Thirdly,  We prove that the limit of Beris-Edwards system (\ref{BE1})-(\ref{Q-IB}) continuous to its initial data in $L^2$.

Proof of Theorem \ref{MT} \textbf{Step 1. Convergence of $\bm{(v^\epsilon, Q^\epsilon)}$.}
It follows from the energy inequality (\ref{Q-energy}) that there exists a subsequence of $\{(v^\epsilon,Q^\epsilon)\}_{\epsilon>0}$ such that
\begin{equation}\label{con1}
  v^\epsilon\mathop{\rightharpoonup}\limits^{\star} v^* \text{ in } L^\infty(0,T;L^2(\mathbb{R}^2)),\quad v^\epsilon\rightharpoonup v^* \text{ in } L^2(0,T; H^1(\mathbb{R}^2)),
\end{equation}
\begin{equation}\label{con2}
  \nabla Q^\epsilon \mathop{\rightharpoonup}\limits^{\star} \nabla Q^* \text{ in } L^\infty(0,T; L^2(\mathbb{R}^2)), \quad H^\epsilon \rightharpoonup H^* \text{ in } L^2(0,T; L^2(\mathbb{R}^2)),
\end{equation}
where $Q^*=s_+(d^*\otimes d^* -\frac13 \mathbb{I})$ a.e. in $\mathbb{R}^2$ and $d^*: (0,T)\times \mathbb{R}^2 \mapsto \mathbb{S}^2$.

For any smooth bounded domain $\tilde{\Omega}\subset \mathbb{R}^2$ and any $\psi\in L^2H^3_0((0,T)\times\tilde{\Omega},\mathbb{R}^3), \partial_1 \psi_1 +\partial_2\psi_2=0$, one can derive from (\ref{Qw1}), H\"{o}lder inequalities, and the Sobolev embedding theorem that
\begin{eqnarray}
  &&|<v_t^\epsilon , \psi>|\nonumber\\
   &\leq& \left| \int_0^T\int_{\mathbb{R}^2}(\eta \overline{D^\epsilon}:\overline{\nabla \psi} -v^\epsilon\otimes v^\epsilon:\overline{\nabla \psi}  -L_1\nabla Q^\epsilon \odot \nabla Q^\epsilon:\underline{\nabla \psi}  )dxdt\right|\nonumber \\
   &&+\left| \int_0^T\int_{\mathbb{R}^2}[Q^\epsilon\cdot H^\epsilon-H^\epsilon \cdot Q^\epsilon-S_{Q^\epsilon}(H^\epsilon)]:\overline{\nabla \psi} dxdt\right|\nonumber \\
   &\leq&C\left(||D^\epsilon||_{L^2}||\nabla \psi||_{L^2} +||v^\epsilon||_{L^\infty_tL^2_x}||\nabla v^\epsilon||_{L^2}|| \psi||_{L^2_tL_x^\infty} +||\nabla Q^\epsilon||_{L^\infty_tL^2_x}^2 ||\nabla \psi||_{L^2_tL^\infty_x} \right)\nonumber\\
   &&+C||H^\epsilon||_{L^2} (|| Q^\epsilon||_{L^\infty_tL^4_x(\tilde{\Omega})}^2+|| Q^\epsilon||_{L^\infty_tL^2_x(\tilde{\Omega})})||\nabla \psi||_{L^2_tL^\infty_x} +C||H^\epsilon||_{L^2}||\nabla \phi||_{L^2} \nonumber\\
   &\leq &C||\psi||_{L^2_tH^3_x} \label{v*t}
\end{eqnarray}
with $\overline{D^\epsilon}$ given in (\ref{X-1}). Therefore, $v_t^\epsilon$ is uniformly bounded in $L^2(0,T; H^{-3}(\tilde{\Omega}))$.
Similarly, for any $\varphi\in L^4((0,T)\times\tilde{\Omega},\mathcal{Q}_0)$, one can get from (\ref{Qw2}) that
\begin{eqnarray}
   && |<Q^\epsilon_t , \varphi>| \nonumber \\
   &=& \left |  \int_0^T\int_{\mathbb{R}^2} [-\underline{v^\epsilon} \cdot \nabla Q^\epsilon + \frac{1}{\Gamma} H^\epsilon +S_{Q^\epsilon}(\overline{D^\epsilon})+\overline{\Lambda^\epsilon} \cdot Q^\epsilon-Q^\epsilon\cdot \overline{\Lambda^\epsilon}]:\varphi dxdt\right | \nonumber\\
  &\leq & C(||v^\epsilon||_{L^4(\tilde{\Omega})}||\nabla Q^\epsilon||_{L^2}||\varphi||_{L^4} +||H^\epsilon||_{L^2}||\varphi||_{L^2}+||D^\epsilon||_{L^2}|| \varphi||_{L^2})\nonumber\\
  &&+C||D^\epsilon||_{L^2} (|| Q^\epsilon||_{L^8(\tilde{\Omega})}^2+|| Q^\epsilon||_{L^4(\tilde{\Omega})})||\varphi||_{L^4} \leq C||\varphi||_{L^4}, \label{Q*t}
\end{eqnarray}
%\begin{eqnarray*}
 %  && |<Q^\epsilon_t , \psi>| \\
  % &=& \left |  \int_0^T\int_{\mathbb{R}^2} [-v^\epsilon \cdot \nabla Q^\epsilon + \frac{1}{\Gamma} H^\epsilon +S_{Q^\epsilon}(D^\epsilon)+\Lambda^\epsilon \cdot Q^\epsilon-Q^\epsilon\cdot \Lambda^\epsilon]:\psi dxdt\right | \\
  %&\leq & C(||v^\epsilon||_{L^\infty_tL^2_x}||\nabla Q^\epsilon||_{L^2}||\psi||_{L^2_tL^\infty_x} +||H^\epsilon||_{L^2}||\psi||_{L^2}+||D^\epsilon||_{L^2}|| \psi||_{L^2})\\
  %&&+C||D^\epsilon||_{L^2} (|| Q^\epsilon||_{L^\infty_tL^4_x(\tilde{\Omega})}^2+|| Q^\epsilon||_{L^\infty_tL^2_x(\tilde{\Omega})})||\psi||_{L^2_tL^\infty_x} \leq C||\psi||_{L^2_tH^2_x},
%\end{eqnarray*}
where $\underline{v^\epsilon}$, $\overline{D^\epsilon}$ and $\overline{\Lambda^\epsilon}$ are given in (\ref{X-1}). This implies that $Q_t^\epsilon$ is uniformly bounded in $L^{\frac43}((0,T)\times\tilde{\Omega})$ for any smooth bounded domain $\tilde{\Omega}\subset \mathbb{R}^2$.
Then, it follows from the energy inequality (\ref{Q-energy}) and Aubin-Lious Lemma that there exists a subsequence of $\{(v^\epsilon,Q^\epsilon)\}_{\epsilon>0}$ such that
\begin{equation}\label{con4}
  v^\epsilon \rightarrow v^* \text{  in  } L^p(0,T; L^p(\tilde{\Omega})),\quad 2\leq p<4,
\end{equation}
\begin{equation}\label{con5}
  Q^\epsilon \rightarrow Q^* \text{  in  } L^q(0,T;L^q(\tilde{\Omega})), \quad 1\leq q<\infty,
\end{equation}
where $\tilde{\Omega}$ is any smooth bounded domain in $\mathbb{R}^2$.

%For any smooth $\psi\in C^\infty([0,T]\times \mathbb{R}^2,\mathbb{R}^2)$ with $\psi(T,\cdot)=0, \nabla \cdot \psi=0$,
%$$\int_0^T\int_{\mathbb{R}^2} v^\epsilon_t\cdot \psi dxdt = -\int_{\mathbb{R}^2} v^\epsilon_0\cdot \psi(0,x) dx-\int_0^T\int_{\mathbb{R}^2} v^\epsilon\cdot \psi_t dxdt,$$
%yields that
%$$\int_0^T\int_{\mathbb{R}^2} v^*_t\cdot \psi dxdt = -\int_{\mathbb{R}^2} v^*_0\cdot \psi(0,x) dx-\int_0^T\int_{\mathbb{R}^2} v^*\cdot \psi_t dxdt,$$
%by using (\ref{QSQ}), (\ref{con1}) and (\ref{v*t}).
%Similarly, for any smooth $\varphi\in C^\infty([0,T]\times \mathbb{R}^2,\mathbb{R}^2)$ with $\varphi(T,\cdot)=0$, one has
%\begin{equation}\label{Qin}
 % \int_0^T\int_{\mathbb{R}^2} Q^*_t:\varphi dxdt =-\int_{\mathbb{R}^2} Q^*_0:\varphi(0,x) dx -\int_0^T\int_{\mathbb{R}^2} Q^*:\varphi_t dxdt,
%\end{equation}
%according to (\ref{QSQ}), (\ref{con2}) and (\ref{Q*t}).

%For any smooth $\varphi\in C^\infty([0,T]\times \mathbb{R}^2,\mathbb{R}^2)$ with $\varphi(T,\cdot)=0$, it is follows from
%$$\int_0^T\int_{\mathbb{R}^2} Q^\epsilon_t:\varphi dxdt =-\int_{\mathbb{R}^2} Q^\epsilon_0:\varphi(0,x) dx -\int_0^T\int_{\mathbb{R}^2} Q^\epsilon:\varphi_t dxdt,$$
%for every $\epsilon>0$ that
%\begin{equation}\label{Qin}
 % \int_0^T\int_{\mathbb{R}^2} Q^*_t:\varphi dxdt =-\int_{\mathbb{R}^2} Q^*_0:\varphi(0,x) dx -\int_0^T\int_{\mathbb{R}^2} Q^*:\varphi_t dxdt,
%\end{equation}
%according to (\ref{QSQ}), (\ref{con2}) and (\ref{Q*t}). Similarly,

Then taking the limiting $\epsilon \rightarrow 0^+$ in the equality (\ref{Qw2}) for $(v^\epsilon, Q^\epsilon)$ yields:
\begin{eqnarray}
   && \int_0^T\int_{\mathbb{R}^2} [-Q^*:\varphi_t -(\underline{v^*}\cdot \nabla \varphi): Q^* -\frac{1}{\Gamma}H^*: \varphi  ]dxdt\nonumber\\
&+&\int_0^T\int_{\mathbb{R}^2} [Q^*\cdot \overline{\Lambda^*}-\overline{\Lambda^*}\cdot Q^*-S_{Q^*}(\overline{D^*})]:\varphi dxdt= \int_{\mathbb{R}^2} Q_0^*(x):\varphi(0,x) dx,  \label{Qstar2}
\end{eqnarray}
where $\overline{D^*},\,\overline{\Lambda^*},\, \underline{v^*}$ are given in (\ref{X-2}), $\varphi\in C^\infty_0((0,T)\times\mathbb{R}^2,\mathcal{Q}_0)$, and $S_{Q^*}(D^*)$ is given in (\ref{SQ-HD}). Note that condition (\ref{QSQ}) implies
$$\lim_{\epsilon\rightarrow 0^+}\int_{\mathbb{R}^2} Q_0^\epsilon:\varphi(0,x) dx =\int_{\mathbb{R}^2} Q_0^*:\varphi(0,x) dx, \quad \lim_{\epsilon\rightarrow 0^+}\int_{\mathbb{R}^2}v_0^\epsilon \cdot \psi(0,x) dx = \int_{\mathbb{R}^2}v_0^*\cdot \psi(0,x) dx.$$
Next, we show that one can pass limit in (\ref{Qw1}) to obtain
\begin{eqnarray}
    &&\int_0^T\int_{\mathbb{R}^2} [-v^*\cdot \psi_t -(v^*\cdot \overline{\nabla\psi})\cdot v^*+\eta \overline{D^*}:\overline{\nabla\psi}+ L_1\nabla Q^*\odot \nabla Q^*:\underline{\nabla \psi}] dxdt \nonumber \\
  &+&\int_0^T\int_{\mathbb{R}^2}[Q^* \cdot H^*-H^*\cdot Q^*-S_{Q^*}(H^*)]:\overline{\nabla \psi }dxdt = \int_{\mathbb{R}^2}v_0^*(x)\cdot \psi(0,x) dx,\label{Qstar1}
\end{eqnarray}
where $S_{Q^*}(H^*)$ is defined by (\ref{SQ-HD}). To this end, due to (\ref{Qw1}) for $(v^\epsilon,Q^\epsilon)$ and (\ref{con1})-(\ref{con5}), it suffices to show only that there exists a subsequence of $\{(v^\epsilon,Q^\epsilon)\}_{\epsilon>0}$ such that
\begin{equation}\label{claim0}
  \int_0^T\int_{\mathbb{R}^2}\nabla Q^\epsilon\odot \nabla Q^\epsilon:\underline{\nabla \psi} dxdt \rightarrow \int_0^T\int_{\mathbb{R}^2}\nabla Q^*\odot \nabla Q^*:\underline{\nabla \psi } dxdt,  \text{ as } \epsilon \rightarrow 0.
\end{equation}
Note that the energy inequality (\ref{Q-energy}) for $(v^\epsilon,Q^\epsilon)$ implies that there exists a subsequence of $\{(v^\epsilon,Q^\epsilon)\}_{\epsilon>0}$ such that
\begin{equation*}%\label{kappa}
  \nabla Q^\epsilon\odot \nabla Q^\epsilon:\underline{\nabla \psi} dxdt \rightharpoonup  \nabla Q^*\odot \nabla Q^*:\underline{\nabla \psi} dxdt +\kappa, \text{ as } \epsilon \rightarrow 0
\end{equation*}
for a possibly non-vanishing measure $\kappa$. Our aim is to show that there exists a subsequence of $\{(v^\epsilon,Q^\epsilon)\}_{\epsilon>0}$ such that $\kappa=0$, which implies (\ref{claim0}).
To this end, we observe that
$$\int_0^T \liminf_{\epsilon\rightarrow 0^+} \int_{\mathbb{R}^2} |H^\epsilon|^2dxdt\leq E_0$$
by the energy inequality (\ref{Q-energy}) and Fatou's lemma.
Set
$$\mathcal{L}_\infty=\left\{t: \liminf_{\epsilon\rightarrow 0^+}\int_{\mathbb{R}^2} |H^\epsilon(t,\cdot)|^2dx < \infty \right \}.$$
%$$\mathcal{M}_m=\left\{t: \liminf_{\epsilon\rightarrow 0^+}\int_{\mathbb{R}^2} |H^\epsilon(t,\cdot)|^2dx> m \right \}.$$
Then, it holds that
$$|\mathcal{L}_\infty|=T.$$
Meanwhile, we claim that there exists a set $\mathcal{A}_\infty\subset (0,T)$ such that $|\mathcal{A}_\infty|=T$ and
\begin{equation}\label{claimwired}
  \nabla Q^\epsilon(t) \rightharpoonup \nabla Q^*(t) \text{ in } L^2(\mathbb{R}^2)
\end{equation}
for any $t\in \mathcal{A}_\infty$. To prove this claim, one notes that (\ref{con2}) implies
\begin{equation}\label{add1}
  \lim_{\epsilon\rightarrow 0^+}\int_0^T \int_{\mathbb{R}^2} \partial _i Q^\epsilon : \hat{\varphi}(x) \hat{\phi}(t)dxdt = \int_0^T \int_{\mathbb{R}^2} \partial _i Q^* : \hat{\varphi}(x) \hat{\phi}(t)dxdt, \quad i=1,2,
\end{equation}
for any $\hat{\varphi}\in C_0^\infty(\mathbb{R}^2,\mathbb{M}^{3\times 3})$ and any $\hat{\phi}\in C_0^\infty((0,T))$. Since $||\nabla Q^\epsilon||_{L^\infty_tL^2_x}$ is uniformly bounded, (\ref{add1}) holds true for $\hat{\varphi} \in L^2(\mathbb{R}^2,\mathbb{M}^{3\times 3})$ and $\hat{\phi}\in L^1((0,T))$. For fixed $\hat{\varphi} \in L^2(\mathbb{R}^2,\mathbb{M}^{3\times 3})$, define
$$g_i^\epsilon(t)=\int_{\mathbb{R}^2} \partial_i Q^\epsilon(t,x) :\hat{\varphi}(x)dx \text{  and  } g_i^*(t)=\lim_{\epsilon\rightarrow 0^+} g_i^\epsilon (t), \, i=1,2.$$
Note that $g^*_1$ and $g^*_2$ can also be regarded as the weak limits of $g^\epsilon_1$ and $g^\epsilon_2$ in $L^p((0,T)), 1<p<\infty$ ($||g^\epsilon_1||_{L^p((0,T))}$ and $||g^\epsilon_2||_{L^p((0,T))}$ are uniformly bounded due to the uniformly bound for $||\nabla Q^\epsilon||_{L^\infty_tL^2_x}$), hence $g^*_1$ and $g^*_2$ are measurable. Then, taking $\hat{\phi}\equiv 1$ into (\ref{add1}) yields that
$$\lim_{\epsilon\rightarrow 0^+} \int_0^T g_i^\epsilon(t) dt =\int_0^T \lim_{\epsilon\rightarrow 0^+} g^\epsilon _i(t)dt = \int_0^T g_i^*(t) dt<\infty,\, i=1,2,$$
due to the Lebesgue dominated convergence theorem and the energy inequality (\ref{Q-energy}). Next, we prove that there exists a set $\mathcal{A}_0\subset (0,T)$ such that $|\mathcal{A}_0|=T$ and
\begin{equation}\label{add2}
  g^*_i(t)=\lim_{\epsilon \rightarrow 0^+}\int_{\mathbb{R}^2} \partial_i Q^\epsilon(t,x) :\hat{\varphi}(x)dx=\int_{\mathbb{R}^2} \partial_i Q^*(t,x):\hat{\varphi}(x)dx
\end{equation}
for any $t\in \mathcal{A}_0$. Since $\nabla Q^*\in L^\infty_tL^2_x$, one can define the following measurable functions: $$\tilde{\delta}_i(t)=g^*_i(t)-\int_{\mathbb{R}^2} \partial_i Q^*(t,x):\hat{\varphi}(x)dx$$
and
\begin{equation*}
\hat{\phi}_i(t)=\left\{
                \begin{array}{ll}
                  1, & \hbox{$\tilde{\delta}_i(t) >0$;} \\
                  -1, & \hbox{$\tilde{\delta}_i(t) <0$;} \\
                  0, & \hbox{others,}
                \end{array}
              \right.
\end{equation*}
$i=1,2$. Note that $\hat{\phi}_i\in L^1((0,T)), i=1,2$. Then, these and (\ref{add1}) yield that
$$\int_0^T \left| g^*_i(t)-\int_{\mathbb{R}^2} \partial_i Q^*(t,x):\hat{\varphi}(x)dx\right|dt=0, \, i=1,2,$$
which implies (\ref{add2}). Since $L^2(\mathbb{R}^2)$ is separable, one can find a countable set $\{\hat{\varphi}_j\}_{j=1}^\infty$ such that $\overline{\{\hat{\varphi}_j\}_{j=1}^\infty}=L^2(\mathbb{R}^2)$. For every fixed $\hat{\varphi}_j\in \{\hat{\varphi}_j\}_{j=1}^\infty$, as in (\ref{add2}), one can find a set $\mathcal{A}_j \subset (0,T)$ such that $|\mathcal{A}_j|=T$ and
$$\lim_{\epsilon \rightarrow 0^+} \int_{\mathbb{R}^2} \partial_i Q^\epsilon(t,x) :\hat{\varphi}_j(x) dx =\int_{\mathbb{R}^2} \partial_i Q^*(t,x) :\hat{\varphi}_j(x) dx,\, i=1,2,$$
for any $t\in \mathcal{A}_j$. Then, let $\mathcal{A}_\infty=\cap _{j=1}^\infty \mathcal{A}_j$, one has that
$$\lim_{\epsilon \rightarrow 0^+} \int_{\mathbb{R}^2} \partial_i Q^\epsilon(t,x) :\hat{\varphi}(x) dx =\int_{\mathbb{R}^2} \partial_i Q^*(t,x) :\hat{\varphi}(x) dx,\, i=1,2,$$
for any $t\in \mathcal{A}_\infty$ and any $\hat{\varphi} \in L^2(\mathbb{R}^2,\mathbb{M}^{3\times 3})$. Hence (\ref{claimwired}) is proved.

%$$\int_0^T\int_{\mathbb{R}^2}\nabla Q^*\odot \nabla Q^*:\nabla \psi dxdt=  \lim_{m\rightarrow \infty} \int_{\mathcal{L}_m} \lim_{\epsilon\rightarrow 0^+}\int_{\mathbb{R}^2}\nabla Q^\epsilon\odot \nabla Q^\epsilon:\nabla \psi dxdt.$$
Therefore, (\ref{claim0}) holds true provided that there exists a subsequence of $\{(v^\epsilon,Q^\epsilon)\}_{\epsilon>0}$ such that for $t\in \mathcal{L}_\infty\cap \mathcal{A}_\infty$,
\begin{equation}\label{claim1}
  \int_{\mathbb{R}^2} \nabla Q^\epsilon(t,\cdot) \odot\nabla Q^\epsilon(t,\cdot) :\underline{\nabla \psi} (t,\cdot)dx \rightarrow \int_{\mathbb{R}^2} \nabla Q^*(t,\cdot) \odot\nabla Q^*(t,\cdot) :\underline{\nabla \psi} (t,\cdot)dx
\end{equation}
as $\epsilon \rightarrow 0^+$. It should be noted that, for fixed $t\in \mathcal{L}_\infty\cap \mathcal{A}_\infty$, there exists a subsequence of $\{(v^\epsilon,Q^\epsilon)\}_{\epsilon>0}$ such that
$$\nabla Q^\epsilon(t,\cdot) \odot\nabla Q^\epsilon(t,\cdot) :\underline{\nabla \psi} (t,\cdot)dx \rightharpoonup \nabla Q^*(t,\cdot) \odot\nabla Q^*(t,\cdot) :\underline{\nabla \psi} (t,\cdot)dx +\kappa_1 \text{ as } \epsilon\rightarrow 0^+$$
for a possibly non-vanishing measure $\kappa_1$. For this subsequence $\{(v^\epsilon,Q^\epsilon)\}_{\epsilon>0}$, $\kappa_1=0$ if there exists a subsequence $\{(v^{\epsilon_j},Q^{\epsilon_j})\}_{j=1}^\infty$ of $\{(v^\epsilon,Q^\epsilon)\}_{\epsilon>0}$ such that
$$\int_{\mathbb{R}^2}\nabla Q^{\epsilon_j}(t,\cdot) \odot\nabla Q^{\epsilon_j}(t,\cdot) :\underline{\nabla \psi} (t,\cdot)dx \rightarrow \int_{\mathbb{R}^2}\nabla Q^*(t,\cdot) \odot\nabla Q^*(t,\cdot) :\underline{\nabla \psi} (t,\cdot)dx  \text{ as } \epsilon_j\rightarrow 0^+.$$

For each time $t_0\in \mathcal{L}_\infty\cap \mathcal{A}_\infty$, there exists a subsequence of $\{ H^{\epsilon}\}_{\epsilon>0}$ such that the $L^2$-norm of this subsequence is uniformly bounded, i.e. $||H^{\epsilon}||_{L^2(\mathbb{R}^2)} <\infty$. This and Lemma \ref{Lm} imply that the concentration point $y_0$ of $\nabla Q^{\epsilon} \odot \nabla Q^{\epsilon}$ at time $t_0$ must satisfy
\begin{equation}\label{y0}
  \liminf_{r\rightarrow0}\left\{\liminf _{\epsilon \rightarrow 0} \int_{B_r(y_0)}\left(|\nabla Q^{\epsilon}|^2 +\frac{\hat{F}_b(Q^\epsilon)}{\epsilon}\right)(t_0,x)dx\right\} \geq \delta_0.
\end{equation}
%$$\liminf_{r\rightarrow0}\left\{\liminf _{\epsilon \rightarrow 0} \int_{B_r(y_0)}\left(|\nabla Q^{\epsilon}(t_0,x)|^2 +\frac{\hat{F}_b(Q^\epsilon)}{\epsilon}\right)dx\right\} \geq \delta_0.$$
Then, this and the energy inequality (\ref{Q-energy}) imply that there exist at most finite points where $\nabla Q^{\epsilon} \odot \nabla Q^{\epsilon}$ may concentrate on and the set of these points is denoted as
\begin{eqnarray*}
 \mathcal{B}(t_0) &=& \bigcap_{r>0} \left\{ x: \liminf _{\epsilon \rightarrow 0} \int_{B_r(x)}\left(|\nabla Q^{\epsilon}|^2 +\frac{\hat{F}_b(Q^\epsilon)}{\epsilon}\right)(t_0,y)dy \geq \delta_0, t_0\in \mathcal{L}_\infty\cap \mathcal{A}_\infty \right\}\\
  &=&  \left\{x_1(t_0),x_2(t_0) , \cdots , x_{L(t_0)}(t_0): t_0\in \mathcal{L}_\infty\cap \mathcal{A}_\infty \right\},
\end{eqnarray*}
where $0\leq L(t_0)\leq[\frac{E_0}{\delta_0}] +1$ is an integer depending on $t_0$, and $E_0$ is given in (\ref{E01}).

Based on this fact, without loss of generality, one can assume that $\mathcal{B}(t)=\{\mathbf{0}\}, t\in \mathcal{L}_\infty\cap \mathcal{A}_\infty$, consists of a single point at the origin.
Since $\partial_1 \psi _1 +\partial_2\psi_2=0$, it is easy to see
$$\nabla Q^\epsilon\odot \nabla Q^\epsilon:\underline{\nabla \psi}=(\nabla Q^\epsilon\odot \nabla Q^\epsilon -\frac12 |\nabla Q^\epsilon|^2 I):\underline{\nabla \psi},$$
where $I$ is the identity matrix in $\mathbb{M}^{2\times 2}$.
Then, direct calculations give
\begin{equation*}
 \nabla Q^\epsilon\odot \nabla Q^\epsilon -\frac12 |\nabla Q^\epsilon|^2 I =\left(
                                                                              \begin{array}{cc}
                                                                                \frac12 (|\partial_1 Q^\epsilon|^2 -|\partial_2 Q^\epsilon |^2) &  \partial_1 Q^\epsilon:\partial_2 Q^\epsilon\\
                                                                                \partial_1 Q^\epsilon :\partial_2 Q^\epsilon &  -\frac12 (|\partial_1 Q^\epsilon|^2 -|\partial_2 Q^\epsilon |^2)\\
                                                                              \end{array}
                                                                            \right).
\end{equation*}
Therefore, it suffices to show that
\begin{equation}\label{q1}
  \int_{\mathbb{R}^2}(|\partial_1 Q^\epsilon|^2 -|\partial_2 Q^\epsilon |^2 )\phi dx \rightarrow   \int_{\mathbb{R}^2}(|\partial_1 Q^*|^2 -|\partial_2 Q^* |^2 )\phi dx
\end{equation}
and
\begin{equation}\label{q2}
   \int_{\mathbb{R}^2}\partial_1 Q^\epsilon :\partial_2 Q^\epsilon\phi dx\rightarrow \int_{\mathbb{R}^2}\partial_1 Q^*:\partial_2 Q^*\phi dx
\end{equation}
as $\epsilon \rightarrow 0$ for every $\phi\in C^\infty_0(\mathbb{R}^2,\mathbb{R})$.
We prove (\ref{q1}) and (\ref{q2}) by a Pohozaev type argument. First, we claim that there exists $R>0$ such that
\begin{equation}\label{No1}
  \int_{B_R(\mathbf{0})} \frac{\hat{F}_b(Q^\epsilon(t,x))}{\epsilon}dx \rightarrow 0 ,\quad t\in \mathcal{L}_\infty\cap \mathcal{A}_\infty, \text{ as } \epsilon \rightarrow 0.
\end{equation}
Note that this claim implies that
$$\int_{\mathbb{R}^2}\frac{\hat{F}_b(Q^\epsilon(t,x))}{\epsilon}dx \rightarrow 0,\quad t\in \mathcal{L}_\infty\cap \mathcal{A}_\infty, \text{ as } \epsilon \rightarrow 0$$
due to the assumption that $\mathcal{B}(t)=\{\mathbf{0}\}$ for $t\in \mathcal{L}_\infty\cap \mathcal{A}_\infty$. By Lemma \ref{Lm}, the small energy condition implies the strong convergence of $\frac{1}{\epsilon}\hat{F}_b(Q^\epsilon)$ in $L^1$ and $Q^\epsilon$ in $H^1$. Then, this and the energy inequality (\ref{Q-energy}) imply that $\frac{1}{\epsilon}\hat{F}_b(Q^\epsilon)$ converges strongly to $0$ in $L^1_{loc}(\mathbb{R}^2\setminus \mathcal{B}(t))$ for $t\in\mathcal{L}_\infty\cap \mathcal{A}_\infty$. Therefore, if (\ref{No1}) fails, one can assume that
\begin{equation}\label{ass1}
  \frac{\hat{F}_b(Q^\epsilon)}{\epsilon}dx \rightharpoonup \beta_1 \delta(\mathbf{0}), \quad \beta_1>0, \text{ as } \epsilon \rightarrow 0
\end{equation}
in $B_r(\mathbf{0})$ for any $r>0$, where $\delta(\textbf{0})$ is the dirac measure centered at the origin. Recall the definition of $H^\epsilon$:
\begin{equation}\label{H0}
  L_1 \Delta Q^\epsilon -\frac{\mathcal{J}(Q^\epsilon)}{\epsilon} =H^\epsilon.
\end{equation}
Multiplying (\ref{H0}) by $\frac{x}{r}\cdot \nabla Q^\epsilon$ and integrating over $B_r(\mathbf{0})$, one gets after integration by parts that
\begin{eqnarray*}
   && \int_{\partial B_r(\mathbf{0})} \left(L_1\left|\frac{\partial Q^\epsilon}{\partial \nu}\right|^2 -\frac{L_1}{2} |\nabla Q^\epsilon |^2 -\frac{\hat{F}_b(Q^\epsilon)}{\epsilon} \right) dS \\
 &=&  -\frac{1}{r}\int_{ B_r(\mathbf{0})}\left( 2\frac{\hat{F}_b(Q^\epsilon)}{\epsilon} -x\cdot\nabla Q^\epsilon :H^\epsilon\right) dx.
\end{eqnarray*}
Then, integrating above identity from $\tau$ to $R$ yields
\begin{eqnarray*}
   && 2\int_\tau^R \frac{1}{r}\int_{ B_r(\mathbf{0})}\frac{\hat{F}_b(Q^\epsilon)}{\epsilon} dx dr\\
   &=& -\int_{ B_R(\mathbf{0})\setminus B_\tau(\mathbf{0})} \left(L_1\left|\frac{\partial Q^\epsilon}{\partial \nu}\right|^2 -\frac{L_1}{2} |\nabla Q^\epsilon |^2 -\frac{\hat{F}_b(Q^\epsilon)}{\epsilon} \right)dx +\int_\tau^R \frac{1}{r}\int_{ B_r(\mathbf{0})}x\cdot\nabla Q^\epsilon :H^\epsilon dx dr.
\end{eqnarray*}
For $t\in \mathcal{L}_\infty\cap \mathcal{A}_\infty$, it holds that
$$\left |\frac{1}{r}\int_{ B_r(\mathbf{0})}x\cdot\nabla Q^\epsilon :H^\epsilon dx \right | \leq C ||\nabla Q^\epsilon||_{L^2}||H^\epsilon||_{L^2} \leq CE_0 $$
thanks to the energy inequality (\ref{Q-energy}) and the definition of $\mathcal{L}_\infty\cap \mathcal{A}_\infty$. Therefore,
\begin{equation}\label{cc1}
  \int_\tau^R \frac{1}{r}\int_{ B_r(\mathbf{0})}\frac{\hat{F}_b(Q^\epsilon)}{\epsilon} dxdr\leq 2E_0+ CE_0(R-\tau)
\end{equation}
for all $0<\tau<R<\infty$.

On the other hand, it follows from (\ref{Q-energy}) and (\ref{ass1}) that
$$\lim_{\epsilon\rightarrow 0}\int_\tau^R \frac{1}{r}\int_{ B_r(\mathbf{0})}\frac{\hat{F}_b(Q^\epsilon)}{\epsilon} dxdr= \int_\tau^R \frac{1}{r}\lim_{\epsilon\rightarrow 0}\int_{ B_r(\mathbf{0})}\frac{\hat{F}_b(Q^\epsilon)}{\epsilon} dxdr=\int_\tau^R \frac{\beta_1}{r}dr=\beta_1\ln{\left(\frac{R}{\tau}\right)}$$
by the Lebesgue dominated convergence theorem, which contradicts (\ref{cc1}) when $\tau$ is very small. This proves the claim.

We now verify (\ref{q1}). Otherwise, one can assume that there exists a real nonzero number $\beta_2$ such that
\begin{equation}\label{ass2}
  (|\partial_1 Q^\epsilon|^2 -|\partial_2 Q^\epsilon |^2)dx \rightharpoonup   (|\partial_1 Q^*|^2 -|\partial_2 Q^* |^2)dx+\beta_2\delta(\mathbf{0}) \text{ as } \epsilon \rightarrow 0
\end{equation}
in $B_r(\mathbf{0})$ for any $r>0$. Multiplying (\ref{H0}) by $\frac{x_1}{r}\partial_1 Q^\epsilon$ and integrating over $B_r(\mathbf{0})$, one can get through integration by parts that
\begin{eqnarray*}
   && \int_{\partial B_r(\mathbf{0})} \left(\frac{L_1x_1}{r}\partial_1 Q^\epsilon:\frac{\partial Q^\epsilon}{\partial \nu}-\frac{L_1}{2} \frac{x_1^2}{r^2}|\nabla Q^\epsilon|^2-\frac{x_1^2 \hat{F}_b(Q^\epsilon)}{r^2\epsilon} \right)dS \\
   &=&\frac{1}{r}  \int_{B_r(\mathbf{0})}\left[ \frac{L_1}{2}(|\partial_1  Q^\epsilon|^2 -|\partial_2 Q^\epsilon|^2) -\frac{\hat{F}_b(Q^\epsilon)}{\epsilon} +x_1\partial_1 Q^\epsilon :H^\epsilon \right].
\end{eqnarray*}
It follows from this and a similar way as for (\ref{cc1}) that
\begin{eqnarray}
   && \left| \lim_{\epsilon\rightarrow 0}\int_\tau^R \frac{L_1}{r}\int_{B_r(\mathbf{0})} (|\partial_1  Q^\epsilon|^2 -|\partial_2  Q^\epsilon|^2)dx dr\right| \nonumber\\
   &\leq& \int_\tau^R  \frac{1}{r}\lim_{\epsilon\rightarrow 0}\int_{B_r(\mathbf{0})} \frac{\hat{F}_b(Q^\epsilon)}{\epsilon}dxdr + 2E_0+CE_0(R-\tau)\leq E_0[2+C(R-\tau)] \label{cc2}
\end{eqnarray}
by the Lebesgue dominated convergence theorem, where $C_0$ is independent of $\epsilon$ and $\tau$.

On the other hand, it follows from the energy inequality (\ref{Q-energy}) and (\ref{ass2}) that
\begin{eqnarray*}
   &&\lim_{\epsilon\rightarrow 0}\int_\tau^R \frac{1}{r}\int_{ B_r(\mathbf{0})} (|\partial_1  Q^\epsilon|^2 -|\partial_2 Q^\epsilon|^2) dxdr=\int_\tau^R\frac{1}{r}\lim_{\epsilon\rightarrow 0} \int_{ B_r(\mathbf{0})} (|\partial_1  Q^\epsilon|^2 -|\partial_2 Q^\epsilon|^2) dxdr \\
  && \qquad\qquad\qquad\qquad \qquad\qquad\qquad\qquad=\int_\tau^R\frac{1}{r}\left[\int_{ B_r(\mathbf{0})} (|\partial_1  Q^*|^2 -|\partial_2 Q^*|^2)  dx +\beta_2\right]dr
\end{eqnarray*}
by the Lebesgue dominated convergence theorem. Then, let $r$ be small enough such that
$$\left|\int_{ B_r(\mathbf{0})} (|\partial_1  Q^*|^2 -|\partial_2 Q^*|^2)  dx \right| \leq \frac{|\beta_2|}{2}.$$
Therefore, one has
$$\left|\lim_{\epsilon\rightarrow 0}\int_\tau^R \frac{1}{r}\int_{ B_r(\mathbf{0})} (|\partial_1  Q^\epsilon|^2 -|\partial_2 Q^\epsilon|^2) dxdr \right| \geq \frac12\int_\tau^R \frac{|\beta_2|}{r}dr=\frac12|\beta_2|(\ln{R}-\ln{\tau})$$
for small $R$, which contradicts (\ref{cc2}) when $\tau$ is very small and $0<\tau<<R$. This proves (\ref{q1}).

Similarly, we can prove (\ref{q2}). Indeed, if (\ref{q2}) fails, one can assume that there exists a real nonzero number $\beta_3$ such that
\begin{equation}\label{ass3}
  \partial_1 Q^\epsilon :\partial_2 Q^\epsilon dx \rightharpoonup   \partial_1 Q^* :\partial_2 Q^* dx+\beta_3\delta(\mathbf{0}) \text{ as } \epsilon \rightarrow 0
\end{equation}
in $B_r(\mathbf{0})$ for any $r>0$. Multiplying (\ref{H0}) by $\frac{x_2}{r}\partial_1 Q^\epsilon$ and integrating over $B_r(\mathbf{0})$ yield
\begin{eqnarray*}
   && \int_{\partial B_r(\mathbf{0})} \left(\frac{L_1x_2}{r}\partial_1 Q^\epsilon:\frac{\partial Q^\epsilon}{\partial \nu}-\frac{L_1}{2} \frac{x_1x_2}{r^2}|\nabla Q^\epsilon|^2-\frac{x_1x_2 \hat{F}_b(Q^\epsilon)}{r^2\epsilon} \right)dS \\
   &=&\frac{1}{r}  \int_{B_r(\mathbf{0})}\left(L_1\partial_1  Q^\epsilon:\partial_2  Q^\epsilon +x_2\partial_1 Q^\epsilon :H^\epsilon \right)dx.
\end{eqnarray*}
As for the derivation of (\ref{cc1}), one can obtain that as $\epsilon\rightarrow 0$,
\begin{equation}\label{cc3}
  \left|\int_\tau^R\frac{L_1}{r}  \int_{B_r(\mathbf{0})}\partial_1  Q^\epsilon:\partial_2  Q^\epsilon dxdr\right|\leq 2E_0+CE_0(R-\tau).
\end{equation}
On the other hand, (\ref{Q-energy}) and (\ref{ass3}) imply that
\begin{eqnarray*}
    \lim_{\epsilon\rightarrow 0}\int_\tau^R \frac{1}{r}\int_{ B_r(\mathbf{0})} \partial_1  Q^\epsilon:\partial_2  Q^\epsilon dxdr&=&\int_\tau^R \frac{1}{r}\lim_{\epsilon\rightarrow 0}\int_{ B_r(\mathbf{0})} \partial_1  Q^\epsilon:\partial_2  Q^\epsilon dxdr \\
   &=&  \int_\tau^R \frac{1}{r}\left[\int_{ B_r(\mathbf{0})} \partial_1  Q^*:\partial_2  Q^* dx + \beta_3 \right]dr
\end{eqnarray*}
by the Lebesgue dominated convergence theorem. Then, let $r$ be small enough such that
$$\left|\int_{ B_r(\mathbf{0})} \partial_1  Q^*:\partial_2  Q^*\right|\leq \frac{|\beta_3|}{2}.$$
Therefore, one has
$$\left|\lim_{\epsilon\rightarrow 0}\int_\tau^R \frac{1}{r}\int_{ B_r(\mathbf{0})} \partial_1  Q^\epsilon:\partial_2  Q^\epsilon dxdr \right| \geq \frac12\int_\tau^R \frac{|\beta_3|}{r}dr=\frac12|\beta_3|(\ln{R}-\ln{\tau})$$
for small $R$, which contradicts (\ref{cc3}) when $\tau$ is very small and $0<\tau<<R$. This proves (\ref{q2}) and hence completes the proof of (\ref{claim1}).

\textbf{Step 2. We prove that $\bm{(v^*(\cdot,t),\nabla d^*(\cdot,t)) \rightarrow (v_0^*,\nabla d_0^*)}$ as $\bm{t\rightarrow 0}$  in $\bm{L^2(\mathbb{R}^2)}$.}
To this end, we prove first that
\begin{equation}\label{D0}
  \frac12||v^*(t,\cdot)||_{L^2(\mathbb{R}^2)}^2 +L_1s_+^2 ||\nabla d^*(t,\cdot)||_{L^2(\mathbb{R}^2)}^2 \leq \frac12||v^*_0||_{L^2(\mathbb{R}^2)}^2 +L_1s_+^2||\nabla d^*_0||_{L^2(\mathbb{R}^2)}^2\leq E_0
\end{equation}
for all $t\in [0,T]$, where $E_0$ is given in (\ref{E01}). The uniform estimates (\ref{v*t}) and (\ref{Q*t}) imply that $v^*_t\in L^2(0,T; H^{-3}(\tilde{\Omega}))$ and $ Q^*_t \in L^{\frac43}((0,T)\times\tilde{\Omega})$. Therefore, $v^*$ is weakly continuous on $H^3(\tilde{\Omega})$, thus
$$f_1(t) =\int_{\tilde{\Omega}} v^*(t,x)\cdot \tilde{\phi}(x) dx $$
is continuous in $[0,T]$ for any $\tilde{\phi} \in H^3(\tilde{\Omega})$. Similarly, $\nabla Q^*$ and $\nabla d^*$ are weakly continuous on $W^{1,4}(\tilde{\Omega})$. Note that $|Q^*_t|^2= 2s_+^2|d^*_t|^2$ and $|\nabla Q^*|^2= 2s_+^2 |\nabla d^*|^2$. Then, as shown in \cite[Chapter III, Lemma 1.4]{T01}, $v^*,\nabla Q^*$ and $ \nabla d^* \in L^\infty(0,T; L^2(\mathbb{R}^2)$ yield that
\begin{equation}\label{w-conti}
  v^*, \nabla Q^* \text{ and } \nabla d^* \text{ are weakly continuous on } L^2(\tilde{\Omega})
\end{equation}
for any smooth bounded domain $\tilde{\Omega}\subset \mathbb{R}^2$. Based on this fact, one has
\begin{equation*}
  ||v^*(t)||_{L^2(\tilde{\Omega})} \leq  ||v^*||_{L^\infty_tL^2_x},  \quad|| \nabla Q^*(t)||_{L^2(\tilde{\Omega})} \leq  ||\nabla Q^*||_{L^\infty_tL^2_x}
\end{equation*}
for all $t\in[0,T]$. Then, letting $\tilde{\Omega}=B_{R}$ and $R\rightarrow \infty$ yields
\begin{equation*}
  ||v^*(t)||_{L^2(\mathbb{R}^2)} \leq  ||v^*||_{L^\infty_tL^2_x},  \quad || \nabla Q^*(t)||_{L^2(\mathbb{R}^2)} \leq  ||\nabla Q^*||_{L^\infty_tL^2_x}
\end{equation*}
for all $t\in [0,T]$. Meanwhile, similar arguments as for (\ref{claimwired}) imply that
\begin{equation}\label{new-0}
  v^\epsilon(t) \rightharpoonup v^*(t) \text{ and } \nabla Q^\epsilon(t) \rightharpoonup \nabla Q^*(t) \text{  in  } L^2(\mathbb{R}^2) \text{  as  } \epsilon \rightarrow 0^+
\end{equation}
for any fixed $t\in \mathcal{L}_\infty\cap \mathcal{A}_\infty$. Then, it follows from the energy inequality (\ref{Q-energy}) and the lower semicontinuity that
$$\frac12||v^*(t)||_{L^2(\mathbb{R}^2)}^2 +\frac{L_1}{2} ||\nabla Q^*(t)||_{L^2(\mathbb{R}^2)}^2 \leq \frac12||v^*_0||_{L^2(\mathbb{R}^2)}^2 +\frac{L_1}{2} ||\nabla Q^*_0||_{L^2(\mathbb{R}^2)}^2$$
for any fixed $t\in \mathcal{L}_\infty\cap \mathcal{A}_\infty$. Since $|\mathcal{L}_\infty\cap \mathcal{A}_\infty|=T$, one has that
$$\frac12||v^*||_{L^\infty_tL^2_x}^2 +\frac{L_1}{2} ||\nabla Q^*||_{L^\infty_tL^2_x}^2 \leq \frac12||v^*_0||_{L^2(\mathbb{R}^2)}^2 +\frac{L_1}{2} ||\nabla Q^*_0||_{L^2(\mathbb{R}^2)}^2.$$
Therefore, one gets the desired (\ref{D0}).

Next, We show that $v^*$ and $\nabla d^*$ are weakly continuous to the initial data $v_0^*$ and $\nabla d_0^*$ as $t\rightarrow 0$ on $L^2(\tilde{\Omega})$.  Let
$$\tilde{\varphi}(t)=\left\{
                       \begin{array}{ll}
                         e^{\frac{1}{|t|^2-1}}, & \hbox{$|t|<1$;} \\
                         0, & \hbox{$|t|\geq 1$,}
                       \end{array}
                     \right.
$$
and $\overline{\varphi}=\tilde{\varphi}/\int_{\mathbb{R}} \tilde{\varphi}(t) dt$. Then, it is easy to see that $\overline{\varphi}$ is smooth and $\int_{\mathbb{R}} \overline{\varphi}(t) dt=1$. For $t_2\in (0,T)$ and $\tilde{\phi}\in \mathcal{D}$, set
$$\psi_{\tilde{\epsilon}}(t,x)=\tilde{\phi}(x) \left(\int_0^t\frac{1}{\tilde{\epsilon}}\overline{\varphi}(\frac{t_2-\tau}{\tilde{\epsilon}})d\tau-1\right).
$$
Note that
$|\int_0^t\frac{1}{\tilde{\epsilon}}\overline{\varphi}(\frac{t_2-\tau}{\tilde{\epsilon}})d\tau-1|\leq 1$ for all $t\geq0$ and
$$\lim_{\tilde{\epsilon}\rightarrow 0^+} \int_0^t\frac{1}{\tilde{\epsilon}}\overline{\varphi}(\frac{t_2-\tau}{\tilde{\epsilon}})d\tau-1= \left\{
                                            \begin{array}{ll}
                                              0, & \hbox{$t>t_2$;} \\
                                              -1 , & \hbox{$t<t_2$.}
                                            \end{array}
                                          \right.
$$
Then, it follows from (\ref{w-conti}) that
$$\lim_{\tilde{\epsilon}\rightarrow 0^+}\int_0^T\int_{\mathbb{R}^2} v^*(t,x)\cdot  \partial_t\psi_{\tilde{\epsilon}}(t,x) = \lim_{\tilde{\epsilon}\rightarrow 0^+}\int_0^T\frac{1}{\tilde{\epsilon}}\overline{\varphi}(\frac{t_2-t}{\tilde{\epsilon}})\int_{\mathbb{R}^2} v^*(t,x)\cdot  \tilde{\phi}(x)=\int_{\mathbb{R}^2} v^*(t_2,x)\cdot \tilde{\phi}(x)  .$$
Using this, taking $\psi_{\tilde{\epsilon}}$ into (\ref{Qstar1}) and sending $\tilde{\epsilon} \rightarrow 0^+$ yield that
\begin{eqnarray}
    &&  -\int_0^{t_2}\int_{\mathbb{R}^2} [ -(v^*\cdot \overline{\nabla\tilde{\phi}})\cdot v^*+\eta\overline{ D^*}:\overline{\nabla\tilde{\phi}}+ L_1\nabla Q^*\odot \nabla Q^*:\underline{\nabla \tilde{\phi}}] dxdt \nonumber \\
  &&-\int_0^{t_2}\int_{\mathbb{R}^2}[Q^* \cdot H^*-H^*\cdot Q^*-S_{Q^*}(H^*)]:\overline{\nabla \tilde{\phi}} dxdt\nonumber\\
 &=& \int_{\mathbb{R}^2} v^*(t_2)\cdot \tilde{\phi}(x)  dx- \int_{\mathbb{R}^2}v_0^*\cdot \tilde{\phi}(x) dx\nonumber
\end{eqnarray}
by (\ref{con1})-(\ref{con2}) and the Lebesgue dominated convergence theorem. Therefore, one has that
$$\lim_{t_2\rightarrow 0}\int_{\mathbb{R}^2} v^*(t_2,x)\cdot \tilde{\phi}(x)  dx = \int_{\mathbb{R}^2}v_0^*(x)\cdot \tilde{\phi}(x) dx,$$
which implies that
\begin{equation}\label{wtov}
  v^* \text{ is weakly continuous on } L^2(\tilde{\Omega}) \text{ to the initial data } v_0^*
\end{equation}
by (\ref{w-conti}). Similarly, for $\tilde{\zeta}\in C^\infty_0(\mathbb{R}^2, \mathcal{Q}_0)$, taking $\varphi_{\tilde{\epsilon}}(t,x)=\tilde{\zeta}(x) (\int_0^t\frac{1}{\tilde{\epsilon}}\overline{\varphi}(\frac{t_2-\tau}{\tilde{\epsilon}})d\tau-1)$ into (\ref{Qstar2}) and sending $\tilde{\epsilon} \rightarrow 0^+$, one can get
\begin{eqnarray}
   && \int_0^{t_2}\int_{\mathbb{R}^2} [(\underline{v^*}\cdot \nabla \tilde{\zeta}): Q^* +\frac{1}{\Gamma}H^*: \tilde{\zeta} ]-\int_0^{t_2}\int_{\mathbb{R}^2} [Q^*\cdot \overline{\Lambda^*}-\overline{\Lambda^*}\cdot Q^*-S_{Q^*}(\overline{D^*})]:\tilde{\zeta} \nonumber\\
&=& \int_{\mathbb{R}^2} Q^*(t_2,x):\tilde{\zeta}(x)dx-\int_{\mathbb{R}^2} Q_0^*:\tilde{\zeta}(x) dx  \nonumber
\end{eqnarray}
by (\ref{w-conti}), (\ref{con1})-(\ref{con2}) and the Lebesgue dominated convergence theorem. Therefore,
\begin{equation}\label{yaom}
  \lim_{t_2\rightarrow 0}\int_{\mathbb{R}^2} Q^*(t_2,x):\tilde{\zeta}(x)  dx = \int_{\mathbb{R}^2}Q_0^*: \tilde{\zeta}(x) dx.
\end{equation}
Note that (\ref{w-conti}) implies that $d^*$ is weakly continuous on $L^2(\tilde{\Omega})$. Due to $\nabla Q^* \in L^\infty_tL^2_x$, one can take $\tilde{\zeta}=(d_0^*\otimes d_0^*-\mathbb{I}) \hat{\zeta}$ with $\hat{\zeta}\in C^\infty_0(\mathbb{R}^2, \mathbb{R}$) into (\ref{yaom}), then
$$\lim_{t_2\rightarrow 0}\int_{\mathbb{R}^2} d^*(t_2,x)\cdot d_0^*(x) \hat{\zeta} dx = \int_{\mathbb{R}^2} d^*(0,x)\cdot d_0^*(x) \hat{\zeta} dx=\int_{\mathbb{R}^2} \hat{\zeta} dx,$$
which implies that $d^*$ is weakly continuous on $L^2(\tilde{\Omega})$ to the initial data $d_0^*$. Then, this together with (\ref{w-conti}), implies that
\begin{equation}\label{wtod}
  \nabla d^* \text{ is weakly continuous on } L^2(\tilde{\Omega}) \text{ to the initial data } \nabla d_0^*.
\end{equation}

Finally we show that $(v^*(\cdot,t),\nabla d^*(t,\cdot)) \rightarrow (v_0^*,\nabla d_0^*)$ as $t\rightarrow 0$  in $L^2(\mathbb{R}^2)$. If not, there exist a real number $\delta_4>0$ and a subsequence of $\{(v^*(t,\cdot), d^*(t,\cdot))\}_{t>0}$ such that
\begin{equation*}\label{ee}
  \lim_{t\rightarrow 0^+} \left(\frac12||v^*(t,\cdot)-v^*_0||_{L^2(\mathbb{R}^2)}^2 +L_1s_+^2||\nabla d^*(t,\cdot)-\nabla d^*_0||_{L^2(\mathbb{R}^2)}^2\right) \geq \delta_4>0,
\end{equation*}
which yields
\begin{eqnarray}
   && \lim_{t\rightarrow 0^+} \left[\int_{\mathbb{R}^2} (\frac12|v^*|^2 +L_1s_+^2|\nabla d^*|)dx -\int_{\mathbb{R}^2} (v^*\cdot v^*_0 +2L_1s_+^2\nabla d^*:\nabla d^*_0)dx \right] \nonumber\\
  &\geq&  \delta_4-\int_{\mathbb{R}^2} (\frac12|v^*_0|^2 +L_1s_+^2|\nabla d^*_0|)dx.\label{ee1}
\end{eqnarray}
It follows from (\ref{wtov}) and (\ref{wtod}) that
\begin{equation}\label{ee2}
  \lim_{t\rightarrow 0^+} \int_{\tilde{\Omega}} (v^*\cdot v^*_0 +2L_1s_+^2\nabla d^*:\nabla d^*_0)dx =\int_{\tilde{\Omega}} (|v^*_0|^2 +2L_1s_+^2|\nabla d^*_0|^2)dx
\end{equation}
for any smooth bounded domain $\tilde{\Omega}\subset \mathbb{R}^2$. Choose suitable big $\tilde{\Omega}$ such that
$$\int_{\mathbb{R}^2\setminus\tilde{\Omega}} (|v^*_0|^2 +2L_1s_+^2|\nabla d^*_0|^2)dx \leq \min\left\{ \frac{\delta_4^2}{32E_0}, \frac{\delta_4}{4}\right\},$$
where $E_0$ is given in (\ref{E01}). Based on this fact, it follows from (\ref{D0}) that
\begin{equation*}\label{D5}
  \left|\lim_{t\rightarrow 0^+} \int_{\mathbb{R}^2\setminus \tilde{\Omega}} (v^*\cdot v^*_0 +2L_1s_+^2\nabla d^*:\nabla d^*_0)dx\right| \leq \frac{\delta_4}{2}.
\end{equation*}
Then, substituting the above estimate and (\ref{ee2}) into (\ref{ee1}), one has
\begin{eqnarray*}
   && \lim_{t\rightarrow 0^+} \int_{\mathbb{R}^2} (\frac12|v^*|^2 +L_1s_+^2|\nabla d^*|)dx \geq  \int_{\mathbb{R}^2} (\frac12|v^*_0|^2 +L_1s_+^2|\nabla d^*_0|)+ \delta_4 \\
   && \quad - \int_{\mathbb{R}^2\setminus \tilde{\Omega}} (| v^*_0|^2 +2L_1s_+^2|\nabla d^*_0|^2)dx +\lim_{t\rightarrow 0^+} \int_{\mathbb{R}^2\setminus \tilde{\Omega}} (v^*\cdot v^*_0 +2L_1s_+^2\nabla d^*:\nabla d^*_0)dx \\
   && \qquad \qquad\geq \int_{\mathbb{R}^2} (\frac12|v^*_0|^2 +L_1s_+^2|\nabla d^*_0|)+ \frac{\delta_4}{4},
\end{eqnarray*}
which contradicts (\ref{D0}). Hence the desired conclusion follows.

\textbf{Step 3. We prove that the limit $\bm{(v^*,d^*)}$ satisfies the equalities (\ref{ELw1}) and (\ref{ELw2}).} We show first that $(v^*,d^*)$ satisfies the equalities (\ref{ELw2}). To this end, one can show first that
\begin{equation}\label{Qin}
  \int_0^T\int_{\mathbb{R}^2} d^*_t\cdot\zeta dxdt =-\int_{\mathbb{R}^2} d^*_0\cdot\zeta(0,\cdot) dx -\int_0^T\int_{\mathbb{R}^2} d^*\cdot\zeta_t dxdt.
\end{equation}
for any $\zeta\in C_0^\infty([0,T)\times \mathbb{R}^2,\mathbb{R}^3)$.
Note that $|\nabla Q^*|^2=2s_+^2|\nabla d^*|^2$ and $|Q^*_t|^2=2s_+^2 |d^*_t|^2$. Then, it follows from $\nabla Q^*\in L^\infty_tL^2_x$ and $Q^*_t\in L^{\frac43}((0,T)\times\tilde{\Omega})$ that $\nabla d^*\in L^\infty_tL^2_x$ and $d^*_t\in L^{\frac43}((0,T)\times\tilde{\Omega})$ for any smooth bounded domain $\tilde{\Omega}\subset \mathbb{R}^2$. By \textbf{Step 2} in this section, $d^*$ is weakly continuous on $L^2(\tilde{\Omega})$ to the initial data $d_0^*$. Therefore, (\ref{Qin}) holds.

It follows from (\ref{con1})-(\ref{con2}) and (\ref{Q*t}) that (\ref{Qstar2}) holds for $ \varphi \in L^4((0,T)\times\tilde{\Omega},\mathcal{Q}_0)$ with any smooth bounded domain $\tilde{\Omega}\subset \mathbb{R}^2$. Therefore, one can take $\varphi=d^*\otimes [\zeta -(d^* \cdot\zeta)d^* ]+[\zeta -(d^* \cdot\zeta)d^*]\otimes d^*$ in (\ref{Qstar2}) due to $d^*\in \mathbb{S}^2$ and $\zeta\in C^\infty_0([0,T)\times \mathbb{R}^2,\mathbb{R}^3)$. Then, one has
\begin{equation}\label{w-1}
  \int_0^T\int_{\mathbb{R}^2} Q^*:\varphi_t dxdt=-2s_+\int_0^T\int_{\mathbb{R}^2} d^*_t\cdot\zeta dxdt=2s_+\int_0^T\int_{\mathbb{R}^2} d^*\cdot\zeta_t dxdt+2s_+\int_{\mathbb{R}^2} d_0^*\cdot \zeta(0,\cdot)dx
\end{equation}
by using (\ref{Qin}). Note that
\begin{equation}\label{w-in}
  \int_{\mathbb{R}^2}Q_0^*:\varphi(0,\cdot) dx= 2\int_{\mathbb{R}^2}Q_0^*:\{d^*_0\otimes [\zeta(0,\cdot) -(d_0^* \cdot\zeta(0,\cdot))d_0^* ]\}dx=0.
\end{equation}
Since $\partial_1 v^\epsilon_1 +\partial _2 v^\epsilon_2=0$ for all $\epsilon>0$, therefore the limit $v^*$ of $v^\epsilon$ satisfies $\partial_1 v^*_1 +\partial _2 v^*_2=0$. Then, direct calculations yield that
\begin{equation}\label{w-2}
  \int_0^T\int_{\mathbb{R}^2} (\underline{v^*}\cdot\nabla \varphi): Q^*dxdt =-2s_+\int_0^T\int_{\mathbb{R}^2} (\underline{v^*}\cdot\nabla d^*)\cdot \zeta dxdt=2s_+\int_0^T\int_{\mathbb{R}^2}(\underline{v^*}\cdot \nabla \zeta) \cdot d^* dxdt,
\end{equation}
\begin{equation}\label{w-3}
 \int_0^T\int_{\mathbb{R}^2} (Q^*\cdot \overline{\Lambda^*}-\overline{\Lambda^*}\cdot Q^*):\varphi dxdt =-2s_+\int_0^T\int_{\mathbb{R}^2}  (\overline{\Lambda^*} \cdot d^*)\cdot \zeta dxdt,
\end{equation}
\begin{equation}\label{w-4}
   \int_0^T\int_{\mathbb{R}^2} S_{Q^*}(\overline{D^*}) :\varphi dxdt=\frac{2\xi(s_+ +2)}{3} [\overline{D^*}\cdot d^*-(\overline{D^*}:d^*\otimes d^*)d^* ]\cdot \zeta dxdt.
\end{equation}

It remains to calculate $\int_0^T\int_{\mathbb{R}^2} H^*:\varphi dxdt$. For any fixed $t\in \mathcal{L}_\infty\cap \mathcal{A}_\infty$, recall the definition of $\mathcal{B}(t)$ in \textbf{Step 1} in this section. For each $x_i(t) \in \mathcal{B}(t), 1\leq i\leq L(t)$, define a smooth function
\begin{equation*}\label{ssz}
 \chi_{r(t)}^{x_i(t)}(y)= \left\{
                             \begin{array}{ll}
                               1, & \hbox{$y\in B_{r(t)}(x_i(t))$ ;} \\
                               0, & \hbox{$y \notin B_{2r(t)}(x_i(t))$,}
                             \end{array}
                           \right.
\end{equation*}
with $0\leq \chi_{r(t)}^{x_i(t)} \leq 1, |\nabla \chi_{r(t)}^{x_i(t)}| \leq \frac{C}{r(t)}$. Then, one can choose $r(t)$ small enough such that $B_{2r(t)}(x_1(t))$, $\cdots$, $B_{2r(t)}(x_{L(t)}(t))$ are disjoint to each other. Then, it follows from $\varphi\in T_{Q^*}\mathcal{N}$ and Lemma \ref{Lm} that
\begin{eqnarray}
   &&  \int_{\mathbb{R}^2} (H^*: \varphi)(t,x) dx=\lim_{r(t)\rightarrow 0}\int_{\mathbb{R}^2} H^*(t,x): \varphi(t,x) \left(1-\sum_{i=1}^{L(t)} \chi_{r(t)}^{x_i(t)}\right) dx \nonumber\\
&=&\lim_{r(t)\rightarrow 0}\int_{\mathbb{R}^2}L_1 \Delta Q^*(t,x): \varphi(t,x) \left(1-\sum_{i=1}^{L(t)} \chi_{r(t)}^{x_i(t)}\right) dx\nonumber\\
&=&-2L_1s_+ \int_{\mathbb{R}^2} (\partial_k d^* \cdot\partial_k \zeta -|\nabla d^*|^2 d^*\cdot \zeta ) (t,x)dx, \label{w-5}
\end{eqnarray}
where one has used the geometric structure of $\mathcal{J}^*(t,x)$ for $x\in \mathbb{R}^2\setminus \mathcal{B}(t), t\in \mathcal{L}_\infty\cap \mathcal{A}_\infty$. Therefore, it follows from (\ref{Qstar2}) and (\ref{w-1})-(\ref{w-5}) that (\ref{ELw2}) holds true.

Next, we prove that $(v^*,d^*)$ satisfies the equality (\ref{ELw1}). By direct calculations, one has
\begin{equation}\label{Astar}
  \int_0^T\int_{\mathbb{R}^2}\nabla Q^*\odot \nabla Q^* : \underline{\nabla \psi} dxdt=2s_+^2\int_0^T\int_{\mathbb{R}^2}(\nabla d^* \odot \nabla d^*):\underline{\nabla \psi} dxdt
\end{equation}
and
\begin{eqnarray}
   &&[ Q^* \cdot H^*-H^*\cdot Q^*-S_{Q^*}(H^*)]:\overline{\nabla \psi}\nonumber \\
  &=&[(1-\xi)Q^* \cdot H^*-(1+\xi)H^*\cdot Q^* -\frac23 \xi H^* +2\xi (Q^*:H^*)(Q^*+\frac13 \mathbb{I})]:\overline{\nabla \psi}. \label{vq-1}
\end{eqnarray}
It follows from (\ref{con1})-(\ref{con2}), (\ref{Q*t}), (\ref{w-conti}) and (\ref{yaom}) that the equality (\ref{Qstar2}) can be rewritten as
\begin{equation}\label{Q*new}
  \int_0^T\int_{\mathbb{R}^2} H^*:\varphi dxdt=\Gamma\int_0^T\int_{\mathbb{R}^2}[Q^*_t +\underline{v^*}\cdot\nabla Q^* +Q^*\cdot \overline{\Lambda^*}-\overline{\Lambda^*}\cdot Q^*-S_{Q^*}(\overline{D^*}) ]:\varphi dxdt,
\end{equation}
%Then, (\ref{con1})-(\ref{Q*t}) imply that $\varphi$ can be choose in $L^4(\tilde{\Omega})$ for any smooth bounded domain $\tilde{\Omega} \subset \mathbb{R}^2$.
where $\varphi\in L^4((0,T)\times \tilde{\Omega},\mathcal{Q}_0)$ for any smooth bounded domain $\tilde{\Omega}\subset \mathbb{R}^2$. Since $\nabla Q^*\in L^\infty_tL^2_x$ and $\psi\in C_0^\infty([0,T)\times\mathbb{R}^2,\mathbb{R}^3),\partial_1\psi_1 +\partial_2\psi_2=0$, it can be checked that the test function $\varphi$ can be taken as $\varphi= \frac12[Q^*\cdot \overline{\nabla \psi}+(\overline{\nabla \psi})^T\cdot Q^* -2(Q^*:\overline{\nabla \psi}) \mathbb{I}]$. Then, one can get that
\begin{eqnarray}
   &&\int_0^T\int_{\mathbb{R}^2} (Q^* \cdot H^*): \overline{\nabla \psi} dxdt = \frac12\int_0^T\int_{\mathbb{R}^2}  H^*: [Q^*\cdot \overline{\nabla \psi}+(\overline{\nabla \psi})^T\cdot Q^* -2(Q^*:\overline{\nabla \psi}) \mathbb{I}]dxdt \nonumber \\
   && =\Gamma s_+^2 \int_0^T\int_{\mathbb{R}^2}(\frac23 d^*\otimes \overline{N^*} -\frac13 \overline{N^*}\otimes d^*)  :\overline{\nabla \psi} dxdt\nonumber\\
   && \quad-\Gamma \xi\int_0^T\int_{\mathbb{R}^2} \left[ \frac{s_+^2(1-2s_+)}{3} (\overline{D^*}: d^*\otimes d^*)d^*\otimes d^* -\frac{2 s_+(1-s_+)}{9}\overline{D^*} \right]: \overline{\nabla \psi} dxdt \nonumber\\
   &&\quad-\Gamma \xi \int_0^T\int_{\mathbb{R}^2} \left(\frac{2s_+}{3} d^*\otimes d^*\cdot \overline{D^*} - \frac{s_+^2}{3}\overline{D^*}\cdot d^*\otimes d^* \right): \overline{\nabla \psi} dxdt. \label{vq-2}
\end{eqnarray}
where $\overline{N^*}=d_t^*+\underline{v^*}\cdot\nabla d^* -\overline{\Lambda^*}\cdot d^*$, and one has used the fact that $Q^*,H^*$ and $\overline{D^*}$ are symmetric with zero traces. Similarly, choosing $\varphi=\frac12[\overline{\nabla \psi} \cdot Q^*+Q^* \cdot (\overline{\nabla \psi})^T -2(Q^*:\overline{\nabla \psi})\mathbb{I}]$ leads to
\begin{eqnarray}
   &&\int_0^T\int_{\mathbb{R}^2} ( H^*\cdot Q^*): \overline{\nabla \psi} dxdt = \Gamma s_+^2 \int_0^T\int_{\mathbb{R}^2}(-\frac13 d^*\otimes \overline{N^*} +\frac23 \overline{N^*}\otimes d^*)  :\overline{\nabla \psi} dxdt \nonumber \\
   && -\Gamma \xi\int_0^T\int_{\mathbb{R}^2} \left[ \frac{s_+^2(1-2s_+)}{3} (\overline{D^*}: d^*\otimes d^*)d^*\otimes d^* -\frac{2 s_+(1-s_+)}{9}\overline{D^*} \right]: \overline{\nabla \psi} dxdt \nonumber\\
   &&-\Gamma \xi \int_0^T\int_{\mathbb{R}^2} \left(-\frac{s_+^2}{3} d^*\otimes d^*\cdot \overline{D^*} +\frac{2s_+}{3} \overline{D^*}\cdot d^*\otimes d^* \right): \overline{\nabla \psi} dxdt. \label{vq-3}
\end{eqnarray}
Next, choosing $\varphi=\frac12[\overline{\nabla \psi}+(\overline{\nabla \psi})^T]$, one can get
\begin{eqnarray}
   &&\int_0^T\int_{\mathbb{R}^2} H^*: \overline{\nabla \psi} dxdt = \Gamma s_+ \int_0^T\int_{\mathbb{R}^2}( d^*\otimes \overline{N^*}+ \overline{N^*}\otimes d^*)  :\overline{\nabla \psi} dxdt \nonumber \\
   && -\Gamma \xi\int_0^T\int_{\mathbb{R}^2} \left[ -2s_+^2 (\overline{D^*}: d^*\otimes d^*)d^*\otimes d^* +\frac{2(1-s_+)}{3}\overline{D^*} \right]: \overline{\nabla \psi} dxdt \nonumber\\
   &&-\Gamma \xi s_+\int_0^T\int_{\mathbb{R}^2} \left( d^*\otimes d^*\cdot \overline{D^*} + \overline{D^*}\cdot d^*\otimes d^* \right): \overline{\nabla \psi} dxdt. \label{vq-4}
\end{eqnarray}
Finally, choosing $\varphi=Q^*(Q^*+\frac13 \mathbb{I}) :\overline{\nabla \psi} $, one can get
\begin{eqnarray}
  &&  \int_0^T\int_{\mathbb{R}^2} (H^*:Q^*)(Q^*+\frac13 \mathbb{I}):\overline{\nabla \psi} dxdt \nonumber \\
  &=& -2\Gamma \xi \frac{s_+^2+s_+^3-2s_+^4}{3}\int_0^T\int_{\mathbb{R}^2} (\overline{D^*}:d^*\otimes d^*)d^*\otimes d^*: \overline{\nabla \psi} dxdt\label{vq-5}
\end{eqnarray}
Then, substituting (\ref{Astar})-(\ref{vq-5}) into the equality (\ref{Qstar1}), we obtain the equality (\ref{ELw1}).

\textbf{Step 4. This step is aim to prove that $\bm{(v^*,d^*)}$ satisfies the energy inequality (\ref{d-energy}).}
By the definition of $\mathcal{L}_\infty\cap \mathcal{A}_\infty$ in the \textbf{Step 1} in this section, it follows from the energy inequality (\ref{Q-energy}), (\ref{new-0}) and the lower semicontinuity that
\begin{eqnarray}
   && \int_{\mathbb{R}^2} \left(\frac12 |v^*|^2 +\frac{L_1}{2} |\nabla Q^*|^2 \right)(\cdot,t)dx +\int_0^t\int_{\mathbb{R}^2} \left(\eta|\overline{D^*}|^2 +\frac{1}{\Gamma}|H^*|^2\right)dxdt \nonumber\\
  &\leq& \int_{\mathbb{R}^2} \left(\frac12 |v^*_0|^2 +\frac{L_1}{2} |\nabla Q^*_0|^2 \right)dx \label{Q-star-energy}
\end{eqnarray}
for any $t\in  \mathcal{L}_\infty\cap \mathcal{A}_\infty$. The remaining task is to show that (\ref{d-energy}) and (\ref{Q-star-energy}) are equivalent for any $t\in  \mathcal{L}_\infty\cap \mathcal{A}_\infty$. To this end, due to $Q^*=s_+(d^*\otimes d^* -\frac13 \mathbb{I})$, it suffices to show that
\begin{eqnarray}
 &&\int_{\mathbb{R}^2}|H^*(t,x)|^2dx = \int_{\mathbb{R}^2}\left[\Gamma(\alpha_1+\frac{\gamma_2^2}{\gamma_1})|\overline{D^*}:d^*\otimes d^*|^2 +\Gamma(\alpha_5+\alpha_6-\frac{\gamma_2^2}{\gamma_1}) |\overline{D^*}\cdot d^*|^2 \right](t,x)dx\nonumber\\
   && \qquad\qquad +\int_{\mathbb{R}^2}\left[\frac{4\Gamma^2\xi^2(1-s_+)^2}{9}|\overline{D^*}|^2+2L_1^2s_+^2 |\Delta d^* +|\nabla d^*|^2d^*|^2\right](t,x)dx \label{last}
\end{eqnarray}
for any $t\in \mathcal{L}_\infty\cap \mathcal{A}_\infty$.

For fixed $t\in  \mathcal{L}_\infty\cap \mathcal{A}_\infty$ and $x\in\mathbb{R}^2\setminus\mathcal{B}(t)$, Lemma \ref{Lm} implies that
$$\mathcal{J}^*(x,t) = a_3(x,t) e_3 + a_4(x,t) e_4 +a_5(x,t) e_5,$$
where
$$e_3=\frac{1}{\sqrt{2}}(d_2\otimes d_1+d_1\otimes d_2),e_4=\frac{1}{\sqrt{2}}(d_1\otimes d_1-d_2\otimes d_2),$$
$$ e_5=\sqrt{6} (\frac12 d_1\otimes d_1 +\frac12 d_2\otimes d_2-\frac13 \mathbb{I}), $$
and $d_1,d_2\in \mathbb{S}^2, d_1\cdot d_2 =0, d_1\cdot d^*=0, d_2\cdot d^*=0$. Since $d_1,d_2$ and $d^*$ form an orthonormal basis to $\mathbb{R}^3$, one has that
\begin{equation}\label{d-1}
|\partial_k d^*|^2=|\partial_k d^*\cdot d_1|^2 +|\partial_k d^*\cdot d_2|^2,\quad k=1,2,
\end{equation}
\begin{equation}\label{d-2}
  |\overline{D^*}\cdot d_1|^2 =(\overline{D^*}:d_1\otimes d_1)^2+(\overline{D^*}:d_2\otimes d_1)^2+(\overline{D^*}:d^*\otimes d_1)^2,
\end{equation}
\begin{equation}\label{d-3}
|\overline{D^*}\cdot d_2|^2 =(\overline{D^*}:d_1\otimes d_2)^2+(\overline{D^*}:d_2\otimes d_2)^2+(\overline{D^*}:d^*\otimes d_2)^2,
\end{equation}
\begin{equation}\label{d-4}
|\overline{D^*}\cdot d^*|^2 =(\overline{D^*}:d_1\otimes d^*)^2+(\overline{D^*}:d_2\otimes d^*)^2+(\overline{D^*}:d^*\otimes d^*)^2,
\end{equation}
\begin{equation}\label{d-5}
|\overline{D^*}|^2=|\overline{D^*}\cdot d_1|^2 +|\overline{D^*}\cdot d_2|^2+|\overline{D^*}\cdot d^*|^2.
\end{equation}
For any $\phi\in C_0^\infty([0,T)\times \mathbb{R}^2, \mathbb{R})$, one can take $\varphi=e_3\phi$ into the equality (\ref{Q*new}) due to $|d_1|=|d_2|=1$. Then direct calculations yield that
\begin{eqnarray*}
   0&=&\int_{\mathcal{L}_\infty\cap \mathcal{A}_\infty}\int_{\mathbb{R}^2\setminus \mathcal{B}(t)}\left[\frac{1}{\Gamma}(L_1 \Delta Q^* -\mathcal{J}^*) +S_{Q^*}(\overline{D^*})\right]: e_3\phi dxdt \\
   &=& \int_{\mathcal{L}_\infty\cap \mathcal{A}_\infty}\int_{\mathbb{R}^2\setminus \mathcal{B}(t)}\left[\sqrt{8}\frac{L_1}{\Gamma}s_+(\partial_k d^*\cdot d_1)(\partial_k d^*\cdot d_2)-\frac{a_3}{\Gamma} +\sqrt{8}\xi\frac{1-s_+}{3}\overline{D^*}:d_2\otimes d_1\right]\phi .
\end{eqnarray*}
This means
$$a_3= 2\sqrt{2}L_1s_+(\partial_k d^*\cdot d_1)(\partial_k d^*\cdot d_2)-2\sqrt{2}\Gamma\xi\frac{s_+-1}{3}\overline{D^*}:d_2\otimes d_1 $$
in $(\mathcal{L}_\infty\cap \mathcal{A}_\infty)\times(\mathbb{R}^2\setminus \mathcal{B}(t))$. Similarly, it holds that
$$a_4=\sqrt{2}L_1s_+ \sum_{k=1}^2(|\partial_k d^*\cdot d_1|^2 -|\partial_k d^*\cdot d_2|^2) -\sqrt{2}\Gamma\xi\frac{s_+-1}{3}\overline{D^*}:(d_1\otimes d_1-d_2\otimes d_2)$$
and
\begin{equation}\label{a5}
a_5= \sqrt{6}L_1s_+ |\nabla d^*|^2 +\frac{\sqrt{6}\Gamma \xi(2s_+^2-s_+-1)}{3}\overline{D^*}:d^*\otimes d^*
\end{equation}
in $(\mathcal{L}_\infty\cap \mathcal{A}_\infty)\times(\mathbb{R}^2\setminus \mathcal{B}(t))$, where one has used the equality (\ref{d-1}), $\mathbb{I}=d_1\otimes d_1 + d_2\otimes d_2+d^*\otimes d^*$ and $\overline{D^*}_{ii}=0$ in (\ref{a5}). Set
$$b_1=-2\sqrt{2}\Gamma\xi\frac{s_+-1}{3}\overline{D^*}:d_2\otimes d_1, \quad b_2=-\sqrt{2}\Gamma\xi\frac{s_+-1}{3}\overline{D^*}:(d_1\otimes d_1-d_2\otimes d_2),$$
$$b_3=\frac{\sqrt{6}\Gamma \xi(2s_+^2-s_+-1)}{3}\overline{D^*}:d^*\otimes d^*, \quad c_1=a_3-b_1= 2\sqrt{2}L_1s_+(\partial_k d^*\cdot d_1)(\partial_k d^*\cdot d_2),$$
and
$$c_2= a_4-b_2=\sqrt{2}L_1s_+ \sum_{k=1}^2(|\partial_k d^*\cdot d_1|^2 -|\partial_k d^*\cdot d_2|^2), \quad c_3=a_5-b_3=\sqrt{6}L_1s_+ |\nabla d^*|^2.$$
Then, it follows from the definition of $Q^*,\mathcal{J}^*, a_i$ ($i=3,4,5$), $b_j$ and $c_j$ ($j=1,2,3$) and direct calculations that for $(t,x)\in (\mathcal{L}_\infty\cap \mathcal{A}_\infty)\times(\mathbb{R}^2\setminus \mathcal{B}(t))$, it holds that
\begin{eqnarray*}
  |L_1\Delta Q^* - \mathcal{J}^*|^2 &=& L_1^2 (\Delta Q^*: \Delta Q^*) -2L_1(\Delta Q^*:\mathcal{J}^*) +(\mathcal{J}^*:\mathcal{J}^*) \\
   &=& 2L_1^2 s_+^2 \left(|\Delta d^*|^2 +|\nabla d^*|^4 +2 \sum_{k,l=1}^2 |\partial_k d^*\cdot \partial_l d^*|^2\right)\\
   &&-2(a_3c_1 +a_4c_2 +a_5c_3) +a_1^2 +a_2^2 +a_3^2\\
   &=&2L_1s_+^2 \left(|\Delta d^*|^2 +|\nabla d^*|^4 +2 \sum_{k,l=1}^2 |\partial_k d^* \cdot \partial_l d^*|^2\right)\\
   && -(c_1^2 +c_2^2 +c_3^2) +b_1^2 +b_2^2 +b_3^2\\
   &=&2L_1s_+^2 (|\Delta d^*|^2 -|\nabla d^*|^4)+b_1^2 +b_2^2 +b_3^2\\
   && +4L_1^2 s_+^2 \sum_{k,l=1}^2|\partial_k d^* \cdot \partial_l d^*|^2 -(c_1^2 +c_2^2 + 2L^2_1s_+^2 |\nabla d^*|^4).
\end{eqnarray*}
Using (\ref{d-1}) and the fact that
$$\partial_k d^* \cdot \partial_l d^* =(d_1\cdot \partial _k d^*)(d_1 \cdot \partial_l d^*) +(d_2 \cdot\partial_k d^*)(d_2 \cdot \partial_l d^*),$$
one can calculate that
$$4L_1^2 s_+^2 \sum_{k,l=1}^2|\partial_k d^* \cdot \partial_l d^*|^2 -(c_1^2 +c_2^2 + 2L^2_1s_+^2 |\nabla d^*|^4)=0.$$
Thus, one obtains that on $(\mathcal{L}_\infty\cap \mathcal{A}_\infty)\times(\mathbb{R}^2\setminus \mathcal{B}(t))$,
\begin{eqnarray}
   |L_1\Delta Q^* - \mathcal{J}^*|^2 &=& 2L_1s_+^2 (|\Delta d^*|^2 -|\nabla d^*|^4)+b_1^2 +b_2^2 +b_3^2 \nonumber\\
   &=& 2L_1s_+^2 \left|\Delta d^* +|\nabla d^*|^2d^*\right|^2+b_1^2 +b_2^2 +b_3^2. \label{Z-1}
\end{eqnarray}
It remains to calculate $b_1^2 +b_2^2+b_3^2$. It follows from the definitions of $b_i$ ($i=1,2$) that
\begin{equation}\label{Z-2}
   b_1^2+b_2^2=\frac{2\Gamma^2 \xi^2(1-s_+)^2}{9}\left[4(\overline{D^*}: d_2\otimes d_1)^2+ (\overline{D^*} :d_1\otimes d_1- \overline{D^*}: d_2\otimes d_2)^2 \right].
\end{equation}
Note that direct computations yield that
\begin{eqnarray}
   &&  (\overline{D^*} :d_1\otimes d_1- \overline{D^*}: d_2\otimes d_2)^2 \nonumber\\
   &=& -(\overline{D^*} :d_1\otimes d_1+\overline{D^*}: d_2\otimes d_2)^2 +2(\overline{D^*}:d_1\otimes d_1)^2+2(\overline{D^*}:d_2\otimes d_2)^2. \label{Z-3}
\end{eqnarray}
Since $d_1\otimes d_1+d_2\otimes d_2+d^*\otimes d^*=\mathbb{I}$ and $\overline{D^*}_{ii}=0$, so
\begin{equation}\label{Z-4}
  \overline{D^*}:(d_1\otimes d_1 +d_2\otimes d_2)=-\overline{D^*} :d^*\otimes d^*.
\end{equation}
(\ref{Z-3}) and (\ref{Z-4}) yield
\begin{equation}\label{Z-5}
   (\overline{D^*} :d_1\otimes d_1- \overline{D^*}: d_2\otimes d_2)^2= 2(\overline{D^*}:d_1\otimes d_1)^2+2(\overline{D^*}:d_2\otimes d_2)^2-(\overline{D^*} :d^*\otimes d^*)^2.
\end{equation}
On the other hand, (\ref{d-2}) and (\ref{d-3}) imply that
$$ (\overline{D^*}:d_1\otimes d_1)^2+(\overline{D^*}:d_2\otimes d_1)^2=|\overline{D^*}\cdot d_1|^2-(\overline{D^*}:d^*\otimes d_1)^2,$$
$$ (\overline{D^*}:d_1\otimes d_2)^2+(\overline{D^*}:d_2\otimes d_2)^2=|\overline{D^*}\cdot d_2|^2-(\overline{D^*}:d^*\otimes d_2)^2.$$
These, together (\ref{Z-5}), (\ref{d-4}) and (\ref{d-5}), imply that
\begin{eqnarray*}
   && 4(\overline{D^*}: d_2\otimes d_1)^2+ (\overline{D^*} :d_1\otimes d_1- \overline{D^*}: d_2\otimes d_2)^2 \\
  &=& 2(\overline{D^*} :d_1\otimes d_1)^2 +2(\overline{D^*} :d_2\otimes d_2)^2 +4(\overline{D^*}: d_2\otimes d_1)^2 -(\overline{D^*}:d^*\otimes d^*)^2\\
   &=& 2(|\overline{D^*}|^2 -2|\overline{D^*}\cdot d^*|^2)+(\overline{D^*}:d^*\otimes d^*)^2.
\end{eqnarray*}
This and (\ref{Z-2}) yield that
\begin{equation}\label{key-bb}
  b_1^2+b_2^2= \frac{2\Gamma^2\xi^2(1-s_+)^2}{9}(2|\overline{D^*}|^2-4|\overline{D^*}\cdot d^*|^2 +|\overline{D^*}:d^*\otimes d^*|^2 ).
\end{equation}
Consequently,
\begin{eqnarray}
  b_1^2 +b_2^2 +b_3^2 &=&  \frac{4\Gamma^2\xi^2(1-s_+)^2}{9}|\overline{D^*}|^2-\frac{8\Gamma^2 \xi^2 (1-s_+)^2}{9}|\overline{D^*}\cdot d^*|^2 \nonumber\\
   &&+\frac{(1-s_+)^2 +3(2s_+-s_+-1)^2}{9} 2\Gamma^2 \xi^2(\overline{D^*}: d^*\otimes d^*)^2\nonumber \\
   &=& \frac{4\Gamma^2\xi^2(1-s_+)^2}{9}|\overline{D^*}|^2 +\Gamma (\alpha_5+\alpha _6 -\frac{\gamma_2^2}{\gamma_1})|\overline{D^*}\cdot d^*|^2\nonumber\\
   && +\Gamma (\alpha_1 +\frac{\gamma_2^2}{\gamma_1})(\overline{D^*}: d^*\otimes d^*)^2. \label{Z-7}
\end{eqnarray}

Therefore, for any $t\in \mathcal{L}_\infty\cap \mathcal{A}_\infty$, we have from (\ref{Z-1}) and (\ref{Z-7}) that
\begin{eqnarray*}
 \int_{\mathbb{R}^2\setminus \mathcal{B}(t)}|H^*|^2dx&=& \int_{\mathbb{R}^2\setminus \mathcal{B}(t)} \left[2L_1^2s_+^2|\Delta d^*+|\nabla d^*|^2 d^*|^2 +b_1^2+b_2^2+b_3^2\right]dx \\
   &=& \int_{\mathbb{R}^2\setminus \mathcal{B}(t)}\left[2L_1^2s_+^2 |\Delta d^* +|\nabla d^*|^2d^*|^2 +\Gamma(\alpha_1+\frac{\gamma_2^2}{\gamma_1})|\overline{D^*}:d^*\otimes d^*|^2  \right]dx\\
   && +\int_{\mathbb{R}^2\setminus \mathcal{B}(t)}\left[\frac{4\Gamma^2\xi^2(1-s_+)^2}{9}|\overline{D^*}|^2+\Gamma(\alpha_5+\alpha_6-\frac{\gamma_2^2}{\gamma_1}) |\overline{D^*}\cdot d^*|^2\right]dx.
\end{eqnarray*}
Note $\int_{\mathbb{R}^2\setminus \mathcal{B}(t)}|H^*|^2dx=\int_{\mathbb{R}^2}|H^*|^2dx$. Thus, the desired (\ref{last}) is proved.

\section*{Acknowledgments.} {Z. Xin was supported in part by Zheng Ge Ru Foundation, HongKong RGC Earmarked Research Grants CUHK14305315, CUHK14302819, CUHK14300917, CUHK14302917, and by Guangdong Province Basic and Applied Basic Research Fund 2020B1515310002. X. Zhang was supported in part by the National Natural Science Foundation of China grant 11901209, by the National Postdoctoral Program for Innovative Talents(BX20200135) of China, by China Postdoctoral Science Foundation(2020M682745), by the Natural Science Foundation of Guangdong Province grant 2019A1515011621, by the Guangdong Province Basic and Applied Basic Research Fund 2020B1515310005.}

\end{document}